\documentclass[11pt]{amsart}
\usepackage{amssymb}
\usepackage{amsfonts}
\usepackage{amsmath}
\usepackage{graphicx}
\usepackage{xcolor}
\usepackage{amsmath}
\usepackage{mathrsfs}
\usepackage{stmaryrd}
\usepackage{epsfig,color}
\usepackage{blindtext}
\usepackage{enumerate}
\usepackage{hyperref}
\usepackage{url}
\usepackage{bbm}
\usepackage{nicefrac,mathtools}
\usepackage{bm}   
\DeclareGraphicsExtensions{.pdf,.jpeg,.png}
\usepackage{epstopdf}
\usepackage{cancel} %for editing
\usepackage[normalem]{ulem} %for editing
\usepackage{verbatim} %for comment

\usepackage{ulem}

\usepackage{color}
\usepackage[msc-links, lite]{amsrefs}
\usepackage{geometry}
\newtheorem{theorem}{Theorem}[section]

\newtheorem{proposition}[theorem]{Proposition}
\newtheorem{lemma}[theorem]{Lemma}

\theoremstyle{definition}
\newtheorem{definition}[theorem]{Definition}
\theoremstyle{remark}
\newtheorem{remark}[theorem]{Remark}

\numberwithin{equation}{section}

\newcommand{\dv}{\mathrm{div}}

\newcommand{\mf}{\mathbf}
\newcommand{\mb}{\mathbb}
\newcommand{\mc}{\mathcal}
\newcommand{\ms}{\mathscr}
\newcommand{\mk}{\mathfrak}
\newcommand{\mr}{\mathrm}
\newcommand{\oli}{\overline}

\newcommand{\wti}{\widetilde}

\newcommand{\Vol}{\mathrm{Vol}}
\newcommand{\Area}{\mathrm{Area}}

\newcommand{\dist}{\operatorname{dist}}

\title{Improved $C^{1,1}$ regularity for multiple membranes problem}

\date{}

\author{Zhichao Wang}
\address{Shanghai Center for Mathematical Science, 2005 Songhu Road, Fudan University, Shanghai, 200438, China}
\email{zhichao@fudan.edu.cn}

\author{Xin Zhou}
\address{Department of Mathematics, 531 Malott Hall, Cornell University,	Ithaca, NY 14853, USA}
\email{xinzhou@cornell.edu}

\begin{document}

\begin{abstract}
%We prove the optimal $C^{1,1}$-regularity for $C^{1,\alpha}$ ($\alpha\in(0,1)$) solutions to the multiple membrane problem of the same divergence form. In particular, the isotopy minimizer of prescribed mean curvature surfaces are $C^{1,1}$ provided they have bounded first variation as an immersion. This forms a key step in our recent work on Yau's four minimal spheres conjecture \cite{wang-Zhou-4-spheres}.
%We provide an elementary proof of the $C^{1,1}$-regularity for $C^{1,\alpha}$ ($\alpha\in(0,1)$) solutions to the multiple membrane problem of the same divergence form. This regularity estimate was essentially used in our recent work on Yau's four minimal spheres conjecture \cite{wang-Zhou-4-spheres}.
%We prove the $C^{1,1}$-regularity for stationary $C^{1,\alpha}$ ($\alpha\in(0,1)$) solutions to the multiple membrane problem of the same divergence form. This regularity estimate was essentially used in our recent work on Yau's four minimal spheres conjecture \cite{wang-Zhou-4-spheres}.
We prove the $C^{1,1}$-regularity for stationary $C^{1,\alpha}$ ($\alpha\in(0,1)$) solutions to the multiple membrane problem. This regularity estimate was essentially used in our recent work on Yau's four minimal spheres conjecture \cite{wang-Zhou-4-spheres}.
\end{abstract}

\maketitle

\section{Introduction}
The {\em multiple membranes problem}, firstly studied by Vergara-Caffarelli \cite{VC-71}, is to describe the equilibrium position of multiple membranes subject to forces and fixed boundary conditions, and constrained by the condition that they are not allowed to cross each other.

In \cite{VC-two-cmc}, Vergara-Caffareli proved $C^{1,\alpha}$-regularity for the two membranes that minimize the constant mean curvature functional. The regularity for minimizers of two-membranes problem was improved to be $C^{1,1}$ by Silvestre \cite{Sil-2-membrane} in a more general case; in particular, the optimal $C^{1,1}$-regularity holds true for prescribing mean curvature functional if the prescribing function does not change sign. The observation was that the height difference of the two membranes is a solution to the obstacle problem and then the classical work due to L. Caffarelli \cite{Caffarelli-obstacle-98} can be applied. We also refer to \citelist{\cite{Caffaralli-De-Silva-Savin-different-operator} \cite{Caffarelli-Duque-Vivas-fully-nonlinear}} for the two membranes problem for operators of different forms or fully nonlinear operators.

%However, to our knowledge (e.g. \citelist{\cite{Caffaralli-De-Silva-Savin-different-operator} \cite{Caffarelli-Duque-Vivas-fully-nonlinear}}), even for the two membranes problem for the prescribing mean curvature functional, the $C^{1,1}$ regularity is unknown at the points where the prescribing function vanishes.

The problem with more than two membranes is more challenging. For the linear case, Chipot-Vergara-Caffarelli \cite{Chipot-VC-N-mambranes85} proved the existence of the $C^{1,\alpha}$ solution. Particularly, for the Laplacian operator, the regularity can be improved to be $C^{1,1}$ by Savin-Yu \cite{Savin-Yu}. For quasi-linear case, we refer to \cite{Azevedo-Rodrigues-Jose-Santos-p-laplacian} for the $p$-Laplacian operators. In a recent work \cite{SS23}*{\S 11}, Sarnataro-Stryker outlined a proof for $C^{1,1}$-regularity for minimizers of multiple membranes problem for prescribing mean curvature functionals using \cites{Sil-2-membrane, Savin-Yu}; see Remark \ref{rmk:after main thm}\eqref{item:rmk:SS} for more discussions.

\begin{comment}
    In this paper, we provide a %direct 
proof that a stationary $C^{1,\alpha}$ solution to the multiple membranes problem, with the same divergence form, can always be improved to be $C^{1,1}$ using only basic elliptic estimates. %see Remark \ref{rmk:after main thm} for more discussions.
\end{comment}

In this paper, we prove the optimal $C^{1,1}$-regularity for all stationary $C^{1,\alpha}$ solutions to the multiple membranes problem, satisfying a natural bounded first variation assumption; see Remark \ref{remark:two assumptions}. The proof relies only on classical elliptic estimates.

\subsection{Notations and main results}
We consider a general second order elliptic operator of divergence form. For simplicity, we write 
\[ 
    \bm U_r:=B_r(0)\times \mb R\times B_1(0).\] 
Let $\lambda\in(0,1),\Lambda>1,\alpha\in(0,1)$ be constants and 
$\ms A(x,z,\bm p)\in C^3(\mb R^n\times \mb R\times \mb R^n,\mb R^n), 
\ms B(x,z,\bm p)\in C^2(\mb R^n\times \mb R\times \mb R^n)$ 
satisfying
\begin{gather}
    \Lambda|\bm \xi |^2
        \geq  (\partial_{\bm p}\ms A)|_{\bm U_2}(\bm \xi,\bm \xi)
        \geq \lambda |\bm \xi|^2, \ \forall \bm \xi\in \mb R^n;
        \label{eq:Lambda and lambda}\\
    \| \ms A\|_{C^3(\bm U_2)}+\|\ms B\|_{C^2(\bm U_2)}\leq \Lambda;
        \label{eq:ms AB norm}\\
    \ms A(x,0,0)=0,\quad \ms B(x,0,0)=0.
        \label{eq:ms AB at x00}
\end{gather}
Define an operator $\mc Q$ by 
\[  
    \mc Qu:=-\dv\, \ms A(x,u,Du)+\ms B(x,u,Du).\]
Let $f^1,\cdots,f^n$ and $g$ be measurable functions. Denote by 
$\mf f=(f^1,\cdots,f^n)$.
We say that
$\mc Qu=-\dv\, \mf f+ g$ in the {\em weak sense} in $U\subset \mb R^n$ if for all $\phi\in C^1_c(U)$,
\[
    \int_{U}\ms A(x,u,Du)D\phi+\ms B(x,u,Du)\phi-\mf f\cdot D\phi-g\phi\,\mr dx=0.\]
%\textcolor{red}{(Why do we need this line?)}

We call a collection of $C^1$ functions $\{u_j\}_{j=1}^m$ over a smooth domain $U\subset\mb R^n$ {\em ordered graphs} if 
\[  u_1\leq u_2\leq \cdots \leq u_m.\]
Let $\{g_j\}_{j=1}^m$ be a family of measurable function. We say a collection of ordered graphs $\{u_j\}_{j=1}^m$ is a {\em weak solution} to the equation 
\begin{equation}\label{eq:main pde}
    \sum_{j=1}^m(\mc Qu_j+g_j)=0
\end{equation}  
in $U$, if for all $\phi\in C_c^1(U)$,
\[
    \sum_{j=1}^m\int_{U}\ms A(x,u_j,Du_j)D\phi+\ms B(x,u_j,Du_j)\phi+g_j\phi\,\mr dx=0.\]
    
%Let $\{g_j\}$ be a family of measurable function. We consider the {\em weak solution} $\{u_j\}$ to the equation 
%\begin{equation}\label{eq:main pde}
%    \sum_{j=1}^m(\mc Qu_j+g_j)=0
%\end{equation}  
%in $U$; that is, for all $\phi\in C_c^1(U)$,
%\[
%    \sum_{j=1}^m\int_{U}\ms A(x,u_j,Du_j)D\phi+\ms B(x,u_j,Du_j)\phi+g_j\phi\,\mr dx=0.\]

%Let $\{u_j\}_{j=1}^m$ be a collection of ordered continuous functions over a smooth domain $U\subset\mb R^n$ with
%\[  u_1\leq u_2\leq \cdots \leq u_m.\]
%We say that $\{u_j\}$ are {\em ordered graphs}.
\begin{comment}
We further assume that $\{u_j\}$ satisfy the {\em subsystem condition}, i.e., for any open set $U'\subset U$ and $1\leq m_1\leq m_2\leq m$ so that
\[
    u_{m_1-1}<u_{m_1}\leq \cdots\leq u_{m_2}<u_{m_2+1}, \quad 
    (u_{0}:=-\infty, \ u_{m+1}:=+\infty), \quad \forall x\in U',
    \]
then 
\[   
    \sum_{j=m_1}^{m_2}(\mc Qu_j+g_j)=0\]
in $U'$ in the weak sense. We will also require that $\{u_j\}$ satisfy the {\em $\kappa$-condition}, i.e., for each $j=1,\cdots,m$,
\[
    -\kappa\leq \mc Qu_j\leq \kappa
\]
in $U$ in the weak sense; that is, for all $\phi\in C_c^1(U)$ with $\phi\geq 0$, 
\[ 
    -\kappa\int_{U}\phi\,\mr dx 
        \leq  \int_{U}\ms A(x,u_j,Du_j)D\phi+\ms B(x,u_j,Du_j)\phi\,\mr dx
        \leq \kappa\int_{U}\phi\,\mr dx.\]
\end{comment}

We assume that $\{u_j\}$ satisfy two additional natural conditions: 
\begin{enumerate}[(I)]
    \item\label{item:subsystem} the {\em subsystem condition}: for any open set $U'\subset U$ and $1\leq m_1\leq m_2\leq m$ so that
\[
    u_{m_1-1}<u_{m_1}\leq \cdots\leq u_{m_2}<u_{m_2+1}, \quad 
    (u_{0}:=-\infty, \ u_{m+1}:=+\infty), \quad \forall x\in U',
    \]
we have 
\[   
    \sum_{j=m_1}^{m_2}(\mc Qu_j+g_j)=0\]
in $U'$ in the weak sense;

    \item\label{item:kappa condition} the {\em $\kappa$-condition}: for each $j=1,\cdots,m$,
\[
    -\kappa\leq \mc Qu_j\leq \kappa
\]
in $U$ in the weak sense; that is, for all $\phi\in C_c^1(U)$ with $\phi\geq 0$, 
\[ 
    -\kappa\int_{U}\phi\,\mr dx 
        \leq  \int_{U}\ms A(x,u_j,Du_j)D\phi+\ms B(x,u_j,Du_j)\phi\,\mr dx
        \leq \kappa\int_{U}\phi\,\mr dx.\]
\end{enumerate}
\begin{remark}\label{remark:two assumptions}
In Section \ref{SS:background from isotopy minimizing problem} below, we will introduce the connection between this PDE problem and a geometric isotopy energy minimizing problem. The subsystem condition \eqref{item:subsystem} is a direct consequence when the system is only a stationary point for the prescribing mean curvature energy, while the $\kappa$-condition \eqref{item:kappa condition} is implied by assuming that each graph of $\{u_j\}$ has bounded $\kappa$-first variation. Both conditions are satisfied by energy minimizers, but weak solutions satisfying \eqref{item:subsystem} and \eqref{item:kappa condition} form a much more general class than just minimizers. For instance, both conditions are preserved by $C^{1, \alpha}$ convergence, but the $C^{1, \alpha}$-limits are not necessarily isotopy energy minimizers. In particular, the $C^{1,1}$-regularity of these $C^{1, \alpha}$-limits play a crucial role in our recent work on Yau's four minimal spheres conjecture \cite{wang-Zhou-4-spheres}.
\end{remark}

%In this paper, we prove the optimal $C^{1,1}$ regularity for the ordered $C^{1,\alpha}$ solution to the equation \eqref{eq:main pde} provided they satisfy the subsystem condition and $\kappa$-condition.

We now state our main $C^{1,1}$-regularity result.
\begin{theorem}\label{thm:main thm}
Let $\mc Q$ be an operator satisfying \eqref{eq:Lambda and lambda}, \eqref{eq:ms AB norm} and \eqref{eq:ms AB at x00}.
Let $\{g_j\}_{j=1}^m$ be a collection of $C^1$ functions over $B_1(0)$.
Let $\{u_j\}_{j=1}^m$ be ordered $C^{1,\alpha}$ ($\alpha\in (0,1)$) graphs over $B_1(0)\subset \mb R^n$ which form a weak solution to \eqref{eq:main pde} in $B_1(0)$ and satisfy the subsystem condition \eqref{item:subsystem} and $\kappa$-condition \eqref{item:kappa condition} ($\kappa>0$).
%Suppose in addition that $\{u_j\}$ satisfy the subsystem condition and $\kappa$-condition ($\kappa>0$). 
Then $u_j\in C^{1,1}_{loc}(B_1(0))$. 
Moreover, if 
\begin{equation} \label{eq:assumption of bound 1}
    \kappa\leq 1,\quad  \|g_j\|_{C^1(B_1(0)}\leq 1, \quad \|u_j\|_{C^{1,\alpha}(B_1(0))}\leq 1,
    \end{equation}
then 
\begin{equation}\label{eq:C11 bound}
    \sum_{j=1}^m\|u_j\|_{C^{1,1}(B_{1/2}(0))}
    \leq C\sum_{j=1}^m\Big(
        \|u_j\|_{C^{1,\alpha}(B_1(0))}
        +\|g_j\|_{C^1(B_1(0))}+\kappa
    \Big),
\end{equation}
where $C=C(n,m,\Lambda,\lambda,\alpha)$.
\end{theorem}
\begin{remark}\label{rmk:after main thm}
\begin{enumerate}[(i)]
    
    \item If each term in \eqref{eq:assumption of bound 1} is bounded by $\Lambda'$, then \eqref{eq:C11 bound} still holds and the constant $C$ will also depend on $\Lambda'$. 
    
    \item\label{item:rmk:general two} Our regularity improves Silvestre \cite{Sil-2-membrane} even for 
\[ 
    m=2,\quad \ms A(x,z,\bm p)=\bm p/\sqrt{1+|\bm p|^2},\quad \ms B=0, \]
since we don't require that $g_2>g_1$.  
%We need the $\kappa$-condition and subsystem condition since  $\{u_j\}$ in Theorem \ref{thm:main thm} is only a $C^{1,\alpha}$ solution which may not be a minimizer to the multiple membrane problem. These two conditions are always satisfied for minimizers.

    \item\label{item:rmk:SS} Sarnataro-Stryker \cite{SS23}*{Corollary 11.2} outlined a similar estimate for minimizers by virtue of free boundary problems \cite{Caffarelli-obstacle-98} and methods in \citelist{\cite{Sil-2-membrane}\cite{Savin-Yu}}. Their result applied to arbitrary prescription functions, even though their proof relied on \cite{Sil-2-membrane} which only holds for prescription functions which do not change sign.  In this paper, we give a proof which holds for stationary solutions with arbitrary $\{g_j\}$ while only using standard elliptic estimates in \cite{GT}.
\end{enumerate}
\end{remark}

\subsection{Background from isotopy minimizing problem}
\label{SS:background from isotopy minimizing problem}
Fix a smooth open domain $U\subset \mb R^n$, a smooth curvature prescription function $\hslash\in C^\infty\big(\oli{U}\times[-1, 1]\big)$, and a smooth Riemannian metric $g$ in $\oli{U}\times [-1, 1]$.
We also assume that $U\times\{0\}$ is a minimal hypersurface in $(\oli U\times [-1,1],g)$, so that the second condition of \eqref{eq:ms AB at x00} is satisfied.
%so that $U\times\{0\}$ is a minimal hypersurface in $(\oli U\times [-1,1],g)$\footnote{This is to satisfy the second one of \eqref{eq:ms AB at x00}.}.
Note that all discussion in the paper will be restricted to $U\times (-1,1)$. Let $\zeta(x, z)$ denote the $(n+1)$-dimensional volume element of $g(x, z)$. Then the area element w.r.t. $g$ is a smooth function:
\[ 
    F: U\times (-1, 1) \times \mb R^n \to [0, \infty), \]
defined as follows: for any $(x, z, \bm p)\in U\times(-1, 1)\times\mb R^n$, we use $P_{x, z, \bm p}$ to denote the $n$-dimensional parallelogram in $\mb R^{n+1}$ generated by
\[ 
    \{e_1, \cdots, e_n\}, \text{ where } e_i=\partial_{x^i} + {\bm p}^i \partial_z. \]
Then,
\[ 
    F(x, z, \bm p)= \text{$n$-volume of $P_{x, z, \bm p}$ under $g(x, z)$}.\] 

Given a $C^1$-function $u: U\to\mb R$, we define the generalized area as the following elliptic functional,
\[    
    \Area(\mr{Graph}_u)=\int_U F(x,u,Du)\,\mr dx .
\]
%Where  is a smooth function.
The first variation formula of $\Area(\mr{Graph}_u)$ w.r.t. variations $t\mapsto \mr{Graph}_{u+t \phi}$ 
for a fixed $\phi\in C^{1}_c(\Omega)$ is given by
\begin{equation}\label{eq:general first variation0}
    \delta\Area_u(\phi) 
    = \int_{U} {\partial_{\bm p}}
        F\big(x, u(x), D u(x)\big)\cdot D \phi 
        + {\partial_z} F \big(x, u(x), D u(x)\big) \cdot \phi
    \, \mr d\mc H^n(x).
\end{equation}

Let $\{u_j\}_{j=1}^m$ be a collection of ordered $C^{1,\alpha}$-graph functions over $U\subset\mb R^n$.  %with \[  u_1\leq u_2\leq \cdots \leq u_m.\]
Let $\Sigma$ be the varifold induced by the graphs of $\{u_j\}$ and 
$\Omega\in \mc C\big(U\times(-1, 1)\big)$ 
be a Caccioppoli set so that 
$\partial \Omega= \sum_{j=1}^m (-1)^{j-1}\llbracket\mr{Graph}_{u_j}\rrbracket$ in the sense of currents; see \cite{Si}. 
The $\mc A^h$-functional for the pair $(\Sigma, \Omega)$, $\mc A^h(\Sigma, \Omega) = \Area(\Sigma) - \int_\Omega \hslash\, \mr d \Vol_g$, defined in \cite{wang-Zhou-4-spheres}*{(1.1)} (see also \cites{ZZ18, SS23}), can be written as
\[ 
    \mc A^h(\Sigma, \Omega) 
    = \Area(\mr{Graph}_u) 
    - \int_U\Big(
        \sum_{j=1}^m (-1)^{j-1} \int_{-1}^{u_j(x)} 
        \hslash(x, z)\cdot \zeta(x, z) 
        \,\mr d z
    \Big) \mr d x. \]
For simplicity, we will abuse the notation by writing 
\[ 
    h(x, z) 
    := \hslash(x, z)\cdot \zeta(x, z) 
    \in C^{\infty}\big(
        \oli{U}\times [-1, 1]
        \big).\]
The first variation of $\mc A^h(\Sigma,\Omega)$ w.r.t. variations $t\mapsto \mr{Graph}_{u+t \phi}$ is 
\begin{align*}%\label{eq:1st variation}
    \delta\mc A^h_{\Sigma,\Omega}(\phi)
    =\sum_{j=1}^m\int_{U} {\partial_{\bm p}}
        F\big(x, u_j, D u_j\big)\cdot D \phi 
        + {\partial_z} F\big(x, u_j, D u_j\big)\cdot \phi 
        +(-1)^jh(x,u_j)\phi
    \,\mr dx, 
\end{align*}
for any $\phi\in C_c^1(U)$. 

We say that $(\Sigma, \Omega)$ is {\em $\mc A^h$-stationary} if for any one parameter family of diffeomorphisms $\{\varphi_t\}_{-1\leq t\leq 1}$ of $U\times (-1,1)$ satisfying that $\varphi_t=\mr{id}$ in a neighborhood of $\partial U\times [-1,1]$, and $\varphi_0=\mr {id}$, then 
\[  \frac{d}{dt}\Big|_{t=0}\mc A^h\big(\varphi_t(\Sigma),\varphi_t(\Omega)\big) = 0.\]
Clearly, this implies that for any $\phi\in C_c^1(U)$, 
\begin{equation}\label{eq:Ah stationarity}
    \sum_{j=1}^m\int_{U} 
    {\partial_{\bm p}} F\big(x, u_j, D u_j\big)\cdot D \phi 
    + {\partial_z} F\big(x, u_j, D u_j\big)\cdot \phi 
    +(-1)^jh(x,u_j)\phi
    \,\mr dx 
        = 0.   
\end{equation}  

%{\color{blue}Moreover, if $\Sigma$ is not connected in $U\times [-1,1]$ (i.e., for each $j=1,\cdots, m-1$, there exists $x_j\in U$ such that $u_j(x_j)=u_{j+1}(x_j)$), then each connected component will satisfy \eqref{eq:Ah stationarity}. This is exactly the subsystem condition that we defined. }

Moreover, if a subcollection of graphs $\{u_j\}_{j=m_1}^{m_2}$ do not touch other graphs, we may choose the variation $\{\varphi_t\}$ to be supported away from other graphs, and then $\{u_j\}_{j=m_1}^{m_2}$ will satisfy \eqref{eq:Ah stationarity}, which is exactly the subsystem condition \eqref{item:subsystem}.

The $\kappa$-condition \eqref{item:kappa condition} is natural in the sense that, if $\mr{Graph}_{u_j}$ has $\kappa$-bounded first variation in $U$, then by the first variation formula of the area functional, we have
\[
-\kappa\leq -\mr{div}{\partial_{\bm p}} F\big(x, u_j, D u_j\big)
    + {\partial_z} F\big(x, u_j, D u_j\big)\leq \kappa
\]
in the weak sense in $U$.
%\textcolor{red}{Explain how subsystem condition follows, and explain by a sentence that $\kappa$-condition follows from bounded first variation.}

%Applying Theorem \ref{thm:main thm} to $\mc A^h$-stationary graphs, we obtain the $C^{1,1}$ regularity.
Applying Theorem \ref{thm:main thm} to $\{u_j\}$, we obtain the $C^{1,1}$ regularity as follows.
\begin{theorem}
%Let $\alpha\in(0,1)$ be a constant, and assume $B_1(0)\times\{0\}$ is a minimal hypersurface in $(B_1(0)\times(-1,1),g)$.
Fix a constant $\alpha\in(0,1)$. Assume that
\begin{equation}\label{eq:condition M}
    \text{$B_1(0)\times\{0\}$ is a minimal hypersurface in $(B_1(0)\times(-1,1),g)$.}
\end{equation}
Let $\{u_j\}$ be ordered $\mc A^h$-stationary graphs over $B_1(0)$. Suppose that for all $j=1,\cdots,m$, $\mr{Graph}_{u_j}$ has $\kappa$-bounded first variation. Then $u_j\in C^{1,1}_{loc}(B_1(0))$, $j=1,\cdots,m$. Moreover, if 
\[
    \kappa\leq 1, \quad \|u_j\|_{C^{1,\alpha}(B_1(0))}\leq 1, \quad 
    \|h\|_{C^1(B_1(0)\times [-1,1])}\leq 1,
\]
then
\[
    \sum_{j=1}^m\|u_j\|_{C^{1,1}(B_{1/2}(0))}
    \leq C\sum_{j=1}^m(\|u_j\|_{C^{1,\alpha}(B_1(0))}
    +\|h\|_{C^1(B_1(0)\times [-1,1])}+\kappa),
 \]
 where $C$ depends only on $n$, $m$, $\alpha$ and the metric $g$.
\end{theorem}

\begin{remark}
Consider $\Sigma$ induced by ordered $C^{1,\alpha}$-graphs over $U\subset \mb R^n$ such that each sheet has bounded first variation. Then by choosing a small enough neighborhood centered at any $p\in \Sigma$, and up to a change of coordinates therein, we can always assume \eqref{eq:condition M} is satisfied. Indeed, assuming without loss of generality $p\in U$, by taking sufficiently small $\epsilon$, $B^n_\epsilon(p)\times [-\epsilon,\epsilon]$ admits a minimal foliation $\{\Gamma_t\}$ with $\partial\Gamma_t$ given by $\partial B^n_\epsilon(p)\times\{t\}$; see \cite{Whi87}*{Appendix}. In particular, there exists a minimal slice $\Gamma$ containing $p$. By taking sufficiently small $\epsilon$, $\Gamma$ is sufficiently close to $B^n_\epsilon(p)\subset\mb R^n$. Thus $\Sigma$ can be written as ordered graphs over $\Gamma$ with uniformly bounded $C^{1,\alpha}$-norm. Then the above theorem can be applied by rescaling $\epsilon$ to $1$.
%By taking sufficiently small $\epsilon$, $T_p\Gamma$ is close to $\mb R^n$. Thus $\Sigma$ can be written as ordered graphs over $\Gamma$. Since each sheet of $\Sigma$ has $\kappa$-bounded first variation, then by Allard's regularity, in a small neighborhood of $p$, the graph functions of $\Sigma$ over $\Gamma$ will have uniformly bounded $C^{1,\alpha}$ norm.
\end{remark}

%%%%%%%%%%%%%%%%%%%%%%%%%%%%%%%%%%%%%%%%%%%%%%%%%%%%%%%%
\subsection{Idea of the proof}

Observe that by the $\kappa$-condition, for each $j=1,\cdots,m$, $\mc Qu_j\in L^{\infty}$. Then one can consider the equation of $u_i-u_j$ and apply the weak Harnack inequality in \cite{GT}*{Theorem 8.17 and Theorem 8.18} to give the bound of $u_i-u_j$ around the touching set. As a result, the difference of $u_i$ and $u_j$ is bounded by $C\kappa r^2$, where $r$ is the distance to the touching set. Such a result together with the subsystem condition will imply the upper bound of the first derivative of $u_i-u_j$. These two inequalities will be crucial in the remaining of this paper. 

The next step is to prove that $u_j$ is $W^{2,2}$. To do this, we will use an inductive method to prove that the $L^2$-integrals of the corresponding difference quotients are uniformly bounded. Suppose that there are $m$ sheets. Then for any points with density less than $m$, one can find a ball so that $\{u_j\}$ is not connected. Then by the inductive process, one can bound the local $L^2$-integral of the difference quotients; see \eqref{eq:hat r inequality at low density}. On the other hand, for the points with density $m$, the first derivative upper bound in the previous paragraph will also give such a small ball with desired bound for the local $L^2$-integral of the difference quotients; see \eqref{eq:hat r inequality at touching}. Then by a covering argument, we will finish the induction and obtain the desired uniform $L^2$ bound.

After that, we will consider the equation of $D_k\varphi$, where $\varphi$ is the average of $u_j$. Thanks to the $C^1$ estimates of $u_i-u_j$, we can adapt the H\"older estimates in \cite{Han-Lin-PDE}*{Theorem 3.8} to prove that the $L^2$-integral of $D_k\varphi$ over $B_r(y)$ has an order $n-\epsilon$ for any $\epsilon\in (0,1)$; see Lemma \ref{lem:n-epsilon} and Proposition \ref{lem:n-epsilon order for uj}. Then the argument of H\"older estimates for gradients (see \cite{Han-Lin-PDE}*{Theorem 3.13}) will be adapted to improve the order estimates of 
\[
    \int_{B_r(y)}|D^2\varphi-(D^2\varphi)_{B_r(y)}|^2\,\mr dx.
\] 
Note that the above two estimates are only valid for the radius $r$ so that $\{u_j\}$ is connected in $B_r(y)$ since we have to use the $C^1$ estimates of $u_i-u_j$. Nevertheless, such an order estimate will be applied to give two delicate estimates (see Lemma \ref{lem:bdd average diff} and Lemma \ref{lem:improved order of D^2u}), which yield a uniform upper bound for 
\[   
   \sum_{j=1}^m\frac{1}{|B_r|}\int_{B_r(y)}|D^2u_j|^2\,\mr dx.
\]
This together with the Poincar\'e inequalities implies the growth estimates of $u_j$ over a tilt-plane. Such a bound enables us to apply the $C^{1,\alpha}$ estimates \cite{GT}*{Theorem 8.32} to bound the derivative of $u_j$ over the tilt-plane, i.e. we obtain the bound of $|Du_j-(Du_j)_{B_r(y)}|$ for all $x\in B_r(y)$. This gives the desired upper bound of $|Du_j(x)-Du_j(y)|$ by taking $r=|x-y|$. Then the main theorem is proved.

%%%%%%%%%%%%%%%%%%%%%%%%%%%%%%%%%%%%%%%%%%%%%%%%%%%%%%
\subsection{Outline}
This paper is organized as follows. 
In Section \ref{sec:preliminary}, we introduce some notions and basic inequalities for the coefficients of the equations. Then we will list some variants of elliptic PDEs that will be used in this paper. 
In Section \ref{sec:height difference}, we give the $C^1$ norm of the difference of any two sheets. As a consequence, the $W^{2,2}$ norm of each $u_j$ is also proved.
Section \ref{sec:order of hessian} is devoted to prove the order estimates of the hessians of $u_j$, which imply that $D^2u_j$ is bounded in the sense of average.
Finally, we prove the desired Lipschitz upper bound for $Du_j$ in Section \ref{sec:lipschitz}.

\subsection*{Acknowledgement} Z. W. would like to thank Professor Jingyi Chen and Professor Ailana Fraser for their support and encouragement. X. Z. is supported by NSF grant DMS-1945178, an Alfred P. Sloan Research Fellowship, and a grant from the Simons Foundation (1026523, XZ).

%%%%%%%%%%%%%%%%%%%%%%%%%%%%%%%%%%%%%%%%%%%%%%%%%%%%%
%Preliminary
%%%%%%%%%%%%%%%%%%%%%%%%%%%%%%%%%%%%%%%%%%%%%%%%%%%%%

\section{Preliminary}\label{sec:preliminary}

In this section, we will derive a number of differential equations to be used later and provide some basic estimates for the coefficients.

%%%%%%%%%%%%%%%%%%%%%%%%%%%%%%%%%%%%%%%%%%%%%%%%%%%%%
\subsection{Some estimates of the coefficients}

We always assume that there exist $\alpha\in (0,1)$ and $\delta_0,\kappa_0,\kappa_1\in (0,1]$,
so that for all $i,j=1,\cdots,m$, the ordered graphs $\{u_j\}$ satisfy:
\begin{gather}
    \label{eq:small C1alpha}
        \sup_{x\in U} |Du_j(x)| +\sup_{x,y\in U} \frac{|D u_j(x)-D u_j(y)|}{|x-y|^\alpha}
        \leq \delta_0; \\
    \label{eq:eta kappa0 and kappa1}
         |g_j|\leq \kappa_0 ,\quad |Dg_j|\leq \kappa_1.
\end{gather}
For simplicity, we denote 
\[  \bm\kappa^*:=\kappa^2+\kappa_0^2+\kappa_1^2+\delta_0^2.\]

\begin{definition}\label{def:Q operator}
Let 
$\ms A(x,z,\bm p)\in C^3(\mb R^n\times \mb R\times \mb R^n,\mb R^n), 
\ms B(x,z,\bm p)\in C^2(\mb R^n\times \mb R\times \mb R^n)$ and 
$\mc Qu:=-\dv\, \ms A(x,u,Du)+\ms B(x,u,Du)$. 
For a family of $C^1$ functions $\{g_j\}_{j=1}^m$, we say that $C^{1,\alpha}$ ordered graphs $\{u_j\}_{j=1}^m$ over $U\subset \mb R^n$ form a solution to $(\mc Q,\{g_j\})$ if $\{u_j\}$ satisfy \eqref{eq:main pde} in $U$ in the weak sense, the subsystem condition \eqref{item:subsystem}, and the $\kappa$-condition \eqref{item:kappa condition}.
\end{definition}

The following will be used in various blow-up arguments. 
%as a blow-up argument in this paper.
\begin{lemma}\label{lem:blow up}
Suppose that the ordered collection $\{u_j\}$ is a solution to $(\mc Q,\{g_j\})$ in $U$. Let $y\in \mb R^n$, $r\in (0,1]$ and define $\{v_j\}$ by 
\[     
    v_j(x)= r^{-1}u_j(y+rx).\]
Define a new operator $\wti {\mc Q}$ by 
\[  
    \wti {\mc Q}v=-\dv \wti {\ms A}(x,v,Dv)+\wti{\ms B}(x,v,Dv),\]
where 
\[  
    \wti {\ms A}(x,z,\bm p):=\ms A(y+rx,rz,\bm p),\quad
    \wti {\ms B}(x,z, \bm p):=r\ms B(y+rx,rz,\bm p).
\]
Define the new functions $\{\wti g_j\}$ by
\[ 
    \wti g_j(x)=rg_j(y+rx).\] 
Then $\{v_j\}$ is a solution to $(\wti{\mc Q},\{\wti g_j\})$ in 
$\wti U:=\{r^{-1}(x-y);x\in U\}$.
\end{lemma}

\begin{remark}\label{rmk:blow up AB bound}
We now describe the changes of $\delta_0,\kappa_0,\kappa_1$ and $\kappa$ in the blow-up process.
Suppose that in Lemma \ref{lem:blow up}, if \eqref{eq:small C1alpha} and \eqref{eq:eta kappa0 and kappa1} hold for 
$\delta_0,\kappa_0,\kappa_1$. 
Then we have that 
\begin{gather*}
     \sup_{x\in \wti U} |Dv_j(x)|\leq \delta_0, \quad  
     \sup_{x,x'\in\wti U} \frac{|D v_j(x)-D v_j(x')|}{|x-x'|^\alpha}
        \leq r^\alpha\delta_0; \\
    |\wti g_j|    \leq r\kappa_0 ,\quad 
    |D\wti g_j|   \leq r^2\kappa_1.
\end{gather*}
Moreover, if $\{u_j\}$ satisfy $\kappa$-condition, then $\{v_j\}$ satisfy $r\kappa$-condition. Without loss of generality, we always assume that $0< \kappa\leq 1$.
\end{remark}

We now introduce the concept of connected components. Note that each connected component is a solution to $(\mc Q,\{g_j'\})$ for the associated $\{g_j'\}\subset \{g_j\}$ by the subsystem condition.
\begin{definition}[Connected component for ordered graphs]
Let $\{u_j\}_{j=1}^m$ be ordered graphs over $U$. 
We say they are connected in $U$ if and only if for each $i=1,\cdots ,m-1$, there exists $x_i\in U$ such that $u_i(x_i)=u_{i+1}(x_i)$.   

Any ordered graphs $\{u_j\}$ over $U$ can divided into subcollections 
$\{u_1,\cdots u_{i_1}\}$, $\{u_{i_1+1},\cdots,u_{i_2}\}$, $\cdots$, 
$\{u_{i_k},\cdots,u_{m}\}$ 
so that each collection is connected and the collections are pairwise disjoint; each collection is called a {\em connected component}.
\end{definition}

\begin{lemma}\label{lem:large gap}
Let $\{u_j\}_{j=1}^m$ be ordered graphs over $B_1(0)$. 
Given $t\in(0,1)$, then there exists $r\in [t^m,1]$ so that $\{u_j\}$ has the same number of connected components in $B_r(0)$ and $B_{tr}(0)$.
\end{lemma}
\begin{proof}
Note that the number of connected components of the system $\{u_j\}$ in $B_r(0)$ is decreasing when $r$ increases. 
Let $N(k)$ be the number of connected components of $\{u_j\}$ in $B_{r_k}(0)$ for $r_k:=t^{k}$. Then the desired result follows from the fact that 
\[ 
    1\leq N(1)\leq N(2)\leq\cdots \leq N(k)\leq m.\]
\end{proof}

Let $u, v$ be two $C^1$-functions on $U$. 
When taking the difference of derivatives of $\ms A$ and $\ms B$, we will estimate the remainder in terms of the following notations. %\textcolor{red}{(Maybe use $\wti{\bm A}(x,u, v)$, $\wti{\bm b}(x,u, v)$, $\wti{d}(x,u, v)$?)}
\begin{gather*} 
    \bm A(u,v)  =  \int_0^1\partial_{\bm p} \ms A(x,tu+(1-t)v,tDu+(1-t)Dv)\,\mr dt;\\
    \bm b(u,v)  =  \int_0^1\partial_z \ms A(x,tu+(1-t)v,tDu+(1-t)Dv)\,\mr dt;\\
    \bm c(u,v)  =  \int_0^1\partial_{\bm p} \ms B(x,tu+(1-t)v,tDu+(1-t)Dv)\,\mr dt;\\
    d(u,v)      =  \int_0^1\partial_z\ms B(x,tu+(1-t)v,tDu+(1-t)Dv)\,\mr dt.
\end{gather*}
Note that by definition, we have
\begin{gather*}
    \bm A(u,u)  =  \partial_{\bm p}\ms A(x,u,Du), \quad \bm b(u,u)= \partial _z\ms A\ms b(x,u,Du);\\
    \bm c(u,u)  =  \partial_{\bm p}\ms B(x,u,Du),\quad d(u,u)= \partial_z\ms B(x,u,Du).
\end{gather*}

\begin{lemma}\label{lem:taking difference}
The following formulas follow from a direct computation. %Let $u_t=tu+(1-t)v$.
\begin{gather*}
 \ms A\big(x, u, D u\big)  -\ms A\big(x, v, D v\big)= \bm A(u,v)\cdot D (u-v)+ \bm b(u,v)(u-v) ; \\
 \ms B\big(x, u, D u\big)  - \ms B\big(x, v, D v\big)= \bm c(u,v)\cdot D (u-v) + d(u,v) (u-v).
\end{gather*}

\end{lemma}

The following estimates will be used later.
%in this paper.
\begin{lemma}\label{lem:Abcd diff bound}
Let $u_1,u_2,v_1,v_2$ be $C^1$ functions defined on $U$. Then 
\begin{equation}\label{eq:Abcd diff bound} 
    \begin{split}
        |\bm A(u_1,u_2)-\bm A(v_1,v_2)|+|\bm b(u_1,u_2)-\bm b(v_1,v_2)|&\leq \|\ms A\|_{C^2}(\|u_1-v_1\|_{C^1}+\|u_2-v_2\|_{C^1});\\
        |\bm c(u_1,u_2)-\bm c(v_1,v_2)|+|d(u_1,u_2)-d(v_1,v_2)|&\leq \|\ms B\|_{C^2}(\|u_1-v_1\|_{C^1}+\|u_2-v_2\|_{C^1}).
    \end{split}
\end{equation}
\end{lemma}

%%%%%%%%%%%%%%%%%%%%%%%%%%%%%%%%%%%%%%%%%%%%%%%%%%%%
\subsection{The variants of differential equations}
Since a solution to $(\mc Q,\{g_j\})$ is in the form of $\{u_j\}$, we can not apply the theory for elliptic PDEs. In this section, we will derive some PDEs satisfied by a single function which can be the average, the derivative, or the difference quotient, etc.

\medskip
{\noindent \bf The equation of the average.} Suppose that $\{u_j\}$ is a solution to $(\mc Q,\{g_j\})$; see Definition \ref{def:Q operator}.
Let 
\[   \varphi:=\frac{1}{m}\sum_{j=1}^mu_j. \]%Suppose that $(\Sigma,\Omega)$ is $\mc A^h$-stationary. 
By Lemma \ref{lem:taking difference}, 
\begin{align*}
    \mc Qu_j-\mc Q\varphi
        &=-\dv\Big(\bm A(u_j,\varphi)D(u_j-\varphi)  +  \bm b(u_j,\varphi)(u_j-\varphi)\Big)+\\
        &\hspace{8em}+\bm c(u_j,\varphi) D(u_j-\varphi)  + d(u_j,\varphi)(u_j-\varphi).
\end{align*}
Observe that 
\[ m\mc Q\varphi=\sum_{j=1}^m(\mc Q\varphi-\mc Qu_j-g_j). \]
Then $\varphi$ satisfies the following equation:
\begin{equation}\label{eq:sum form}
    m\mc Q\varphi  =  \dv \bm f(\{u_j\})-g(\{u_j\},\{g_j\}).
	 \end{equation}
Here 
\begin{gather}
    \label{eq:def of bm f}
        \bm f(\{u_j\})
        :=\sum_{j=1}^m\Big[(\bm A(u_j,\varphi)-\bm A(\varphi,\varphi))(Du_j-D\varphi)  +  \big(\bm b(u_j,\varphi)-\bm b (\varphi,\varphi)\big)(u_j-\varphi) \Big];\\
    \label{eq:def of g}
        g(\{u_j\},\{g_j\})
        :=\sum_{j=1}^m\Big[\big(\bm c(u_j,\varphi)-\bm c (\varphi,\varphi)\big) (Du_j-D\varphi)  +  (d(u_j,\varphi)-d(\varphi,\varphi)(u_j-\varphi)-g_j)\Big].
\end{gather}
Note that in \eqref{eq:def of bm f} and \eqref{eq:def of g}, the additional terms containing 
$\bm A(\varphi,\varphi)$, $\bm b (\varphi,\varphi)$,  $\bm c (\varphi,\varphi)$ and $d(\varphi,\varphi)$ sum to be zero, since 
$\sum_{j=1}^m(\varphi-u_j) = 0$.
\begin{lemma}\label{lem:bound bm f and g}
Let $\{u_j\}$ and $\varphi$ be as above. 
Let $\tau>0$ be a constant so that $\|u_i-u_j\|_{C^1(B_1(0))}\leq \tau$. Then 
\begin{gather*}
    |\bm f(\{u_j\})|\leq C(m,\Lambda)\tau^2   ,\quad 
    \sup_{x,y\in B_1(0)}\frac{|\bm f(\{u_j\})(x)-\bm f(\{u_j\})(y)|}{|x-y|^\alpha}  
        \leq C(m,\Lambda)\tau;\\
    |g(\{u_j\},\{g_j\})|\leq C(m,\Lambda)(\tau^2+\kappa_0).
 \end{gather*}
\end{lemma}
\begin{proof}
It follows from Lemma \ref{lem:Abcd diff bound}, equation \eqref{eq:small C1alpha} and \eqref{eq:eta kappa0 and kappa1}.
\end{proof}

\medskip
{\noindent\bf The equation of differences of two solutions.}
Let $\{u_j\}_{j=1}^{m}$ and $\{v_i\}_{i=1}^{m'}$ be two ordered collections of graphs over $U$. 
Suppose that $\{u_j\}_{j=1}^{m}$ and $\{v_i\}_{i=1}^{m'}$ are solutions to $(\mc Q,\{g_j\})$ and $(\mc Q,\{g_i'\})$, respectively. 
Denote by 
\[
    \varphi:=\frac{1}{m}\sum_{j=1}^mu_j,\quad 
    \psi:=\frac{1}{m'}\sum_{i=1}^{m'}v_i.\]
Then both $\varphi$ and $\psi$ satisfy the equation \eqref{eq:sum form}. We can write them as 
\begin{gather*}
	 m\mc Q\varphi=\dv \bm f(\{u_j\})-g(\{u_j\},\{g_j\});\\
   m'\mc Q\psi=\dv\bm f(\{v_i\})-g(\{v_i\},\{g_j'\}).
\end{gather*}
Taking the difference and using Lemma \ref{lem:taking difference}, we obtain the equation for $\Psi:=\varphi-\psi$ as follows: 
\begin{equation}\label{eq:diff of two average}
    \begin{split}
   & -\dv \Big(\bm A(\varphi,\psi)D\Psi+\bm b(\varphi,\psi)\Psi\Big)+\bm c(\varphi,\psi) D\Psi+d(\varphi,\psi)\Psi \\
    &\hspace{6em}=\dv\Big(\frac{1}{m}\bm f(\{u_j\})-\frac{1}{m'}\bm f(\{v_i\})\Big) - \Big(\frac{1}{m}g(\{u_j\},\{g_j\})-\frac{1}{m'}g(\{v_i\},\{g_j'\})\Big).
    \end{split}
\end{equation}

\medskip
{\noindent\bf The equation of difference quotients.}
Given a function $f$ and a real number $\tau\neq 0$, let 
\[
    f^\tau(x):=\frac{f(x+\tau \bm e)-f(x)}{\tau},\]
where $\bm e$ can be any unit vector in $\mb R^n$. For simplicity, denote by
\begin{gather*}
    \bm A_\tau(u)  
    = \int_0^1\partial_{\bm p} \ms A\big(
        x+t\tau \bm e,u(x)+t\tau \cdot u^\tau(x),Du(x)+t\tau \cdot Du^\tau(x)
    \big)
    \,\mr dt;\\
    \bm b_\tau(u)  
    = \int_0^1\partial_z \ms A\big(
        x+t\tau \bm e,u(x)+t\tau \cdot u^\tau(x),Du(x)+t\tau \cdot Du^\tau(x)
    \big)
    \,\mr dt;\\
    \bm c_\tau(u)  
    = \int_0^1\partial_{\bm p} \ms B\big(
        x+t\tau \bm e,u(x)+t\tau \cdot u^\tau(x),Du(x)+t\tau \cdot Du^\tau(x)
    \big)
    \,\mr dt;\\
    d_\tau(u) 
    = \int_0^1\partial_z\ms B\big(
        x+t\tau \bm e,u(x)+t\tau \cdot u^\tau(x),Du(x)+t\tau \cdot Du^\tau(x)
    \big)
    \,\mr dt;\\
    \mk A_\tau(u)
    = \int_0^1\partial_{x_k} \ms A\big(
        x+t\tau \bm e,u(x)+t\tau \cdot u^\tau(x),Du(x)+t\tau \cdot Du^\tau(x)
    \big)
    \,\mr dt;\\
    \mk B_\tau(u)
    = \int_0^1\partial_{x_k} \ms B\big(
        x+t\tau \bm e,u(x)+t\tau \cdot u^\tau(x),Du(x)+t\tau \cdot Du^\tau(x)
    \big)
    \,\mr dt.
\end{gather*}
where we assume that $\bm e$ is the direction of the $k$-th coordinate.
Then we have that 
\begin{gather*}
    \Big[
        \ms A\big(x, u, D u\big)
    \Big]^\tau 
    = \bm A_\tau(u)Du^\tau 
    +\bm b_\tau(u)u^\tau+\mk A_\tau(u),\\
    \Big[
        \ms B\big(x, u, D u\big)
    \Big]^\tau 
    = \bm c_\tau(u)Du^\tau 
    + d_\tau(u)u^\tau+\mk B_\tau(u).
\end{gather*}

The following estimates will be used later.
\begin{lemma}\label{lem:Xtau difference of u v}
Let $u,v$ be two $C^1$ functions. Then we have that for $x\in U$ with $B_\tau(x)\subset U$,
\begin{gather*}
    |\bm A_\tau(u)-\bm A_\tau(v)|\leq \|\partial_{\bm p}\ms A\|_{C^1}\cdot \|u-v\|_{C^1};\quad  
    |\bm b_\tau(u)-\bm b_\tau(v)|\leq \|\partial_{z}\ms A\|_{C^1}\cdot \|u-v\|_{C^1};\\
    |\bm c_\tau(u)-\bm c_\tau(v)|\leq \|\partial_{\bm p}\ms B\|_{C^1}\cdot \|u-v\|_{C^1};\quad  
    |d_\tau(u)-d_\tau(v)|\leq \|\partial_{z}\ms B\|_{C^1}\cdot \|u-v\|_{C^1};\\
    |\mk A_\tau(u)-\mk A_\tau(v)|\leq \|\partial_{x}\ms A\|_{C^1}\cdot \|u-v\|_{C^1};\quad  
    |\mk B_\tau(u)-\mk B_\tau(v)|\leq \|\partial_{x}\ms B\|_{C^1}\cdot \|u-v\|_{C^1}.
\end{gather*}
Here we omit the domain $U\times \mb R\times B_1(0)$ of the $C^1$ norms for simplification.
\end{lemma}

Let $\{u_j\}$ be ordered graphs over $U\subset \mb R^n$ and a solution to $(\mc Q,\{g_j\})$. Then for all $\phi\in C_c^1(B_1(0))$, we have
\[  
    \sum_{j=1}^m\int_{U} \Big[
        \ms A\big(x, u_j, D u_j\big)
    \Big]^\tau\cdot D \phi 
    + \Big[
        \ms B\big(x, u_j, D u_j\big)
    \Big]^\tau\cdot \phi
    +g_j^\tau\phi\,\mr dx
    =0. \]	
This can be rewritten as 
\[
    \sum_j \Big(
        -\dv\big(
            {\bm A}_\tau(u_j)Du_j^\tau+\bm b_\tau(u_j)u_j^\tau+ \mk A_\tau(u_j)
            \big)
        +\bm c_\tau(u_j)Du_j^\tau + d_\tau(u_j) u_j^\tau 
        + \mk B_\tau(u_j)+ g_j^\tau
    \Big)=0,  
\] 
Then we obtain the equation of $\varphi^\tau$ as 
\begin{align}
\label{eq:differential quotient sum}
\sum_{j=1}^m\Big[-\dv\Big({\bm A}_\tau(u_j)D\varphi^\tau+{\bm b}_\tau(u_j)\varphi^\tau\Big)+\bm c_\tau(u_j)D\varphi^\tau +d_\tau(u_j) \varphi^\tau\Big] =-\dv\bm f_\tau+g_\tau,  
	 \end{align}
 where 
 \begin{gather}
%&\bm A_\tau^0:=\sum_{j=1}^m{\bm A_\tau}(x,u_j,Du_j);\ \bm b_\tau^0:=\sum_{j=1}^m{\bm b_\tau}(x,u_j,Du_j); \ d_\tau^0=\sum_{j=1}^md_\tau(x,u_j,Du_j) ;\label{eq:A tau 0}\\
    \bm f_\tau
    :=\sum_{j=1}^m\Big[
        (\bm A_\tau(u_j)-\bm A_\tau(\varphi))(D\varphi^\tau-Du_j^\tau)
        +\big(\bm b_\tau(u_j)-\bm b_\tau(\varphi)\big)(\varphi^\tau-u_j^\tau)
        -\mk A_\tau(u_j)
    \Big];
    \label{eq:f tau}\\
   g_\tau
   :=\sum_{j=1}^m\Big[
       \big(\bm c_\tau(u_j)-\bm c_\tau(\varphi)\big) (D\varphi^\tau-Du_j^\tau)
       +\big(d_\tau(u_j)-d_\tau(\varphi)\big)(\varphi^\tau-u_j^\tau) 
       -\mk B_\tau(u_j) -	g_j^\tau
    \Big].
    \label{eq:g tau}
%   \mk A_\tau'(u_j):=\mk A_\tau(u_j)-\int_0^1\partial_{x_k}\ms A \big(0+t\tau \bm e,u_j(0)+t\tau\cdot u_j^\tau(0),Du_j(0)+t\tau\cdot Du_j^\tau(0)\big)\,\mr dt.\nonumber
 	\end{gather}
%Here we used the fact that \textcolor{red}{(why not writing it as $\mk A_\tau(u_j)(0)$?)}
%\[    
 %   \int_0^1\partial_{x_k}\ms A \big(0+t\tau \bm e,u_j(0)+t\tau\cdot u_j^\tau(0),Du_j(0)+t\tau\cdot Du_j^\tau(0)\big)\,\mr dt
 %   \]
%is a constant.
Recall that by \eqref{eq:ms AB at x00},
\[
    \int_0^1\partial_{x_k}\ms A(x+t\tau\bm e,0,0)\,\mr dt=0.
\]
By the mean value theorem, 
\begin{equation}\label{eq:bound mk A tau j}
\begin{split}
    &\hspace{1.5em}|\mk A_\tau(u_j)|\\
    &\leq \int_0^1 \Big| \partial_{x_k} \ms A\big(
        x+t\tau \bm e,u_j(x)+t\tau \cdot u_j^\tau(x),Du_j(x)+t\tau \cdot Du_j^\tau(x)
    \big)-\partial_{x_k}\ms A(x+t\tau\bm e,0,0) \Big|
    \,\mr dt\\
    &\leq \|\partial_z\partial_x\ms A\|_{C^0(B_1(0))}\Big( |u_j(x)|+|u_j(x+\tau\bm e)|
    \Big) +\|\partial_{\bm p}\partial_x\ms A\|_{C^0(B_1(0))}\Big(
    |Du_j(x)|+|Du_j(x+\tau\bm e)|
    \Big).
\end{split}
\end{equation}
Since $\ms B(x,0,0)=0$, we have $\mk B_\tau(0)=0$. Then Lemma \ref{lem:Xtau difference of u v} gives that 
\begin{equation}\label{eq:mk B tau bound}
    |\mk B_\tau(u_j)|\leq \|\partial_x\ms B\|_{C^1}\|u_j\|_{C^1}.
\end{equation}
Moreover, we obtain the bound of $\bm f_\tau$ and $g_\tau$ by Lemma \ref{lem:Xtau difference of u v}.
\begin{lemma}\label{lem:bound ftau and gtau}
With the notions as above, there exists $C=C(n,m,\Lambda)$ so that for all $x\in B_{1-\tau}(0)$,
\begin{gather*}
\begin{split}
    |\bm f_\tau|    
    &\leq C\sum_{j=1}^m\Big(\|\partial _{\bm p}\ms A\|_{C^1}\|u_j-\varphi\|_{C^1}\cdot |D\varphi^\tau-Du_j^\tau|+\|\partial_z\ms A\|\cdot \|u_j-\varphi\|_{C^1}\cdot |\varphi^\tau-u_j^\tau|+|\mk A_\tau(u_j)|\Big)\\
    &\leq C\sum_{j=1}^m\Big(\|u_j-\varphi\|_{C^1}\cdot |D\varphi^\tau-Du_j^\tau|+\|u_j-\varphi\|_{C^1}\cdot |\varphi^\tau-u_j^\tau|+ \|\partial_z\partial_x\ms A\|_{C^0(B_1(0))}|u_j|_{C^0}+\\
    &\hspace{25em}+\|\partial_{\bm p}\partial_x\ms A\|_{C^0(B_1(0))}\cdot|Du_j|_{C^0}\Big);
\end{split}\\
\begin{split}
    |g_\tau|
    \leq C\sum_{j=1}^m\Big(\|\partial _{\bm p}\ms B\|_{C^1}\|u_j-\varphi\|_{C^1}\cdot |D\varphi^\tau-Du_j^\tau|+\|\partial_z\ms B\|_{C^1}\cdot \|u_j-\varphi\|_{C^1}\cdot |\varphi^\tau-u_j^\tau|+\\
    \hspace{28em}+|\mk B_\tau(u_j)| +|g_j^\tau|\Big)\\
    \leq C\sum_{j=1}^m\Big(\|u_j-\varphi\|_{C^1}\cdot |D\varphi^\tau-Du_j^\tau|+\|u_j-\varphi\|_{C^1}\cdot |\varphi^\tau-u_j^\tau|+\|\partial_x\ms B\|_{C^1}\|u_j\|_{C^1}+\\
    +|Dg_j|_{C^0}\Big).
\end{split}
\end{gather*}
\end{lemma}

{\noindent\bf The equation of the first derivatives.}
Let $\{u_j\}$ be ordered graphs over $U\subset \mb R^n$ and a solution to $(\mc Q,\{g_j\})$.
If $\{u_j\}\subset  W^{2,2}(U)$, by letting $\tau\to 0$ in \eqref{eq:differential quotient sum} we have, 
\begin{equation}\label{eq:pde of Dk varphi}
      -\dv\Big(\bm A^*DD_k\varphi+{\bm b^*}D_k\varphi\Big)+\bm c^*DD_k\varphi +d^* D_k\varphi=-\dv\Big(\hat{\bm f}-m\partial_{x_k}\ms A(x,\varphi,D\varphi)\Big)+\hat g,
\end{equation}
where
\begin{gather*}
      \bm A^*(x):=\sum_{j=1}^m\partial_{\bm p} \ms A(x,u_j(x),Du_j(x)); 
               \quad \bm b^*(x):=\sum_{j=1}^m\partial_{z} \ms A(x,u_j(x),Du_j(x));\\
      \bm c^*(x):=\sum_{j=1}^m\partial_{\bm p} \ms B(x,u_j(x),Du_j(x)); 
               \quad  d^*(x):=\sum_{j=1}^m\partial_{z} \ms B(x,u_j(x),Du_j(x));
\end{gather*}
and 
\begin{align}
\begin{split}
    \hat{\bm f}
    &:=\sum_{j=1}^m\Big[
        (\bm A(u_j,u_j)-\bm A(\varphi,\varphi))(DD_k\varphi-DD_ku_j)+\\
        &\hspace{1.5em}+\big(\bm b(u_j,u_j)-\bm b(\varphi,\varphi)\big)(D_k\varphi-D_ku_j)-(\partial_{x_k}\ms A(x,u_j,Du_j)- \partial_{x_k}\ms A(x,\varphi,D\varphi))
    \Big];
    \label{eq:fd}
\end{split}\\
\begin{split}
    &\hat g
    :=\sum_{j=1}^m\Big[
        \big(\bm c(u_j,u_j)-\bm c(\varphi,\varphi)\big) (DD_k\varphi-DD_ku_j)+\\
        &\hspace{7em}+\big(d(u_j,u_j)-d(\varphi,\varphi)\big)(D_k\varphi-D_ku_j)-\partial_{x_k}\ms B(x,u_j,Du_j)-D_kg_j
    \Big].\label{eq:gd}
    \end{split}
 	\end{align}
%Note that $y\in U$ can be any given point. 
Since $\ms B(x,0,0)=0$, then $\partial_{x_k}\ms B(x,0,0)=0$. By the mean value theorem,
\begin{equation}\label{eq:partial xk B estimate}
    |\partial_{x_k}\ms B(x,u_j,Du_j)|
    \leq \|\partial_{x_k}\ms B\|_{C^1}\Big(
        |u_j(x)|+|Du_j(x)|
    \Big).
\end{equation}
\begin{lemma}\label{lem:bound hat f and hat g}
Suppose that $U=B_R(0)$, $R\in (0,1]$ and $\{u_j\}\subset W^{2,2}(B_R(0))$. 
%Suppose that $\eta>0$ is a constant so that
%\[   \|\partial_z\partial_{x}\ms A\|_{C^0(\bm U_R)}|u_j|_{C^0(B_R(0))}\leq \eta.\]
Then there exists $C=C(n,m,\Lambda)$ such that
\begin{gather*}
\begin{split}
    \int_{B_R(0)}|\hat {\bm f}|^2\,\mr dx
    &\leq C\sum_{j=1}^m\int_{B_R(0)}\|u_j-\varphi\|^2_{C^1(B_R(0))}\Big(|DD_k\varphi-DD_ku_j|^2+|D\varphi-Du_j|^2+1\Big)\,\mr dx;
\end{split}\\
\begin{split}
    \int_{B_R(0)} |\hat g|\,\mr dx
    \leq C\sum_{j=1}^m\int_{B_R(0)}\Big( \|u_j-\varphi\|_{C^1(B_R(0))}(|DD_k\varphi-DD_ku_j|+|D\varphi-Du_j|)+\\
    +\|\partial_{x}\ms B\|_{C^1}\big(
    |u_j(x)|+|Du_j(x)|\big)+|Dg_j|_{C^0(B_1(0))}\Big)\,\mr dx;
\end{split}\\
\begin{split}
    \int_{B_R(0)} \partial_{x_k}\ms A(x,\varphi,D\varphi) D\phi\,\mr dx\leq C\int _{B_R(0)}\Big(\|\partial_x\partial_x\ms A\|_{C^1(\bm U_R)}|\varphi| |\phi|+|D\varphi| |\phi|
    + |D\varphi-D\varphi(0)| |D\phi|\\
    +\big(R+|\varphi-\varphi(0)|+|D\varphi-D\varphi(0)|\big)|D^2\varphi||\phi|\Big)\,\mr dx,
\end{split}
\end{gather*}
for all $\phi\in C_c(B_R(0))$.
\end{lemma}
\begin{proof}
The first one follows from Lemma \ref{lem:Abcd diff bound}. The same argument together with \eqref{eq:partial xk B estimate} gives the second inequality.

We now prove the third inequality. By a direct computation,
\begin{equation*}
\begin{split}
    \dv\,\partial_{x_k}\ms A(x,\varphi,D\varphi)=\sum_{i=1}^n\partial_{x_i}\partial_{x_k}\ms A(x,\varphi,D\varphi)+\partial_z\partial_{x_k}\ms A(x,\varphi,D\varphi)D\varphi+\\
    +\sum_{i=1}^n\partial_{\bm p^i}\partial_{x_k}\ms A(x,\varphi,D\varphi)DD_i\varphi.
    \end{split}
\end{equation*}
Then we have that 
\begin{equation}\label{eq:div for xz}
\begin{split}
    &\ \ \ \ \int_{B_R(0)} \partial_{x_k}\ms A(x,\varphi,D\varphi)\cdot D\phi\,\mr dx=-\int_{B_R(0)} \dv\, \partial_{x_k}\ms A(x,\varphi,D\varphi)\cdot \phi\,\mr dx\\
    &\leq C\int _{B_R(0)}\|\partial_x\partial_x\ms A\|_{C^1(\bm U_R)}(|\varphi|+|D\varphi|) |\phi|+\|\partial_z\partial_x\ms A\|_{C^0(\bm U_R)}|D\varphi| |\phi|\,\mr dx+\\
    &+\sum_{i=1}^n\Big|\int_{B_R(0)}  \partial_{\bm p^i}\partial_{x_k}\ms A(x,\varphi,D\varphi)DD_i\varphi\cdot \phi   \,\mr dx\Big|.
    \end{split}
    \end{equation}
Observe that 
\begin{align*}
    \partial_{\bm p^i}\partial_{x_k}\ms A(x,\varphi,D\varphi)DD_i\varphi
    &=\partial_{\bm p^i}\partial_{x_k}\ms A(0,\varphi(0),D\varphi(0))D\big(D_i\varphi-D_i\varphi(0)\big)+\\
    &+\Big(\partial_{\bm p^i}\partial_{x_k}\ms A(x,\varphi,D\varphi)-\partial_{\bm p^i}\partial_{x_k}\ms A(0,\varphi(0),D\varphi(0))\Big)DD_i\varphi\\
    &=\dv\,\Big(\partial_{\bm p^i}\partial_{x_k}\ms A(0,\varphi(0),D\varphi(0))\big(D_i\varphi-D_i\varphi(0)\big)\Big)+\\
    &+\Big(\partial_{\bm p^i}\partial_{x_k}\ms A(x,\varphi,D\varphi)-\partial_{\bm p^i}\partial_{x_k}\ms A(0,\varphi(0),D\varphi(0))\Big)DD_i\varphi.
\end{align*}
This implies that 
\begin{align*}
    &\ \ \ \ \Big|\int_{B_R(0)}  \partial_{\bm p^i}\partial_{x_k}\ms A(x,\varphi,D\varphi)DD_i\varphi\cdot \phi   \,\mr dx\Big|\\
    &\leq \Big|\int_{B_R(0)}\Big(\partial_{\bm p^i}\partial_{x_k}\ms A(0,\varphi(0),D\varphi(0))\big(D_i\varphi-D_i\varphi(0)\big)\Big)D\phi\,\mr dx\Big|+\\
    &+\Big|\int_{B_R(0)}\Big(\partial_{\bm p^i}\partial_{x_k}\ms A(x,\varphi,D\varphi)-\partial_{\bm p^i}\partial_{x_k}\ms A(0,\varphi(0),D\varphi(0))\Big)DD_i\varphi\cdot \phi\,\mr dx\Big|\\
    &\leq C\int_{B_R(0)}|D\varphi-D\varphi(0)|\cdot |D\phi|+\big(R+|\varphi-\varphi(0)|+|D\varphi-D\varphi(0)|\big)|D^2\varphi||\phi|\,\mr dx.
\end{align*}
This together with \eqref{eq:div for xz} gives the desired inequality.

\end{proof}

%%%%%%%%%%%%%%%%%%%%%%%%%%%%%%%%%%%%%%%%%%%%%%%%%%%%
%Section Height difference estimates
%%%%%%%%%%%%%%%%%%%%%%%%%%%%%%%%%%%%%%%%%%%%%%%%%%%%

\section{The height difference estimates}\label{sec:height difference}

In this section, we consider the height difference of two sheets. By the $\kappa$-condition, the height difference is a subsolution and a supsolution. Then the weak Harnack inequality gives the growth order around touching points. Then using the inductive methods, we can bound the first derivative of the difference by the distance to the touching sets. Such an estimate is the main result in Section \ref{subsec:first derivative of height difference} and will be crucial whenever we consider a differential equation because all of the remainder terms contain the height differences.

Then in Section \ref{subsec:W22 estiamtes}, we first consider the difference quotient of the height difference and then give the $W^{2,2}$ estimates. Finally, we obtain the uniform $W^{2,2}$ estimates for each sheet.

%%%%%%%%%%%%%%%%%%%%%%%%%%%%%%%%%%%%%%%%%%%%%%%%%%%
\subsection{The first derivatives of the height difference}
\label{subsec:first derivative of height difference}
Recall that for any $A\subset \mb R^n$ and $x\in \mb R^n$,
\[\dist (x,A):=\inf_{y\in A}\dist(x,y) .\]
We define $\dist(x,\emptyset)=+\infty$.

\begin{lemma}\label{lem:osc of diff}
Let $\{u_j\}_{j=1}^m$ be $C^{1,\alpha}$ ordered graphs over $B_1(0)\subset \mb R^n$ which is a solution to $(\mc Q,\{g_j\})$. Suppose that \eqref{eq:Lambda and lambda}, \eqref{eq:ms AB norm}, \eqref{eq:small C1alpha} and \eqref{eq:eta kappa0 and kappa1} are satisfied. Then for each $x_0\in B_1(0)$ and $r>0$ with $u_{i+1}(x_0)=u_{i}(x_0)$ for some $i\in\{1,\cdots,m-1\}$ and $B_{4r}(x_0)\subset B_1(0)$, we have that 
\[ |u_{i+1}(x)-u_{i}(x)|\leq C\cdot \kappa r^2, \quad \text{for each $x\in B_r(x_0)$},   \]	
where $C$ depends on $n,\Lambda/\lambda$. 
\end{lemma}
\begin{proof}
Recall by the $\kappa$-condition in Definition \ref{def:Q operator},for $j=1,\cdots,m$,
\begin{gather*}
     -\kappa\leq -\dv \Big( \ms A(x,u_j,Du_j) \Big)+\ms B(x,u_j,Du_j) \leq \kappa.
\end{gather*}
Let $w=u_{i+1}-u_{i}$. Then by Lemma \ref{lem:taking difference},
\[
-2\kappa\leq -\dv\Big(\bm A(u_{i+1},u_i)Dw+\bm  b(u_{i+1} ,u_i)w\Big)+\bm c(u_{i+1},u_i) Dw+d(u_{i+1},u_i)w\leq 2\kappa
\]
in the weak sense. Applying weak Harnack inequality \cite{GT}*{Theorem 8.17 and Theorem 8.18} with $q=n+1$, then we have 
\[
\sup_{x\in B_r(x_0)}w\leq C (\inf_{x\in B_r(x_0)}w+r^{2-\frac{2n}{q}}\|2\kappa\|_{L^{q/2}(B_{4r}(x_0))})\leq C\kappa r^2,
\]
where $C=C(n,\Lambda/\lambda)$. Hence Lemma \ref{lem:osc of diff} is proved.
\end{proof}

Now consider a point $x$ such that $u_i(x)\neq u_{i+1}(x)$. Then one can consider the distance $d$ from $x$ to the touching set of $u_i$ and $u_{i+1}$. By the choice, in the ball $B_d(x)$, $\{u_i\}$ has at least two connected components. Then by the subsystem condition \eqref{item:subsystem}, the average of each connected components will satisfy a PDE \eqref{eq:sum form}. Moreover, the difference of the averages will also satisfy a PDE \eqref{eq:diff of two average}. Applying the $C^{1,\alpha}$ interior estimates (Theorem \ref{thm:C1alpha}), Lemma \ref{lem:osc of diff}, one can inductively obtain the following estimates of gradients of $u_i-u_j$. We remark that such a method for two sheets graphs has been used to study the obstacle problems; see  \cite{Lin}*{Theorem 3.2}.

\begin{lemma}\label{lem:gradient diff}
Let $\{u_j\}_{j=1}^m$ be ordered $C^{1,\alpha}$ graphs over $B_1(0)\subset \mb R^n$ which is a solution to $(\mc Q,\{g_j\})$. Suppose that \eqref{eq:Lambda and lambda}, \eqref{eq:ms AB norm}, \eqref{eq:small C1alpha} and \eqref{eq:eta kappa0 and kappa1} are satisfied. Let $x_0\in B_1(0)$ %\textcolor{red}{(why not $z$ as above?)} 
and $r>0$ with $u_i(x_0)=u_{i+1}(x_0)$ for some $i\in\{1,\cdots,m-1\}$ and $B_{7r}(x_0)\subset B_1(0)$. Then we have that for each $x\in B_{r}(x_0)$,
\[  |Du_i(x)-Du_{i+1}(x)|\leq C(n,m,\Lambda,\lambda)\cdot (\kappa+\kappa_0) r.\]
	\end{lemma}
\begin{proof}
We prove the lemma inductively.	So long as $m=1$, then there is nothing to be proved. Suppose that the statements in the lemma hold true for the number of the sheets less than or equal to $m-1$. Now let $\{u_j\}_{j=1}^{m}$ be ordered graphs which is a solution to $(\mc Q,\{g_j\})$. Without loss of generality, we assume that $\{u_j\}_{j=1}^{m}$ is connected in $B_1(0)$.

Now fix $y\in B_r(x_0)$ and let $\wti r:=\dist(y,\{u_i=u_{i+1}\})$. Clearly, there exists 
\[ 
    x_i^*\in\{u_i=u_{i+1}\}\cap B_{2r}(x_0) \text{  with } \wti r=|y-x_i^*|.\]
So long as $\wti r=0$, then $Du_i(y)=Du_{i+1}(y)$. We now consider the case of $\wti r>0$. Since $x_0\in \{u_i=u_{i+1}\}$, then it follows that $\wti r<r$ and $B_{5\wti r}(x_i^*)\subset B_{7r}(x_0)\subset B_1(0)$. By the choice of $\wti r$, we know that $u_i<u_{i+1}$ in $B_{\wti r}(y)$. By applying Lemma \ref{lem:osc of diff} in $B_{5\wti r}(x_i^*)$ for $u_i$ and $u_{i+1}$, we have that 
\begin{equation}\label{eq:ui and ui+1 before scaling}
    |u_i(x)-u_{i+1}(x)|\leq C\kappa \wti r^2, \quad \forall\, x\in B_{\frac{5}{4}\wti r}(x_i^*)\supset B_{\frac{1}{4}\wti r}(y),
\end{equation}
where $C=C(n,\Lambda/\lambda)$.
Now we define new functions $v_1,\cdots,v_{m}:B_1(0)\to \mb R$ by
\[  v_j(x):=\wti r^{-1}u_j(\wti rx+y).\]
Since $u_{i+1}>u_i$ in $y\in B_{\wti r}(y)$, then $v_{i+1}>v_{i}$ in $B_1(0)$. Define $\wti {\mc Q}$ by 
\[  \wti {\mc Q}v=-\dv \wti {\ms A}(x,v,Dv)+\wti{\ms B}(x,v,Dv),\]
where 
\[  
\wti {\ms A}(x,z,\bm p):=\ms A(y+rx,rz,\bm p),\quad \wti {\ms B}(x,z ,\bm p):=r\ms B(y+rx,rz,\bm p).
\]
Define the new functions $\{\wti g_j\}$ by
\[ \wti g_j(x)=rg_j(y+rx).\] 
By Lemma \ref{lem:blow up}, $\{v_j\}$ is a solution to $(\wti{\mc Q},\{\wti g_j\})$. Moreover, by the subsystem condition, $\{v_j\}_{j=i+1}^{m}$ and $\{v_j\}_{j=1}^i$ are solutions to $(\wti {\mc Q},\{\wti g_j\}_{j={i+1}}^m)$ and $(\wti{\mc Q},\{\wti g_j\}_{j=1}^i)$, respectively.

By Lemma \ref{lem:large gap}, there exists $r_1\in (9^{-m-1},9^{-1})$ such that $\{v_j\}$ has the same number of connected components in $B_{9r_1}(0)$ and $B_{r_1}(0)$. For simplicity, we assume that $\{v_j\}_{j=1}^i$ and $\{v_j\}_{j=i+1}^m$ are connected components in $B_{r_1}(0)$. Observe that for $j=1,\cdots,m$, 
\[ -\kappa\wti r\leq \wti{\mc Q}v_j\leq \kappa \wti r. \]
That is, $\{v_j\}$ satisfies $\kappa\wti r$-condition. Since $\{v_j\}_{j=1}^i$ and $\{v_j\}_{j=i+1}^m$ are connected in $B_{r_1}(0)$, then for any $j=1,\cdots,i-1$ or $i+1,\cdots m-1$, there exists $x_j^*\in B_{r_1}(0)$ so that $v_j(x_j^*)=v_{j+1}(x_j^*)$. Applying Lemma \ref{lem:osc of diff} in $B_{8r_1}(x_j^*)\subset B_{9r_1}(0)\subset B_1(0)$, we have that for all $x\in B_{2r_1}(x_j^*) \supset B_{r_1}(0)$,
\begin{equation}\label{eq:diff order except j}   
    |v_j(x)-v_{j+1}(x)|\leq C(n,\Lambda/\lambda)\kappa\wti r, \quad j=1,\cdots,i-1, i+1,\cdots, m-1.
    \end{equation}
On the other hand, by \eqref{eq:ui and ui+1 before scaling},
\[
 |v_i(x)-v_{i+1}(x)|\leq C\kappa \wti r, \quad \forall x\in B_{\frac{1}{4}}(0)\supset B_{r_1}(0).
\]
Together with \eqref{eq:diff order except j}, we conclude that for all $x \in B_{r_1}(0)$ and $j=1,\cdots,m$,
\begin{equation}\label{eq:diff of vj and vj+1}
    |v_j(x)-v_{j+1}(x)|\leq C(n,\Lambda/\lambda,m)\kappa\wti r.
\end{equation}
By induction, there exists a constant $C=C(n,m,\Lambda,\lambda)$ so that for all $x\in B_{r_1}(0)$,
\begin{equation}\label{eq:diff derivative by induction}
     |Dv_j(x)-Dv_{j+1}(x)|\leq C \cdot (\kappa+\kappa_0)\wti r 
\end{equation}  
for $j=1,\cdots, i-1,i+1,\cdots,m$. Let 
\[
\varphi:=\frac{1}{i}\sum_{j=1}^iv_j,\quad \psi:=\frac{1}{m-i}\sum_{j=i+1}^mv_j,\quad \Psi:=\varphi-\psi.
\]
Then by \eqref{eq:diff of two average}, $\Psi$ satisfies the equation
\begin{equation}\label{eq:diff of two solution}
    \begin{split}
    -\dv \Big(\bm A(\varphi,\psi)D\Psi+\bm b(\varphi,\psi)\Psi\Big)+\bm c(\varphi,\psi) D\Psi+d(\varphi,\psi)\Psi =\dv {\bm f^*}-g^*,
    \end{split}
\end{equation}
where $\bm A,\bm b,\bm c$ and $d$ are defined with respect to $\wti{\mc Q}$; and by Lemma \ref{lem:bound bm f and g}, Lemma \ref{lem:blow up}, Equation \eqref{eq:diff of vj and vj+1} and \eqref{eq:diff derivative by induction},
\begin{gather*}
|\bm f^*|\leq C(n,m,\Lambda,\lambda)(\kappa+\kappa_0) \wti r ,\quad  \sup_{x,x'\in B_{r_1}(0)}\frac{|\bm f^*(x)-\bm f^*(x')|}{|x-x'|^\alpha}\leq C(n,m,\Lambda,\lambda)(\kappa+\kappa_0) \wti r;\\
|g^*|\leq C(n,m,\Lambda,\lambda)(\kappa +\kappa_0) \wti r.
\end{gather*}
Then the $C^{1,\alpha}$-estimates Theorem \ref{thm:C1alpha} give that for $x\in B_{r_1/2}(0)$,
\begin{align*} 
	|D\Psi(x)| &\leq C(n,\alpha)(\|\Psi\|_{C^0(B_{r_1}(0))}+\|g^*\|_{C^0(B_{r_1}(0))}+\|\bm f^*\|_{C^\alpha(B_{r_1}(0))})\\
	&\leq C(n,m,\Lambda,\lambda)(\kappa+\kappa_0)\wti r.
\end{align*}
Together with \eqref{eq:diff derivative by induction}, we then conclude that
\[ |Du_{i+1}(y)-Du_i(y)| = |Dv_{i+1}(0)-Dv_i(0)| \leq C(n,m,\Lambda,\lambda) (\kappa+\kappa_0)\wti r. \]
This completes the induction. By the arbitrariness of $y$, we have proved Lemma \ref{lem:gradient diff}.
	\end{proof}

%%%%%%%%%%%%%%%%%%%%%%%%%%%%%%%%%%%%%%%%%%%%%%%%%%%%
\subsection{Integral estimates of the second derivatives}
\label{subsec:W22 estiamtes}
Recall that for any function $f$ and a constant $\tau\neq 0$, 
\[ 
    f^\tau(x):=\frac{f(x+\tau \bm e)-f(x)}{\tau},\]
where $\bm e$ can be any unit vector in $\mb R^n$. 
\begin{lemma}\label{lem:diff to sum}
Let $\{u_j\}_{j=1}^m$ be $C^{1,\alpha}$ ordered, connected graphs over $B_1(0)\subset \mb R^n$ which is a solution to $(\mc Q,\{g_j\})$. 
Suppose that \eqref{eq:Lambda and lambda}, \eqref{eq:ms AB norm}, \eqref{eq:ms AB at x00}, \eqref{eq:small C1alpha} and \eqref{eq:eta kappa0 and kappa1} are satisfied.
Denote by
\[   \eta:=\Big(\|\partial_z\partial_x\ms A\|_{C^0(\bm U_1)}+\|\partial_x\ms B\|_{C^1(\bm U_1)}\Big)\cdot \sum_{j=1}^m\|u_j\|_{C^0(B_{1}(0))}.\]
Suppose that there exists a constant $K>0$ such that for all $1\leq i\leq m$ and $\tau\in(0,1/4)$,
\[ 
    \int_{B_{1-\tau}(0)} |Du_i^\tau-Du_j^\tau|^2\, \mr dx\leq K.    \]
Then the average 
$\varphi=\frac{1}{m}\sum_ju_j$ satisfies
\[ 
    \int_{B_{1/2}(0)}|D\varphi^\tau|^2
    \leq C(n,m,\Lambda,\lambda)\Big(\delta_0^2(K+1)+\eta^2+\kappa_1\delta_0\Big).\]	
	\end{lemma}
\begin{proof}
By \eqref{eq:differential quotient sum}, in $B_{3/4}(0)$,
\begin{align}\label{eq:varphi tau pde}
    \sum_{j=1}^m\Big[  -\dv \Big( {\bm A}_\tau(u_j)D\varphi^\tau 
        + {\bm b}_\tau(u_j)\varphi^\tau \Big)
        + \bm c_\tau(u_j)D\varphi^\tau 
        +d_\tau(u_j) \varphi^\tau \Big] 
    =-\dv\bm f_\tau+g_\tau,  
	 \end{align}
where $\bm f_\tau$ and $g_\tau$ are given by \eqref{eq:f tau} and \eqref{eq:g tau}. 
By \eqref{eq:Lambda and lambda} and \eqref{eq:eta kappa0 and kappa1}, we have
\begin{equation}\label{eq:bound Abcd tau} 
    \lambda \leq \bm A_\tau\leq\Lambda  , \quad 
    |\bm b_\tau| + |\bm c_\tau|+|d_\tau|\leq 3\Lambda.  
\end{equation}
Moreover, by \eqref{eq:small C1alpha},
\begin{equation}\label{eq:bound utau}
    |u_j^\tau| \leq \sup|Du_j|\leq \delta_0.
\end{equation}
Since $\{u_j\}$ is connected, then for all $j=1,\cdots,m-1$, there exists $x_j\in B_1(0)$ such that $u_j(x_j)=u_{j+1}(x_j)$, which yields that for all $x\in B_1(0)$,
\begin{align*}
    |u_j(x)-u_{j+1}(x)|
    &\leq |u_j(x)-u_{j}(x_j)|+|u_{j+1}(x_j)-u_{j+1}(x)|\\
    &\leq 2\delta_0|x-x_j|\leq 4\delta_0.
\end{align*}
By \eqref{eq:small C1alpha},
\[  |Du_j(x)-Du_{j+1}(x)|\leq 2\delta_0.\]
By the arbitrariness of $j$, we conclude that 
\[  \|u_j-\varphi\|_{C^1(B_1(0))}\leq C(m)\delta_0.\]
This together with Lemma \ref{lem:bound ftau and gtau} implies
\begin{gather}
\begin{split}\label{eq:bound ftau} 
    |\bm f_\tau|
    \leq C(n,m,\Lambda)\sum_{j=1}^m\Big(\delta_0 |D\varphi^\tau-Du_j^\tau|+\eta+\delta_0\Big);
\end{split}\\
\begin{split}\label{eq:bound gtau}
    |g_\tau|
    \leq C\sum_{j=1}^m\Big(\delta_0 |D\varphi^\tau-Du_j^\tau|+\delta_0+\eta+\kappa_1\Big);
\end{split}
\end{gather}

Now let $\phi= \varphi^\tau\cdot \zeta^2$, where $\zeta\in C_c^1(B_{3/4}(0))$ is a cut-off function so that 
\begin{equation}\label{eq:cut off zeta}
    |D \zeta|\leq C; \quad 
    0\leq \zeta\leq 1; \quad 
    \zeta\equiv 1 \text { on } B_{1/2}(0).
\end{equation}
Then by taking $\phi$ as a test function in \eqref{eq:varphi tau pde} and using \eqref{eq:bound Abcd tau}, we have
\begin{align*}
    \int_{B_{3/4}(0)}|D\varphi^\tau|^2\zeta^2\,\mr dx
    &\leq C\int_{B_{3/4}(0)} \Big(|D\varphi^\tau||\varphi^\tau||D\zeta^2| 
          +|\varphi^\tau||D(\varphi^\tau\zeta^2)| +|D\varphi^\tau||\varphi^\tau|\zeta^2
          +|\varphi^\tau|^2\zeta^2+\\
          &\hspace{12em}+|\bm f_\tau|(|D\varphi^\tau|\zeta^2+|\varphi^\tau||D\zeta^2|)+|g_\tau||\varphi^\tau|\zeta^2\Big)\,\mr dx\\
    &\leq \frac{1}{2}\int_{B_{3/4}(0)} |D\varphi^\tau|^2\zeta^2\,\mr dx 
          +C\int_{B_{3/4}(0)}\Big(|\varphi^\tau|^2|D\zeta|^2 +|\varphi^\tau|^2\zeta^2+|\bm f_\tau|^2\zeta^2\\
          &\hspace{24.5em}+|g_\tau||\varphi^\tau|\zeta^2\Big)\,\mr dx.
\end{align*}
where $C=C(m,\Lambda,\lambda)$. Plugging \eqref{eq:bound utau}, \eqref{eq:bound ftau}, \eqref{eq:bound gtau} and \eqref{eq:cut off zeta} into the last inequality, we obtain
\begin{align*}
    \int_{B_{3/4}(0)}|D\varphi^\tau|^2\zeta^2\,\mr dx
        &\leq C\Big(\sum_{j=1}^m\int_{B_{3/4}(0)}\delta_0^2 +\delta_0^2|D\varphi^\tau-Du_j^\tau|^2\,\mr dx
        +\eta^2+\kappa_1\delta_0\Big)\\
        &\leq C\Big(\delta_0^2(K+1)+\eta^2+\kappa_1\delta_0\Big),
	\end{align*}
where $C=C(n,m,\Lambda,\lambda)$. This finishes the proof of Lemma \ref{lem:diff to sum}.
	\end{proof}

Now we are ready to prove the hessian estimates for height differences. We will use the inductive method on the number of sheets $m$. For the point $x$ which has density less than $m$, we will take $r_1$ to be the largest radius so that $\{u_j\}$ is not connected in $B_{r_1}(x)$. Then by induction, we can prove the integral estimates over $B_{c(m)r_1}(x)$, where $c(m)\in(0,1)$ is a constant depending only on $m$; Lemma \ref{lem:diff to sum} will be used in this part. Moreover, for the touching points, we will use Lemma \ref{lem:gradient diff} to find such a small radius with desired integral bound. Finally, we will finish the induction by taking a covering argument.

\begin{lemma}\label{lem:diff 22}
Let $\{u_j\}_{j=1}^m$ be $C^{1,\alpha}$ ordered graphs over $B_1(0)\subset \mb R^n$ which is a solution to $(\mc Q,\{g_j\})$. 
Suppose that \eqref{eq:Lambda and lambda}, \eqref{eq:ms AB norm}, \eqref{eq:ms AB at x00}, \eqref{eq:small C1alpha} and \eqref{eq:eta kappa0 and kappa1} are satisfied. 
Denote by
\[   \eta:=\Big(\|\partial_z\partial_x\ms A\|_{C^0(\bm U_1)}+\|\partial_x\ms B\|_{C^1(\bm U_1)}\Big)\cdot \sum_{j=1}^m\|u_j\|_{C^0(B_{1}(0))}.\]
If in addition that $\{u_j\}$ is connected in $B_{1/16}(0)$, then for all $\tau\in(0,1/32)$ and $j=1,\cdots,m-1$,
\[ 
    \int_{B_{1/32}(0)}|Du_j^\tau-Du_{j+1}^\tau|^2\,\mr dx
        \leq C\cdot \Big((1+\eta^2)(\kappa+\kappa_0)^2+\kappa_1^2\Big),   \]
where $C=C(n,m,\Lambda,\lambda)$.
This gives that 
$u_i-u_j\in W^{2,2}(B_{1}(0))$ and 
\[ 
    \int_{B_{1/32}(0)}|D^2(u_i-u_j)|^2\,\mr dx
    \leq C\cdot \Big((1+\eta^2)(\kappa+\kappa_0)^2+\kappa_1^2\Big).     \]
As a corollary, if $\{u_j\}$ is connected in $B_s(0)$ and $s\leq 1/32$, then 
\[ 
    \int_{B_{s}(0)}|D^2(u_i-u_j)|^2\,\mr dx
    \leq C\cdot \Big((1+\eta^2)(\kappa+\kappa_0)^2+\kappa_1^2\Big)s^n.  
        \]
\end{lemma}
\begin{proof}
We are going to prove the lemma by induction. 
So long as $m=1$, then there is nothing to be proved.
Suppose that the lemma is true when the number of the sheets is less than or equal to $m-1$. 
Now let $\{u_j\}_{j=1}^{m}$ be ordered graphs which is a solution to $(\mc Q,\{g_j\})$. 
We further assume that $\{u_j\}$ is connected in $B_{1/16}(0)$. 
For any $x\in B_1(0)$, let 
\[ 
    r_1(x):=\sup\left\{ r'<1-|x|; \{u_j\} 
        \text{ is not connected in $B_{r'}(x)$}\right\}.  \] 
Then $r_1(x_0)>0$ if $u_1(x_0)\neq u_m(x_0)$ for $x_0\in B_1(0)$. 

\medskip
{\noindent\em Step 1: Decomposing $\{u_j\}$ into connected components in small balls.}

\medskip
Note that by Lemma \ref{lem:large gap}, there exists 
$r(x)\in (16^{-m}r_1(x),r_1(x))$ 
so that $\{u_j\}$ has same number of connected components in $B_{ r(x)}(x)$ and $B_{r(x)/16}(x)$. 
Since $\{u_j\}$ is connected in $B_{1/16}(0)$, then for each $i=1,\cdots,m-1$, there exists 
$x_i\in B_{1/16}(0)$ with $u_i(x_i)=u_{i+1}(x_i)$. 
Clearly, 
$B_{1/8}(0)\subset B_{3/16}(x_i)$ and $B_{3/4}(x_i)\subset B_1(0)$. 
In particular, for all $x\in B_{1/16}(0)$, 
\[  
    \dist(x,\{u_i=u_{i+1}\})\leq \frac{1}{8}<1-|x|.
\]
Then by Lemma \ref{lem:osc of diff} and Lemma \ref{lem:gradient diff},
there exists $C=C(n,m,\Lambda,\lambda)$ such that for all $x\in B_{1/16}(0)$,
\begin{gather*}
    |u_i(x)-u_{i+1}(x)|
        \leq C\cdot \kappa[\dist(x,\{u_i=u_{i+1}\})]^2
        \leq C \kappa(r(x))^2;\\
    |Du_i(x)-Du_{i+1}(x)|
        \leq C\cdot (\kappa+\kappa_0)\dist(y,\{u_i=u_{i+1}\})
        \leq C (\kappa+\kappa_0)r(x).
\end{gather*}
By the arbitrariness of $i$, we conclude that for all $x\in B_{1/16}(0)$ and $i,j\in\{1,\cdots,m\}$,
\begin{equation}\label{eq:oscillation}
    |u_i(x)-u_j(x)|\leq C\kappa\big(r(x)\big)^2;\quad 
    |Du_i(x)-Du_j(x)|\leq C(\kappa+\kappa_0) r(x).
\end{equation}
Now we fix $y\in B_{1/16}(0)$ and use $r$ to denote $r(y)$ when there is no ambiguity. For any $x\in B_1(0)$, define
\[  
    v_j(x):=r^{-1}u_j(rx+y), \quad 
    \wti g_j(x)=rg_j(rx+y).\]
Then by the choice of $r_1$ and $r$, $\{v_j\}$ is not connected in $B_1(0)$. By relabeling, we have connected components 
\[
    \{v_{1,1},\cdots,v_{1,i_1}\},\quad 
    \{v_{2,1},\cdots,v_{2,i_2}\},\quad \cdots, \quad 
    \{v_{\mk m,1},\cdots,v_{\mk m,i_\mk m}\},
\]
and the associated functions
\[
    \{\wti g_{1,1,},\cdots,\wti g_{1,i_1}\},\quad 
    \{\wti g_{2,1},\cdots,\wti g_{2,i_2}\},\quad \cdots, \quad 
    \{\wti g_{\mk m,1},\cdots,\wti g_{\mk m,i_\mk m}\}.
\]
We assume that the order is preserved, i.e. 
\[   
    v_{1,1}\leq \cdots \leq v_{1,i_1}< v_{2,1}\leq \cdots <v_{\mk m,1}\leq \cdots\leq v_{\mk m,i_\mk m}.
\]
Then by \eqref{eq:oscillation}, for all $j,k\in\{1,\cdots,\mk m\}$ and $1\leq \ell\leq i_j$, $1\leq \ell'\leq i_k$,
\begin{equation}\label{eq:osc of v}
    |v_{j,\ell}(x)-v_{k,\ell'}(x)|\leq C\kappa r(y),\quad  
    |Dv_{j,\ell}(x)-Dv_{k,\ell'}(x)|\leq C\cdot (\kappa+\kappa_0)r(y).
\end{equation}

Define
\[  
    \varphi_j:=\frac{1}{i_j}\sum_{\ell=1}^{i_j}v_{j,\ell}.
\]
By the subsystem condition in Definition \ref{def:Q operator} and Lemma \ref{lem:blow up}, for each $j=1,\cdots, \mk m$, $\{v_{j,\ell}\}_\ell$ is a solution to $(\wti{\mc Q},\{\wti g_{j,\ell}\}_\ell)$, where
\[  \wti {\mc Q}v=-\dv\, \wti {\ms A}(x,v,Dv)+\wti{\ms B}(x,v,Dv),\]
and 
\[  
    \wti {\ms A}(x,z,\bm p):=\ms A(y+rx,rz,\bm p),\quad 
    \wti {\ms B}(x,z, \bm p):=r\ms B(y+rx,rz,\bm p).
\]
Clearly, by \eqref{eq:Lambda and lambda},
\begin{equation}\label{eq:blow up wti AB bound}
    \lambda\leq \partial_{\bm p}\wti{\ms A}\leq \Lambda,\quad   
    \|{\wti{\ms A}}\|_{C^2}+\|\wti{\ms B}\|_{C^2}\leq \Lambda. 
\end{equation}
Moreover, by the definition of $\wti{\ms A}$ and $\wti{\ms B}$, 
\begin{equation}\label{eq:partial x wti AB}
\begin{split}
    \|\partial_{x}\wti{\ms A}\|_{C^1}+\|\partial_{z}\wti{\ms A}\|_{C^1}+\|\wti{\ms B}\|_{C^2}\leq C(n,\Lambda)\cdot r, \\
    \|\partial_z\partial_{x}\wti{\ms A}\|_{C^0}+ \|\partial_{x}\wti{\ms B}\|_{C^1}+\|\partial_{z}\wti{\ms B}\|_{C^1}\leq C(n,\Lambda)\cdot r^2. 
    \end{split}
\end{equation}

Then for any $\tau\in(0,1/4)$, $\varphi_j^\tau$ satisfies the equation in \eqref{eq:differential quotient sum} (w.r.t. $(\wti{\mc Q},\{\wti g_{j,\ell}\}_\ell)$) in $B_{1-\tau}(0)$ in the weak sense,
\[
    -\dv\Big({\bm A}_\tau^jD\varphi^\tau_j+{\bm b}_\tau^j\varphi_j^\tau\Big) + \bm c_\tau^j D\varphi^\tau_j+d_\tau^j \varphi^\tau_j
    =-\dv\bm f_\tau^j+g_\tau^j,  
\]
 where 
 \begin{gather*}
    \bm A_\tau^j:=\sum_{\ell=1}^{i_j}{\bm A_\tau}(v_{j,\ell});\quad  
        \bm b_\tau^j:=\sum_{\ell=1}^{i_j}{\bm b_\tau}(v_{j,\ell});\quad 
        \bm c_\tau^j:=\sum_{\ell=1}^{i_j}{\bm c_\tau}(v_{j,\ell});\quad 
        d_\tau^j=\sum_{j=1}^md_\tau(v_{j,\ell}) ;\\
    \bm f_\tau^j:=\sum_{\ell=1}^{i_j}\Big(
        (\bm A_\tau(v_{j,\ell}) -\bm A_\tau(\varphi_j)) (D\varphi^\tau_j-Dv_{j,\ell}^\tau)
            +\big(\bm b_\tau(v_{j,\ell})
            -\bm b_\tau(\varphi_j)\big)(\varphi^\tau_j-v_{j,\ell}^\tau) -\mk A_\tau(v_{j,\ell})
            \Big);\\
    g_\tau^j=\sum_{\ell=1}^{i_j}\Big(
        \big(\bm c_\tau(v_{j,\ell})
        -\bm c_\tau(\varphi_j)\big) (D\varphi^\tau_j-Dv_{j,\ell}^\tau)+\big(d_\tau(v_{j,\ell})
        -d_\tau(\varphi_j)\big)(\varphi_j^\tau-v_{j,\ell}^\tau)
        -\mk B_\tau(v_{j,\ell}) -\wti g_{j,\ell}^\tau
    \Big).
 	\end{gather*}
Lemma \ref{lem:bound ftau and gtau} gives the bound of $\bm f_\tau^j$ and $g_\tau^j$:
\begin{gather*}
\begin{split}
    \Big|\bm f_\tau^j+\sum_{\ell=1}^{i_j}\mk A_\tau(v_{j,\ell})\Big|
    \leq C\sum_\ell\Big(\|v_{j,\ell}-\varphi_j\|_{C^1}\cdot |D\varphi_j^\tau-Dv_{j,\ell}^\tau|+\|v_{j,\ell}-\varphi_j\|_{C^1}\cdot |\varphi_j^\tau-v_{j,\ell}^\tau|\Big);
\end{split}\\
\begin{split}
    \Big|g_\tau^j+\sum_{\ell=1}^{i_j}\mk B_\tau(v_{j,\ell})\Big|
    \leq C\sum_\ell\Big(\|v_{j,\ell}-\varphi_j\|_{C^1}\cdot |D\varphi_j^\tau-Dv_{j,\ell}^\tau|+\|v_{j,\ell}-\varphi_j\|_{C^1}\cdot |\varphi_j^\tau-v_{j,\ell}^\tau|+\\
      +|D\wti g_{j,\ell}|_{C^0}\Big);
\end{split}
\end{gather*}  
This together with \eqref{eq:osc of v} and \eqref{eq:partial x wti AB} gives that
\begin{gather}\label{eq:ftauj+mk A}
    \Big|\bm f_\tau^j + \sum_{\ell=1}^{i_j}\mk A_\tau(v_{j,\ell})\Big|
    \leq C\Big(
        (\kappa+\kappa_0)r\sum_{\ell=1}^{i_j}|D\varphi^\tau_j-Dv_{j,\ell}^\tau|
        +(\kappa+\kappa_0)^2r^2
    \Big).\\
    \label{eq:gtauj+mk B}
    \Big|g_\tau^j+\sum_{\ell=1}^{i_j}\mk B_\tau(v_{j,\ell})\Big|
    \leq C\Big(
        (\kappa+\kappa_0)r\sum_{\ell=1}^{i_j}|D\varphi^\tau_j-Dv_{j,\ell}^\tau|
        +(\kappa+\kappa_0)^2r^2+\kappa_1 r^2
    \Big).    
\end{gather}

\medskip
{\noindent\em Step 2: Estimating the height of two connected components.}

\medskip 
For $j,k\in\{1,\cdots, \mk m\}$, letting $\Psi:=\varphi_j-\varphi_k$, we obtain the equation for $\Psi^\tau$ in $B_{1-\tau}(0)$ as follows:
\begin{equation}\label{eq:pde of varphijtau-varphiktau}
    -\dv \big( \overline {\bm A}D\Psi^\tau + \oli{\bm b}\Psi^\tau \big)
          +\oli{\bm c} D\Psi^\tau+\oli{d}\Psi^\tau
    =-\dv\oli{\bm f}+\oli{g},
\end{equation}
where
\begin{align*}
    &\oli{\bm A}:=\frac{1}{i_j}\bm A_\tau^j; \quad  
        \oli{\bm b}:=\frac{1}{i_j}\bm b_\tau^j;\quad 
        \oli{\bm c}:=\frac{1}{i_j}\bm c_\tau^j; \quad 
        {d}:=\frac{1}{i_j}d_\tau^j;\\
    &\oli{\bm f}:=-\Big(\frac{1}{i_j}\bm A_\tau^j
                 -\frac{1}{i_k}\bm A_\tau^k\Big) D\varphi_k^\tau
                 -\Big(\frac{1}{i_j}\bm b_\tau^j-\frac{1}{i_k}\bm b_\tau^k\Big) \varphi_k^\tau
                 +\Big(\frac{1}{i_j}\bm f_\tau^j-\frac{1}{i_k}\bm f_\tau^k\Big);\\
    &\oli{g}=-\Big(\frac{1}{i_j}\bm c_\tau^j
            -\frac{1}{i_k}\bm c_\tau^k\Big)  D\varphi_k^\tau
            -\Big(\frac{1}{i_j}d_\tau^j-\frac{1}{i_k}d_\tau^k\Big) 
            \varphi_k^\tau
            +\Big(\frac{1}{i_j}g_\tau^j-\frac{1}{i_k}g_\tau^k \Big).
\end{align*}
By Lemma \ref{lem:Xtau difference of u v} and \eqref{eq:osc of v}, \eqref{eq:blow up wti AB bound}, \eqref{eq:partial x wti AB}
\begin{equation}\label{eq:bm Atauj-bm Atauk}
    \Big|\frac{1}{i_j}\bm A_\tau^j-\frac{1}{i_k}\bm A_\tau^k\Big|
    \leq C\sup_{\ell,\ell'} |\bm A_\tau(v_{j,\ell})-\bm A_\tau(v_{k,\ell'})|\leq C(n,m,\Lambda,\lambda)(\kappa+\kappa_0)r.
\end{equation}
A similar argument gives that 
\begin{gather}
    \label{eq:btauj ctauj and dtauj}
    \Big|\frac{1}{i_j}\bm b_\tau^j-\frac{1}{i_k}\bm b_\tau^k\Big|
    +\Big|\frac{1}{i_j}\bm c_\tau^j-\frac{1}{i_k}\bm c_\tau^k\Big|
    +\Big|\frac{1}{i_j}d_\tau^j-\frac{1}{i_k}d_\tau^k\Big|
    \leq C(n,m,\Lambda,\lambda)(\kappa+\kappa_0)r^2;\\
    \label{eq:mk A tau vjl and mk tau vkl}
    \Big|\frac{1}{i_j}\sum_{\ell=1}^{i_j}\mk A_\tau(v_{j,\ell})-\frac{1}{i_k}\sum_{\ell=1}^{i_k}\mk A_\tau(v_{k,\ell})\Big|
    \leq \|\partial_x\wti{\ms A}\|_{C^1}\|v_{j,\ell}-v_{k,\ell'}\|_{C^1}\leq C (\kappa+\kappa_0)r^2;\\
    \label{eq:mk B tau vjl and mk tau vkl}
    \Big|\frac{1}{i_j}\sum_{\ell=1}^{i_j}\mk B_\tau(v_{j,\ell})-\frac{1}{i_k}\sum_{\ell=1}^{i_k}\mk B_\tau(v_{k,\ell})\Big|
    \leq \|\partial_x\wti{\ms B}\|_{C^1}\|v_{j,\ell}-v_{k,\ell'}\|_{C^1}\leq C (\kappa+\kappa_0)r^3.
\end{gather}
Obviously, \eqref{eq:mk A tau vjl and mk tau vkl} combining with \eqref{eq:ftauj+mk A}  gives that 
\begin{equation}\label{eq:ftauj-ftauk}
\begin{split}
    &\hspace{1.5em}\Big|\frac{1}{i_j}\bm f_\tau^j-\frac{1}{i_k}\bm f_\tau^k\Big|\\
    &\leq \frac{1}{i_j}\Big|\bm f_\tau^j+\sum_{\ell=1}^{i_j}\mk A_\tau(v_{j,\ell})\Big|
    +\Big|\frac{1}{i_j}\sum_{\ell=1}^{i_j}\mk A_\tau(v_{j,\ell})-\frac{1}{i_k}\sum_{\ell=1}^{i_k}\mk A_\tau(v_{k,\ell})\Big|+\frac{1}{i_k}\Big|\bm f_\tau^k+\sum_{\ell=1}^{i_k}\mk A_\tau(v_{k,\ell})\Big|\\
    &\leq C(\kappa+\kappa_0)r\sum_{\ell=1}^{i_j}|D\varphi^\tau_j-Dv_{j,\ell}^\tau|+C(\kappa+\kappa_0)r^2+C(\kappa+\kappa_0)r\sum_{\ell=1}^{i_k}|D\varphi^\tau_k-Dv_{k,\ell}^\tau|. 
\end{split}
\end{equation}
Similarly, \eqref{eq:mk B tau vjl and mk tau vkl} and \eqref{eq:gtauj+mk B} imply
\begin{equation}\label{eq:gtauj-gtauk}
\begin{split}
    &\hspace{1.5em}\Big|\frac{1}{i_j}g_\tau^j-\frac{1}{i_k}g_\tau^k\Big|\\
    &\leq \frac{1}{i_j}\Big|g_\tau^j+\sum_{\ell=1}^{i_j}\mk B_\tau(v_{j,\ell})\Big|
    +\Big|\frac{1}{i_j}\sum_{\ell=1}^{i_j}\mk B_\tau(v_{j,\ell})-\frac{1}{i_k}\sum_{\ell=1}^{i_k}\mk B_\tau(v_{k,\ell})\Big|+\frac{1}{i_k}\Big|g_\tau^k+\sum_{\ell=1}^{i_k}\mk B_\tau(v_{k,\ell})\Big|\\
    &\leq C(\kappa+\kappa_0)r\sum_{\ell=1}^{i_j}|D\varphi^\tau_j-Dv_{j,\ell}^\tau|+C(\kappa+\kappa_0+\kappa_1)r^2+C(\kappa+\kappa_0)r\sum_{\ell=1}^{i_k}|D\varphi^\tau_k-Dv_{k,\ell}^\tau|. 
\end{split}
\end{equation}
By \eqref{eq:bm Atauj-bm Atauk}, \eqref{eq:btauj ctauj and dtauj}, \eqref{eq:ftauj-ftauk} and \eqref{eq:gtauj-gtauk}, we conclude that 
\begin{equation}\label{eq:|f|+|g|}
\begin{split}
    |\oli{\bm f}|+|\oli g|
        \leq C\cdot (\kappa+\kappa_0)r (|D\varphi_k^\tau|+|\varphi_k^\tau|) +C(\kappa+\kappa_0)r\sum_{\ell=1}^{i_j}|D\varphi_j^\tau-Dv_{j,\ell}^\tau|+\\
+C(\kappa+\kappa_0)r\sum_{\ell=1}^{i_k}|D\varphi_k^\tau-Dv_{k,\ell}^\tau|+C(\kappa+\kappa_0+\kappa_1)r^2,
   \end{split}
\end{equation}
where $C=C(n,m,\Lambda,\lambda)$.

\medskip
Now let $\zeta\in C_c^1(B_{1/64}(0))$ be a cut-off function so that 
\[  
    0\leq \zeta\leq 1; \quad 
    |D\zeta|\leq C; \quad 
    \zeta\equiv 1 \ \text{for all } x\in B_{1/128}(0).
\]
By taking $\phi=\Psi^\tau\zeta^2$ as a test function in \eqref{eq:pde of varphijtau-varphiktau} and using the ellipticity of $\oli{\bm A}$ and the bound of $\oli{\bm b}$, $\oli{\bm c}$, $\oli{d}$, we obtain that 
\begin{align*}
    \int_{B_{3/4}(0)}|D\Psi^\tau|^2\zeta^2\,\mr dx
        &\leq C\int_{B_{3/4}(0)}\Big( |\Psi^\tau|^2|D\zeta|^2 +|D\Psi^\tau||\Psi^\tau|\zeta^2 +|\Psi^\tau|^2\zeta^2+\\
            &\hspace{11em}+|\oli{\bm f}|(|D\Psi^\tau|\zeta^2 +|D\zeta^2||\Psi^\tau|) +|\oli{g}|\zeta^2|\Psi^\tau| \Big)\,\mr dx\\
        &\leq \frac{1}{2}\int_{B_{3/4}(0)}|D\Psi^\tau|^2\zeta^2\,\mr dx +C\int_{B_{3/4}(0)}(\kappa+\kappa_0)^2r^2
        +|\oli {\bm f}|^2\zeta^2+|\oli{g}|^2\zeta^2\,\mr dx,
\end{align*}
where we used $|\Psi^\tau|\leq C\cdot (\kappa+\kappa_0)r$ following from \eqref{eq:osc of v} and the mean value theorem.
After simplification, it becomes
\begin{equation}\label{eq:DPsi2 and f2 g2}
    \int_{B_{3/4}(0)}|D\Psi^\tau|^2\zeta^2\,\mr dx
        \leq C\int_{B_{3/4}(0)} ( \kappa+\kappa_0 )^2r^2 
        + |\oli {\bm f}|^2\zeta^2 + |\oli{g}|^2\zeta^2\,\mr dx.
\end{equation}
Now we estimate the right hand side of this equation. Observe that 
\begin{align*}
    &\Big(
        \|\partial_z\partial_x \wti{\ms A}\|_{C^0(B_1(0))}+\|\partial_x \wti{\ms B}\|_{C^1(B_1(0))}
    \Big)\cdot \sum_{\ell=1}^{i_j}\|v_{j,\ell}\|_{C^0(B_1(0))}\\
    \leq\, & \Big(
       r^2 \|\partial_z\partial_x {\ms A}\|_{C^0(B_1(0))}+r^2\|\partial_x {\ms B}\|_{C^1(B_1(0))}
    \Big)\cdot \sum_{\ell =1}^m\|u_{\ell}\|_{C^0(B_1(0))}\cdot r^{-1}\leq \eta r.
    \end{align*}
Thus by induction,
\begin{gather}\label{eq:D2 diff vj an varphi}
    \int_{B_{1/32}(0)} |Dv_{j,\ell}^\tau-D\varphi_{j}^\tau|^2\,\mr dx
        \leq C\Big((1+\eta^2\cdot r^2)(\kappa r+\kappa_0r)^2+\kappa_1^2 r^2\Big)\leq C(1+\eta^2)r^2.
\end{gather}
This together with Lemma \ref{lem:diff to sum} (by taking $K$, $\delta_0$, $\eta$, $\kappa_1$ to be $C(1+\eta^2)r^2$, $\delta_0$, $\eta r$, $\kappa_1 r^2$ therein) gives that 
\begin{equation*}
\begin{split}
    \int_{B_{1/64}(0)}|D\varphi_j^\tau|^2\,\mr dx
        &\leq C\Big(\delta_0^2(1+\eta^2r^2)+\eta^2r^2+\kappa_1\delta_0r^2\Big)\\
        &\leq C(1+\eta^2),
        \end{split}
\end{equation*}
where $C=C(n,m,\Lambda,\lambda)$ and we used the fact that $\delta_0,\kappa,\kappa_1,r\leq 1$.
Plugging this into \eqref{eq:|f|+|g|} and noting that $|\varphi_k^\tau|\leq \sup|D\varphi_k|\leq \delta_0<1$, we have
\begin{align*}
    \int_{B_{3/4}(0)}|\oli{\bm f}|^2\zeta^2+|\oli g|^2\zeta^2\,\mr dx
        &\leq C\int_{B_{1/64}(0)}\Big(
        (\kappa+\kappa_0)^2r^2 \big(|D\varphi_k^\tau|^2\zeta^2
        +\sum_{\ell=1}^{i_j} |D\varphi_j^\tau-Dv_{j,\ell}^\tau|^2 \big)+\\
        &+(\kappa+\kappa_0)^2r^2\sum_{\ell=1}^{i_k}|D\varphi_k^\tau-Dv_{k,\ell}^\tau|^2 + (\kappa+\kappa_0+\kappa_1)^2r^2
        \Big)\,\mr dx\\
        &\leq C\Big((1+\eta^2)(\kappa+\kappa_0)^2r^2+\kappa_1^2 r^2\Big).
\end{align*}
Plugging this into \eqref{eq:DPsi2 and f2 g2}, we then have that
\[
    \int_{B_{1/128}(0)}|D\Psi^\tau|^2\,\mr dx
    \leq C\Big((1+\eta^2)(\kappa+\kappa_0)^2r^2+\kappa_1^2 r^2\Big),
\]
which combining with \eqref{eq:D2 diff vj an varphi} yields that 
\[  
    \int_{B_{1/128}(0)}|D v_{j,\ell}^\tau-D v_{k,\ell'}^\tau|^2\,\mr dx
    \leq C\Big((1+\eta^2)(\kappa+\kappa_0)^2r^2+\kappa_1^2 r^2\Big).
\]
Note that $v_j$ is from scaling $u_j$. Hence we conclude that 
\begin{equation}\label{eq:hat r inequality at low density} 
    \int_{B_{\hat r(y)}(y)}|D u_j^\tau-D u_k^\tau|^2\,\mr dx\leq C\Big((1+\eta^2 )(\kappa+\kappa_0)^2+\kappa_1^2 \Big)\hat r^n, \quad \hat r(y):=r(y)/128.
\end{equation}

\medskip
{\noindent\em Step 3: Completing the proof by a covering argument.}

\medskip
So far, we have defined $\hat r(y)$ for all $y\in B_{1/16}(0)\setminus \{u_1=u_m\}$ so that \eqref{eq:hat r inequality at low density} is satisfied. Now we define $\hat r(y)=\tau$ for $y\in B_{1/16}(0)\cap\{u_1=u_m\}$. In this case, $Du_i(y)=Du_j(y)$. Then by \eqref{eq:oscillation}, for $\tau\in(0,1/32)$ and $x\in B_{\tau}(0)$,
\begin{align*}
    |Du_i^\tau(x)-Du_j^\tau(x)|
    &\leq \frac{|Du_i(x+\tau \bm e)-Du_j(x+\tau\bm e)|}{\tau} 
          +\frac{|Du_i(x)-Du_j(x)|}{\tau}\\
    &\leq C\cdot (\kappa+\kappa_0)\tau/\tau
          =C(n,m,\Lambda,\lambda)\cdot (\kappa+\kappa_0). 
\end{align*}
This implies 
\begin{equation}\label{eq:hat r inequality at touching}
    \int_{B_{\hat r(y)}(y)}|Du_i^\tau(x)-Du_j^\tau(x)|^2\,\mr dx
        \leq C\cdot (\kappa+\kappa_0)^2\hat r^n.
\end{equation}
Now that $\{B_{\hat r(y)}(y); y\in B_{1/32}(0)\}$ forms a covering of $B_{1/32}(0)$. Hence one can find a finite cover $\{B_{\hat r(y_k)}(y_k);k=1,\cdots,N\}$ so that $B_{\hat r(y_k)/2}(y_k)$ does not intersect $B_{\hat r(y_i)/2}(y_i)$ for $k\neq i$. Then by \eqref{eq:hat r inequality at low density} and \eqref{eq:hat r inequality at touching},
\begin{align*}
    \int_{B_{1/32}(0)}|Du_i^\tau-Du_j^\tau|^2\,\mr dx
    &\leq \sum_{k=1}^N\int_{B_{\hat r(y_k)}(y_k)}|Du_i^\tau-Du_j^\tau|^2\,\mr dx\\
    &\leq C\sum_{k=1}^N\Big((1+\eta^2)(\kappa+\kappa_0)^2+\kappa_1^2 \Big)[\hat r(y_k)]^n\\
    &\leq C\cdot \Big((1+\eta^2 )(\kappa+\kappa_0)^2+\kappa_1^2 \Big).
\end{align*}
This completes the proof Lemma \ref{lem:diff 22}.
\end{proof}

So far, we have proved the hessian estimates of the height difference. This combining with Lemma \ref{lem:diff to sum} will give the hessian estimates for the average, which yields the integral bound of $D^2u_j$.

Recall that 
\[    
    \bm\kappa^*:=\kappa^2+\kappa_0^2+\kappa_1^2+\delta_0^2,\]
and 
\[   \eta:=\Big(\|\partial_z\partial_x\ms A\|_{C^0(\bm U_1)}+\|\partial_x\ms B\|_{C^1(\bm U_1)}\Big)\cdot \sum_{j=1}^m\|u_j\|_{C^0(B_{1}(0))}.\]
\begin{proposition}\label{prop:ui is W22}
Let $\{u_j\}_{j=1}^m$ be $C^{1,\alpha}$ ordered graphs over $B_1(0)\subset \mb R^n$ which is a solution to $(\mc Q,\{g_j\})$. 
Suppose that \eqref{eq:Lambda and lambda}, \eqref{eq:ms AB norm}, \eqref{eq:ms AB at x00}, \eqref{eq:small C1alpha} and \eqref{eq:eta kappa0 and kappa1} are satisfied. 
Then
\[
    \sum_{j=1}^m \int_{B_{3/4}(0)}|D^2u_j|^2\,\mr dx
        \leq C(n,m,\Lambda,\lambda)(\bm\kappa^*+\eta^2).
  \]
	\end{proposition}
\begin{proof}
For every $y\in B_{3/4}(0)$, it follows that $B_{1/4}(y)\subset B_1(0)$. 
By Lemma \ref{lem:large gap}, there exists 
$r\in [16^{-m},1]$ 
so that $\{u_j\}$ has the same number of connected components in 
$B_{r/4}(y)$ and $B_{r/64}(y)$. 
Denote by 
$\{v_j\}_{j=1}^{\mk m}$ 
be one of the connected components. 
Note that $\kappa,\kappa_0,\kappa_1\leq 1$.
Then by Lemma \ref{lem:diff 22} in $B_{r/4}(y)$ for $\{v_j\}$, 
we have that for $i,j=1,\cdots,\mk m$,
\begin{equation}\label{eq:D2 diff at r/64 center y}
    \int_{B_{r/128}(y)}|D^2v_i-D^2v_j|^2\,\mr dx
         \leq C \Big((1+\eta^2)(\kappa+\kappa_0)^2+\kappa_1^2\Big)\leq C(\bm\kappa^*+\eta^2).
\end{equation}
Then applying Lemma \ref{lem:diff to sum} in $B_{r/128}(y)$, we obtain
\begin{equation*}
\begin{split}
    \int_{B_{r/256}(y)}|D^2\varphi|^2\,\mr dx
    &\leq C\Big(\delta_0^2\cdot C(\bm\kappa^*+\eta^2+1)+\eta^2+\kappa_1\delta_0\Big)\leq C(\bm\kappa^*+\eta^2).
\end{split}
\end{equation*}
where $C=C(n,m,\Lambda,\lambda)$ and 
$\varphi=\frac{1}{\mk m}\sum_{j=1}^{\mk m} v_j$. 
This together with \eqref{eq:D2 diff at r/64 center y} gives that for $j=1,\cdots,\mk m$,
\[
    \int_{B_{r/256}(y)}|D^2v_j|^2\,\mr dx
         \leq C(\bm\kappa^*+\eta^2).
\]
In particular, letting $r_m:=16^{-m-2}$, by the arbitrariness of $v_j$, we conclude that for $j=1,\cdots, m$,
\[  
    \int_{B_{r_m}(y)}|D^2u_j|^2\,\mr dx
         \leq C(\bm\kappa^*+\eta^2).
         \]
By taking a finite cover of $B_{3/4}(0)$ with small balls with radius $r_m$, then Proposition \ref{prop:ui is W22} is proved.         
\end{proof}

%%%%%%%%%%%%%%%%%%%%%%%%%%%%%%%%%%%%%%%%%%%%%%%
%Section Order estimates
%%%%%%%%%%%%%%%%%%%%%%%%%%%%%%%%%%%%%%%%%%%%%%%
\section{Order estimates of the hessian}\label{sec:order of hessian}

This section is devoted to prove a monotonicity formula for the average of $|D^2u_j|^2$ over small balls, which implies that $\int_{B_r(y)}|D^2u_j|^2\,\mr dx$ has order $n$. Since we have proved that each $u_j$ is $W^{2,2}$, then we are able to consider the divergence form equation for $Du_j$, $j=1,\cdots,m$. We will follow the strategy in proving H\"older regularity for gradients of solutions of divergence form, where one should approximate the $D_ku$ ($k=1,\cdots,n$) by a harmonic function; see \cite{Han-Lin-PDE}*{\S 3.4}. The main difficulty here is that the connectedness of $\{u_j\}$ may be not preserved and then there are no height difference estimates in Section \ref{sec:height difference}.

To overcome it, we will use an inductive method. Note that at the radius where $\{u_j\}$ is not connected, the height difference estimates are still valid. By the subsystem condition, each connected components will satisfy a partial differential equation. Furthermore, the difference of the average of connected component will satisfy a weak solution. This together with the inequality from the induction will give the desired inequality.

In Section \ref{subsec:primary order}, We will first prove the order $n-\epsilon$ for any $\epsilon>0$. The idea is to approximate the average by a harmonic function which will help us to give the order of the average. Combining with the height difference estimates, we will obtain the order for each $u_j$.

In Section \ref{subsec:improved order}, we will prove the improved order estimates for the graph of $u_j$ over a tilt-plane. Note that the inductive method can not be proceed since the optimal order is $n$ instead of $n+\alpha$ in the minimal surface equation. Here we will consider the order of the average and then use the inductive method to estimate 
\begin{gather*}
      \bm\sigma^*(r):= \sum_{j=1}^m\int_{B_r(0)}|D^2u_j- 
                      (D^2u_j)_{B_r(0)}|^2\,\mr dx.
       \end{gather*}

Finally, in Section \ref{subsec:monotonicity of second derivative}, we will give a monotonicity formula of $r^{-n}\bm\sigma^*(r)$.

\medskip
In this section, we will use the following classical results. Note that here we only require the inequalities for $r>r_0$.
\begin{lemma}[\citelist{\cite{Han-Lin-PDE}*{Lemma 3.4}\cite{Jost13}*{Lemma 14.4.16}}]\label{lem:order}
Let $\gamma>1$, $R_0>0$, $\beta\in(0,1/2)$, $\mu>2\beta$, $\kappa>0$ and $\sigma(r)$ be a nonnegative, monotonically increasing function on $[0, R_0]$ satisfying
\[ 
    \sigma(r)\leq \gamma\Big(\Big(\frac{r}{R}\Big)^\mu+\delta\Big)\sigma(R) 
                  +\kappa R^{\mu-2\beta}   \]
for all $0<r_0 < r\leq R \leq  R_0$ with 
$\delta<\delta_1(\gamma,\mu,\beta)=(2\gamma)^{-\frac{\mu}{\beta}}$. 
We then have
\[ 
    \sigma(r)\leq \gamma_1\Big(\frac{r}{R}\Big)^{\mu-\beta}\sigma(R) 
                  +\gamma_2\kappa r^{\mu-2\beta}   \]
with $\gamma_1,\gamma_2$ depending on $\gamma, \mu, \beta$.
\end{lemma}
\begin{proof}
Let 
\[  
    \tau:=(2\gamma)^{-\frac{1}{\beta}}. \]
Then $\delta<\delta_1=\tau^\mu$. By assumption,
\[   
    \sigma(\tau R)\leq 2\gamma \tau^\mu\sigma(R)+\kappa R^{\mu-2\beta}
        =\tau^{\mu-\beta}\sigma(R)+\kappa R^{\mu-2\beta}.
\]	
Since $r<R$, then there exists an integer $k\geq 0$ so that 
\[  \tau^{k+1}R<r\leq \tau^k R.   \]
Without loss of generality, we assume that $k\geq 1$. Then we have that 
\begin{align*}
	 \sigma(r)\leq \sigma(\tau^kR)&\leq \tau^{(\mu- 
             \beta)k}\sigma(R)+\sum_{j=0}^{k-1}\kappa (\tau^{k-1-j}R)^{\mu-2\beta}\tau^{(\mu-\beta)j}\\
	 &\leq \tau^{\beta-\mu}\cdot \Big(\frac{r}{R}\Big)^{\mu- 
             \beta}\sigma(R)+\kappa\tau^{(k-1)(\mu-2\beta)}R^{\mu-2\beta}(1-\tau^{\beta})^{-1}\\
	 &\leq (2\gamma)^{\frac{\mu-\beta}{\beta}}\Big(\frac{r}{R}\Big)^{\mu- 
             \beta}\sigma(R)+\kappa\cdot (1-(2\gamma)^{-1})^{-1}\cdot (2\gamma)^{\frac{2(\mu-2\beta)}{\beta}}\cdot r^{\mu-2\beta}.	 
	 \end{align*}
Then the lemma follows by taking
\[  
      \gamma_1= (2\gamma)^{\frac{\mu-\beta}{\beta}};\quad \gamma_2= (1-(2\gamma)^{-1})^{-1}\cdot (2\gamma)^{\frac{2(\mu-2\beta)}{\beta}}. 
\]
	\end{proof}
 
%%%%%%%%%%%%%%%%%%%%%%%%%%%%%%%%%%%%%%%%%%%%%%%%%%%%
\subsection{A primary order estimate}\label{subsec:primary order}
Let $\{u_j\}$ be $C^{1,\alpha}$ ordered graphs on $B_1(0)$ which is a solution to $(\mc Q,\{g_j\})$. 
Suppose that $\{u_j\}$ is connected in $B_s(0)$ for some $s\in (0,1/16]$. Then by Lemma \ref{lem:osc of diff} and Lemma \ref{lem:gradient diff}, there exists $C=C(n,m,\Lambda,\lambda)$ for $x\in B_{1/16}(0)\setminus B_s(0)$,
\begin{equation}\label{eq:diff order two}
      |u_i(x)-u_j(x)|\leq C\cdot \kappa|x|^2,\quad  |Du_i(x)-Du_j(x)|\leq C\cdot (\kappa+\kappa_0) |x|.       
\end{equation}
We now derive a local integral estimate for the average of $\{u_j\}$. This process is similar to the H\"older estimates \cite{Han-Lin-PDE}*{Theorem 3.8} by approximating with a harmonic function. However, we have more error terms on the right hand side. To overcome it, we will write the equation of $D_k\varphi$ as a form of \eqref{eq:Dk varphi has good error term} and use \eqref{eq:diff order two} to control all of the error terms.

Recall that 
\[    
    \bm\kappa^*:= \kappa^2+\kappa_0^2+\kappa_1^2+\delta_0^2,\]
and 
\[   \eta:=\Big(\|\partial_z\partial_x\ms A\|_{C^0(\bm U_1)}+\|\partial_x\ms B\|_{C^1(\bm U_1)}\Big)\cdot \sum_{j=1}^m\|u_j\|_{C^0(B_{1}(0))}.\]
For simplicity, let 
\[  
     \epsilon:=\min\left\{\frac{1}{2}\alpha,\frac{1}{2}(1-\alpha)\right\}. \]

\begin{lemma}\label{lem:n-epsilon}
Let $\{u_j\}$ be $C^{1,\alpha}$ ordered graphs on $B_1(0)$ which is a solution to $(\mc Q,\{g_j\})$. Suppose that \eqref{eq:Lambda and lambda}, \eqref{eq:ms AB norm},  \eqref{eq:ms AB at x00}, \eqref{eq:small C1alpha} and \eqref{eq:eta kappa0 and kappa1} are satisfied. 
If  $\{u_j\}_{j=1}^m$ is connected in $B_s(0)$, with $s\leq r< R$, 
then we have that 
\begin{equation*}%\label{eq:n-epsilon oder} 
      \int_{B_r(0)}|D^2\varphi|^2\,\mr dx\leq C\Big(\frac{r}{ R}\Big)^{n-\epsilon}\int_{B_{ R}(0)}|D^2 \varphi|^2\,\mr dx+C(\bm \kappa^*+\eta^2) r^{n-2\epsilon},  
\end{equation*}
where $\varphi =\frac{1}{m}\sum_{j=1}^mu_j$ and $C=C(n,m,\Lambda,\lambda,\alpha)$.
	\end{lemma}
\begin{proof}[Proof of Lemma \ref{lem:n-epsilon}]
\begin{comment}
If $s\geq 1/64$, then $R\geq r\geq 1/64$ and $R/r\leq 64$. It follows that 
\begin{equation*}
\int_{B_r(0)}|D^2\varphi|^2\,\mr dx\leq C\cdot 64^n\Big(\frac{r}{ R}\Big)^{n-\epsilon}\int_{B_{ R}(0)}|D^2 \varphi|^2\,\mr dx.
\end{equation*}
This is the desired inequlity.
\end{comment}
We first assume that $r\leq R\leq 1/64$. 
By Remark \ref{prop:ui is W22}, $u_j\in W^{2,2}_{loc}(B_1(0))$. Then by \eqref{eq:pde of Dk varphi}, 
\begin{equation}\label{eq:diff sum}
      -\dv\Big(\bm A^*DD_k\varphi+{\bm b^*}D_k\varphi\Big)+\bm c^*DD_k\varphi +d^* D_k\varphi=-\dv\Big(\hat{\bm f}-m\partial_{x_k}\ms A(x,\varphi,D\varphi)\Big)+\hat g,
\end{equation}
where $\hat{\bm f}$ and $\hat g$ are defined in \eqref{eq:fd} and \eqref{eq:gd}; and 
\begin{gather*}
      \bm A^*(x):=\sum_{j=1}^m\partial_{\bm p} \ms A(x,u_j(x),Du_j(x)); 
               \quad \bm b^*(x):=\sum_{j=1}^m\partial_{z} \ms A(x,u_j(x),Du_j(x));\\
      \bm c^*(x):=\sum_{j=1}^m\partial_{\bm p} \ms B(x,u_j(x),Du_j(x)); 
               \quad  d^*(x):=\sum_{j=1}^m\partial_{z} \ms B(x,u_j(x),Du_j(x)).
\end{gather*}
We divide the proof into four steps.

\medskip
{\noindent\bf Step I:} {\em Estimating the coefficients in \eqref{eq:diff sum}.}

\medskip
Observe that by \eqref{eq:small C1alpha}, for all $x,y\in B_R(0)$,
\begin{equation}\label{eq:ujx-ujy}
    |u_j(x)-u_j(y)|\leq 2\delta_0 R, \quad |Du_j(x)-Du_j(y)|\leq 2\delta_0R^\alpha.
\end{equation}
By \eqref{eq:ms AB norm} and the definition of $\bm A^*,\bm b^*,\bm c^*, d^*$, 
\begin{equation}\label{eq:bound A* b* c* d*}
        |\bm A^*|+|\bm b^*|+|\bm c^*|+|d^*|\leq C(m,\Lambda). 
\end{equation}
By the definition of $\bm A^*$, we have that 
\begin{align*}
\begin{split}
    |\bm A^*(x)-\bm A^*(y)|\leq \sum_{j=1}^m\Big(\|\partial_x\partial_{\bm p}\ms A\|_{C^0}|x-y|
         +\|\partial_z\partial_{\bm p}\ms A\|_{C^0}|u_j(x)-u_j(y)|+\\
         \|\partial_{\bm p}\partial_{\bm p}\ms A\|_{C^0}|Du_j(x)-Du_j(y)|\Big).
         \end{split}
\end{align*}
Together with \eqref{eq:ujx-ujy}, we have that 
\begin{equation}\label{eq:bound Ax-Ay}
     \sup_{x,y\in B_R(0)} |\bm A^*(x)-\bm A^*(y)|\leq C(m,\Lambda)(R+\delta_0R+\delta_0R^\alpha)  \leq  C(m,\Lambda)( R+\delta_0R^\alpha)  . 
\end{equation}   
A similar argument gives that 
\begin{equation}\label{eq:bound bx-by}
    \sup_{x,y\in B_R(0)} \Big(|\bm b^*(x)-\bm b^*(y)|+|\bm c^*(x)-\bm c^*(y)|+|d^*(x)-d^*(y)|\Big)\leq C(m,\Lambda)( R+\delta_0R^\alpha) .  
\end{equation}
On the other hand, plugging \eqref{eq:ujx-ujy} and \eqref{eq:diff order two} into Lemma \ref{lem:bound hat f and hat g}, we have that 
\begin{gather}
\label{eq:integral estimate of hatf}
\begin{split}
    \int_{B_R(0)}|\hat{\bm f}|^2\,\mr dx
    \leq C\sum_{j=1}^m\int_{B_R(0)} 
        (\kappa+\kappa_0)^2R^2\Big(|DD_k\varphi-DD_ku_j|^2+1\Big)\,\mr dx;
    \end{split}\\
     \label{eq:integral estimate of hatg}
    \begin{split}
    \int_{B_R(0)}|\hat{g}|\,\mr dx
    \leq C\sum_{j=1}^m\int_{B_R(0)} 
        (\kappa+\kappa_0)R\big(|DD_k\varphi-DD_ku_j|+(\kappa+\kappa_0)R\big)+\eta+\delta_0+\kappa_1
    \,\mr dx.
    \end{split}
\end{gather}
Recall that by Lemma \ref{lem:diff 22}, for $1\leq i\leq j\leq m$,
\begin{equation*}%\label{eq:from lem of hessian diff} 
    \int_{B_R(0)}|D^2u_i-D^2u_j|^2\,\mr dx\leq C\cdot \Big((1+\eta^2)(\kappa+\kappa_0)^2+\kappa_1^2\Big) R^n, 
\end{equation*}
which implies that 
\[
      \int_{B_R(0)}|D^2\varphi-D^2u_j|^2\,\mr dx\leq C\cdot \Big((1+\eta^2)(\kappa+\kappa_0)^2+\kappa_1^2\Big) R^n.
\]
Plugging the above two inequalities into \eqref{eq:integral estimate of hatf} and \eqref{eq:integral estimate of hatg}, we have 
\begin{gather}
    \label{eq:bound of L2 hat f}
    \int_{B_R(0)}|\hat{\bm f}|^2\,\mr dx
    \leq C(1+\eta^2) \bm \kappa^*R^{n+2},\\
    \label{eq:bound of L1 hat g}
    \int_{B_R(0)}|\hat g|\,\mr dx \leq C 
         (\sqrt{\bm\kappa^*}+\eta\bm\kappa^*) R^n ,
\end{gather}
where $C=C(n,m,\Lambda,\lambda)$ and we used that $\kappa,\kappa_0,\kappa_1,\delta_0\leq 1$. 

By the definition of $\varphi$ and \eqref{eq:ujx-ujy}, for $x,y\in B_R(0)$,
\begin{equation}\label{eq:Dkvarphi-Dkvarphi0}
    |D_k\varphi|\leq \delta_0,\quad   |D_k\varphi(x)-D_k\varphi(y)|\leq \delta_0(2R)^{\alpha}.
\end{equation} 
Plugging it into Lemma \ref{lem:bound hat f and hat g}, we obtain that for all $\phi\in C_c^1(B_R(0))$,
\begin{equation}
    \label{eq:partial xk A integral bound}
    \Big|\int_{B_R(0)} \partial_{x_k}\ms A(x,\varphi,D\varphi) D\phi\,\mr dx\Big|\leq C\int_{B_R(0)}(\eta+\delta_0)|\phi|+\delta_0R^\alpha|D\phi|+R^\alpha|D^2\varphi||\phi|\,\mr dx.
\end{equation}
\qed

\medskip
{\noindent\bf Step II:} {\em Constructing  a harmonic function to approximate $D_k\varphi$.}

\medskip
Let $ f$ be the weak solution of 
\begin{equation}\label{eq:def of f by A*}
      -\dv\Big(\bm A^*(0)Df\Big)=0, \quad f|_{\partial B_R(0)}=D_k\varphi-D_k\varphi(0). 
\end{equation} 
Let 
\[
       \phi:=D_k\varphi-D_k\varphi(0)-f.
\]
Obviously, $\phi\in W_0^{1,2}(B_R(0))$.
Then we claim that 
\begin{equation}\label{eq:Df and DDvarphi}
	 \int_{B_R(0)}|Df|^2\,\mr dx\leq C(m,\lambda,\Lambda)\int_{B_R(0)}|DD_k\varphi|^2\,\mr dx.   
	  \end{equation}
Indeed, by taking $\phi$ as a test function in \eqref{eq:def of f by A*}, we have
\begin{align*}
      \int_{B_R(0)}|Df|^2\,\mr dx&\leq C(\lambda)\int_{B_R(0)}|\bm A^*(0)||DD_k\varphi| |Df|\,\mr dx \\
      &\leq C(\lambda)\int_{B_R(0)}|\bm A^*(0)|^2|DD_k\varphi|^2\,\mr dx+\frac{1}{2}\int_{B_R(0)}|Df|^2\,\mr dx,
	\end{align*}
which implies 
\begin{equation}\label{eq:Df L2 bound by DDk varphi and phi}
       \int_{B_R(0)}|Df|^2 \,\mr dx\leq C(\lambda)\int_{B_R(0)}|\bm A^*(0)|^2|DD_k\varphi|^2\,\mr dx.
\end{equation}
Plugging \eqref{eq:bound A* b* c* d*} into \eqref{eq:Df L2 bound by DDk varphi and phi}, we then obtain \eqref{eq:Df and DDvarphi}.
By \eqref{eq:Dkvarphi-Dkvarphi0} and the maximum principle for $f$, 
\[ 
      |f(x)|\leq \delta_0(2R)^{\alpha}.  
                                 \]
It follows that 
\begin{equation}\label{eq:osc of phi} 
|\phi|\leq |D_k\varphi-D_k\varphi(0)|+|f| \leq 2^{1+\alpha}\delta_0R^\alpha.  
\end{equation}
\qed

\medskip
{\noindent\bf Step III:} {\em Bounding $L^2$ norm of $\phi$ by $\varphi$.}

\medskip
Note that 
\[
    \dv\big(\bm  b^*(0)D_k\varphi(0)\big)=0, \quad \dv\big(\bm  c^*(0)D_k\varphi(0)\big)=0.
    \]
This together with the definition of $f$ in \eqref{eq:def of f by A*} and \eqref{eq:diff sum} gives that
\begin{equation}\label{eq:Dk varphi has good error term}
\begin{split} 
    &\ \ \ -\dv\Big(\bm A^*D(D_k\varphi-f) + {\bm b}^* (D_k\varphi-D_k\varphi(0)) - \bm c^*(0)(D_k\varphi- D_k\varphi(0))\Big) +d^*  D_k\varphi\\
    &=-\dv\Big((\bm A^*(0)-\bm A^*)Df + (\bm b^*(0)-\bm b^*)D_k\varphi(0) + \hat {\bm f}-m\partial_{x_k} \ms A(x,\varphi,D\varphi) \Big)\\
    & \ \ \ -(\bm c^*-\bm c^*(0))DD_k\varphi+\hat g.
\end{split}
\end{equation}
Then by taking $\phi$ as a test function in \eqref{eq:Dk varphi has good error term} and using \eqref{eq:bound A* b* c* d*}, \eqref{eq:bound Ax-Ay}, \eqref{eq:bound bx-by}, \eqref{eq:Dkvarphi-Dkvarphi0}, \eqref{eq:partial xk A integral bound} and \eqref{eq:osc of phi}, we obtain
\begin{align*}
    &\ \ \ \int_{B_R(0)}|D\phi|^2\,\mr dx\\
    &\leq C(n,m,\Lambda,\lambda)\int_{B_R(0)}\Big((|\bm 
         b^*|+|\bm c^*(0)|)|D_k\varphi-D_k\varphi(0)| |D\phi|+|d^*||D_k\varphi||\phi|+\\
    &\ \ \ +|\bm A^*(0)-\bm A^*||Df||D\phi|+|\bm b^*(0)- 
        \bm b^*||D_k\varphi(0)||D\phi|+ |\bm c^*-\bm c^* 
         (0)||DD_k\varphi||\phi| +\\
    &\ \ \ +|\hat{\bm f}||D\phi|+|\hat g||\phi|\Big)\,\mr dx + m\Big|\int_{B_R(0)}\partial_{x_k} \ms A(x,\varphi,D\varphi)D\phi\,\mr dx\Big|\\
            &\leq C\int_{B_R(0)}\Big(\delta_0 R^\alpha 
         |D\phi|+ \delta_0\cdot \delta_0R^\alpha +(R+\delta_0 R^\alpha)|Df||D\phi|+ (R+ \delta_0R^\alpha) 
         \cdot\delta_0|D\phi|+\\
                &\hspace{1em}+  (R+\delta_0R^\alpha) 
         \delta_0R^{\alpha}|DD_k\varphi|+|\hat {\bm f}||D\phi|+|\hat g|\delta_0R^\alpha +(\delta_0+\eta)\delta_0R^\alpha+R^\alpha|D^2\varphi|\delta_0R^\alpha\Big)\,\mr dx \\
            &\leq \frac{1}{4}\int_{B_R(0)}|D\phi|^2\,\mr dx + C\int_{B_R(0)}\Big(\delta_0^2R^{\alpha}+R^{2\alpha}(|Df|^2+|D^2\varphi|^2) + |\hat{\bm f}|^2+|\hat 
         g|\delta_0R^\alpha+\delta_0\eta R^\alpha\Big)\,\mr dx.
%|R^\alpha|DD_k\varphi|\cdot |D\phi|+\\
%&\hspace{11em}+R\sum_{j=1}^m|DD_k\varphi-DD_ku_j|(|D\phi|+|\phi|)+R^2|D\phi|+R^2|\phi|+\delta_0|\phi|\,\mr dx.  
	 \end{align*}
This together with \eqref{eq:Df and DDvarphi} implies 
\begin{equation}\label{eq:Dphi2 upper bound}
       \int_{B_R(0)}|D\phi|^2\,\mr dx\leq C\int_{B_R(0)}\Big((\delta_0^2+\delta_0\eta)R^{\alpha}+R^{2\alpha}|D^2\varphi|^2+|\hat{\bm f}|^2+|\hat g|\delta_0R^\alpha\Big)\,\mr dx.
\end{equation}
Plugging \eqref{eq:bound of L2 hat f} and \eqref{eq:bound of L1 hat g} into \eqref{eq:Dphi2 upper bound}, we conclude that 
\begin{equation}\label{eq:Dvarphi-f}
    \int_{B_R(0)}|DD_k\varphi-Df|^2\,\mr dx
        \leq C \int_{B_R(0)}R^{2\alpha}|D^2\varphi|^2 \,\mr dx+C(\bm \kappa^*+\eta^2)R^{n+\alpha}.
\end{equation}
By harmoniticity of $f$ \citelist{\cite{Han-Lin-PDE}*{Lemma 3.10}\cite{Jost13}*{Lemma 14.4.5}}, there exists a constant $A(n)$ such that 
\[ 
        \int_{B_r(0)}|D f|^2\,\mr dx\leq A(n)\Big(\frac{r}{R}\Big)^n\int_{B_R(0)}|Df|^2\,\mr dx.   \]
This together with \eqref{eq:Dvarphi-f} and \eqref{eq:Df and DDvarphi} gives
\begin{align*}
    \int_{B_r(0)}|DD_k\varphi|^2
          &\leq \int_{B_r(0)}2|Df|^2+2|DD_k\varphi-Df|^2\,\mr dx\\
          &\leq A'\Big(\Big(\frac{r}{R}\Big)^n+ 
      R^{2\alpha}\Big)\int_{B_R(0)} |D^2\varphi|^2dx+C(\bm \kappa^*+\eta^2)R^{n+\alpha},
	\end{align*}
where $A'=A'(n,m,\Lambda,\lambda,\alpha)$ and we now fix such a constant. 
\qed

\medskip
{\noindent\bf Step IV:} {\em Proving the proposition case by case.}

\medskip
By the arbitrariness of $k$, we conclude that for all $r<R\leq 1/64$,
\begin{equation}
    \int_{B_r(0)}|D^2\varphi|^2
    \leq nA'\Big(\Big(\frac{r}{R}\Big)^n+
      R^{2\alpha}\Big)\int_{B_R(0)} |D^2\varphi|^2\,\mr dx+C(\bm \kappa^*+\eta^2)R^{n+\alpha}.
      \end{equation}
Then Lemma \ref{lem:order} applies to give a positive constant $R_0=R_0(n,A',\alpha)<1/64$ with
\[  R_0^{2\alpha} <(2nA)^{-\frac{n}{\epsilon}} \]
and the following possibilities:
\begin{itemize}
    \item {\em If $s<r\leq R\leq R_0$}, then by Lemma \ref{lem:order},
\[
      \int_{B_r(0)}|D^2\varphi|^2\,\mr dx\leq C\Big(\frac{r}{R}\Big)^{n-\epsilon} \int_{B_{R}(0)}|D^2\varphi|^2\,\mr dx+C(\bm \kappa^*+\eta^2)r^{n-2\epsilon},     \]
where $C$ is depending only on $n, m,\Lambda,\lambda,\alpha$. 
\item {\em If $r>R_0$}, then the inequality is trivial because $R_0$ depends only on $n,m,\Lambda,\lambda,\alpha$ and $R/r\leq 1/R_0$. Indeed, it follows that
\begin{equation*}
\int_{B_r(0)}|D^2\varphi|^2\,\mr dx\leq C\cdot R_0^{-n}\Big(\frac{r}{ R}\Big)^{n-\epsilon}\int_{B_{ R}(0)}|D^2 \varphi|^2\,\mr dx.
\end{equation*}
\item {If $r<R_0<R$}, then by the first case, 
\begin{align*}  
      \int_{B_r(0)}|D^2\varphi|^2\,\mr dx
          &\leq C\Big(\frac{r}{R_0}\Big)^{n-\epsilon} \int_{B_{R_0} 
       (0)}|D^2\varphi|^2\,\mr dx+C(\bm \kappa^*+\eta^2)r^{n-2\epsilon}\\
          &\leq CR_0^{-n+\epsilon}\Big(\frac{r}{R}\Big)^{n-\epsilon} \int_{B_{R} (0)} |D^2\varphi|^2\,\mr dx+ C(\bm \kappa^*+\eta^2)r^{n-2\epsilon}.
\end{align*}
\end{itemize}
This finishes the proof of Lemma \ref{lem:n-epsilon}.
	\end{proof}
\begin{remark}
To prove Lemma \ref{lem:n-epsilon}, one can replace \eqref{eq:Dvarphi-f} by a weaker inequality, e.g.
\[
\int_{B_R(0)}|DD_k\varphi-Df|^2\,\mr dx
        \leq C \int_{B_R(0)}R^{\alpha}|D^2\varphi|^2 \,\mr dx+C(\bm \kappa^*+\eta^2)R^{n},
\]
which can be derived by a much easier way without \eqref{eq:Dk varphi has good error term}. However, the order on the right hand side of \eqref{eq:Dvarphi-f} will be essentially used in Lemma \ref{lem:D2 integral beta}.
\end{remark}

Combining Lemma \ref{lem:n-epsilon} and Lemma \ref{lem:diff 22}, we obtain the primary order estimates for all $r<R\leq t$ even if $\{u_j\}$ is disconnected in $B_R(0)$.
\begin{proposition}\label{lem:n-epsilon order for uj}
Let $t\leq 1$. Let $\{u_j\}$ be $C^{1,\alpha}$ ordered graphs on $B_t(0)$ which is a solution to $(\mc Q,\{g_j\})$. 
Suppose that \eqref{eq:Lambda and lambda}, \eqref{eq:ms AB norm}, \eqref{eq:ms AB at x00}, \eqref{eq:small C1alpha} and \eqref{eq:eta kappa0 and kappa1} are satisfied. 
There exists a constant 
$C=C(n,m,\Lambda,\lambda,\alpha)$
so that for all $r<R\leq t$, 
\[  
    \sum_{j=1}^{m} \int_{B_r(0)}|D^2u_j|^2
        \leq C \Big( \frac{r}{R}\Big )^{n-\epsilon} \sum_{j=1}^m\int_{B_R(0)} |D^2u_j|^2\,\mr dx 
        +C(\bm \kappa^*+\eta^2)r^{n-2\epsilon}.  \]		   
	\end{proposition}
\begin{proof}
%Without loss of generality, we assume that $R_0(n,m,\alpha)\leq R_0(n,j,\alpha)$ for all $j\leq m$. Now 
%Without loss of generality, we assume that $t\leq 1/16$ so that Lemma \ref{lem:diff 22} applies. 
So long as $r\geq t/16$, then the inequality is trivial. Now we assume that $r\leq t/16$.
We are going to prove the lemma inductively on $m$. 
When $m=1$, then the inequality follows from Lemma \ref{lem:n-epsilon}. 
Suppose that the lemma is true for the number of sheets less than or equal to $m-1$. 
Let $\{u_j\}_{j=1}^{m}$ be a solution to $(\mc Q,\{g_j\})$ and is connected in $B_t(0)$. 
Now set
\[ 
    s:=\inf\{r>0; \{u_j \}_{j=1}^{m} \text{ is connected in $B_r(0)$} \}. \]

\noindent{\em Case 1: $r<R\leq s$.} Then $\{u_j\}_{j=1}^{m}$ can be divided into connected components in $B_s(0)$. By induction, each connected component satisfies the inequality in the proposition. Then the desired result follows by taking the sum.
\begin{comment}
$\{u_j\}_{j=1}^i$ and $\{u_j\}_{j=i+1}^{m_0+1}$ on $B_{R_0}(0)$. Then by induction, for all $r<R\leq R_0$, 
\begin{gather*}
	 \sum_{j=1}^i\int_{B_r(0)}|D^2u_j|^2\,\mr dx\leq \mc B(n,i,\alpha) \Big(\frac{r}{R}\Big)^{n-\epsilon}\int_{B_R(0)}\sum_{j=1}^i|D^2u_j|^2\,\mr dx+ \mc Br^{n-2\epsilon};\\
	 \sum_{j=i+1}^{m_0+1}\int_{B_r(0)}|D^2u_j|^2\,\mr dx\leq \mc B(n,m_0+1-i,\alpha) \Big(\frac{r}{R}\Big)^{n-\epsilon}\int_{B_R(0)}\sum_{j=i+1}^{m_0+1}|D^2u_j|^2\,\mr dx+ \mc Br^{n-2\epsilon}. 
	\end{gather*} 
\end{comment}

	\medskip
\noindent{\em Case 2: $s\leq r< R$.} Then by Lemma \ref{lem:n-epsilon}, 
\[  
      \int_{B_r(0)}|D^2\varphi|^2\,\mr dx\leq C\Big(\frac{r}{ R}\Big)^{n-\epsilon}\int_{B_{ R}(0)}|D^2 \varphi|^2\,\mr dx+C(\bm \kappa^*+\eta^2)r^{n-2\epsilon},  
\]
where $C=C(n,m,\Lambda,\lambda,\alpha)$.
Since $r\leq t/16$, then Lemma \ref{lem:diff 22} applies to give that for $1\leq i,j\leq m$, 
\[ 
     \int_{B_r(0)} |D^2u_i-D^2u_j|^2\,\mr dx\leq 
     C\cdot \Big((1+\eta^2)(\kappa+\kappa_0)^2+\kappa_1^2\Big)r^n\leq C(\bm \kappa^*+\eta^2)r^{n}.    \]	
Combining the above two inequalities, we obtain
\[  
     \sum_{j=1}^m \int_{B_r(0)}|D^2u_j|^2\,\mr dx\leq  C\Big(\frac{r}{ R}\Big)^{n-\epsilon}\int_{B_{ R}(0)}\sum_{j=1}^{m}|D^2 u_j|^2\,\mr dx+C(\bm \kappa^*+\eta^2)r^{n-2\epsilon},  
\]	
which is the desired inequality.
%We remark that as $r\to s$, the above inequality holds true.

\medskip
{\noindent\em Case 3: $r\leq s\leq R$}. By Case 1, there exists $C=C(n,m,\Lambda,\lambda,\alpha)$ such that 
\[ 
     \sum_{j=1}^{m}\int_{B_r(0)}|D^2u_j|^2\,\mr dx\leq  C\Big(\frac{r}{s}\Big)^{n-\epsilon}\int_{B_s(0)}\sum_{j=1}^{m}|D^2u_j|^2\,\mr dx+ C(\bm \kappa^*+\eta^2)r^{n-2\epsilon}.  \]
Then by Case 2,
\[  
     \sum_{j=1}^{m}\int_{B_s(0)}|D^2u_j|^2\,\mr dx\leq C\Big(\frac{s}{ R}\Big)^{n-\epsilon}\int_{B_{ R}(0)}\sum_{j=1}^{m}|D^2 u_j|^2\,\mr dx+ C(\bm \kappa^*+\eta^2)s^{n-2\epsilon}.  \]
The above two inequalities give that 
\[
     \sum_{j=1}^{m} \int_{B_r(0)}|D^2u_j|^2\,\mr dx
     \leq  C\Big(\frac{r}{R}\Big)^{n-\epsilon} \int_{B_R(0)}\sum_{j=1}^{m}|D^2u_j|^2\,\mr dx + C(\bm \kappa^*+\eta^2 )r^{n-2\epsilon}.  \]	
%It remains to consider the case of $r<R\leq s<R_0$. Note that $\{u_j\}_{j=1}^{m_0+1}$ can be divided into two disjoint $\mk H$-stationary graphs $\{u_j\}_{j=1}^i$ and $\{u_j\}_{j=i+1}^{m_0+1}$ on $B_{s}(0)$. Then by induction, for all $r<R\leq s$,
%\[ \int_{B_r(0)}|D^2u_j|^2\,\mr dx\leq  \mc B(n,m_0,\alpha) \Big(\frac{r}{R}\Big)^{n-\epsilon}\int_{B_R(0)}\sum_{j=1}^{m_0+1}|D^2u_j|^2\,\mr dx+ \mc Br^{n-2\epsilon}.  \]
This completes the proof of Proposition \ref{lem:n-epsilon order for uj}.
	\end{proof}

%%%%%%%%%%%%%%%%%%%%%%%%%%%%%%%%%%%%%%%%%%%%%
\subsection{Improved order estimates}\label{subsec:improved order}

Let $u\in W^{1,2}(B_1(0))$. 
Then the function defined by 
\[X(\in \mb R^{n})\longmapsto \int_{B_1(0)}|Du-X|^2\,\mr dx\]
is convex and has a minimum at $X=(D^2u)_{B_1(0)}$; 
that is 
\begin{equation}\label{eq:convex function has minimum at average}
     \int_{B_1(0)}|Du-(Du)_{B_1(0)}|^2\,\mr dx\leq \int_{B_1(0)}|Du-X|^2\,\mr dx, \quad \forall \, X\in \mb R^{n}.
\end{equation}

We now adapt the argument of H\"older estimates for gradients in \cite{Han-Lin-PDE}*{Theorem 3.13} to give the improved order estimates.
\begin{lemma}\label{lem:D2 integral beta}
Let $\{u_j\}_{j=1}^m$ be $C^{1,\alpha}$ ordered graphs on $B_t(0)$ which is a solution to $(\mc Q,\{g_j\})$. 
Suppose that \eqref{eq:Lambda and lambda}, \eqref{eq:ms AB norm}, \eqref{eq:ms AB at x00}, \eqref{eq:small C1alpha} and \eqref{eq:eta kappa0 and kappa1} are satisfied. 
Suppose that $\{u_j\}$ is connected in $B_s(0)$ ($s<t$) with
\begin{equation}\label{eq:D2uj assumption at t}
	 \sum_{j=1}^m\int_{B_{t}}|D^2u_j|^2\,\mr dx\leq \mc Kt^{n-\alpha} 
	   \end{equation}
for some constant $\mc K$. Then for $s\leq r< R\leq t$,
\begin{align*}
      \int_{B_r(0)}|DD_k\varphi-(DD_k\varphi)_{B_r(0)}|^2\,\mr dx
             &\leq C \Big(\frac{r}{R}\Big)^{n+1} \int_{B_R(0)}|DD_k\varphi-(DD_k\varphi)_{B_R(0)}|^2\,\mr dx+\\
             &\hspace{9em}+C\big(\mc K+\bm\kappa^*+\eta^2\big)r^{n+\alpha}.  
  \end{align*}
where $C$ is a constant depending only on $n,m,\Lambda,\lambda,\alpha$.
	\end{lemma}
\begin{proof}
So long as $r\geq t/16$, then $R/r\leq 16$ and by \eqref{eq:convex function has minimum at average}, 
\begin{align*}
    \int_{B_r(0)}|D^2\varphi-(D^2\varphi)_{B_r(0)}|^2\,\mr dx 
             &\leq \int_{B_r(0)}|D^2\varphi-(D^2\varphi)_{B_R(0)}|^2\,\mr dx\\
             &\leq 16^{n+1}\Big(\frac{r}{R}\Big)^{n+1} \int_{B_R(0)}|D^2\varphi-(D^2\varphi)_{B_R(0)}|^2\,\mr dx.
\end{align*}
Now we assume that $s<t/16$ so that Lemma \ref{lem:diff 22} applies.
By Proposition \ref{lem:n-epsilon order for uj}, 
there exists a constant $\mc B=\mc B(n,m,\Lambda,\lambda,\alpha)$ 
so that for $r<t$,
\begin{equation*}
	 \sum_{j=1}^{m} \int_{B_r(0)}|D^2u_j|^2\,\mr dx\leq \mc B \Big(\frac{r}{t}\Big)^{n-\epsilon}\sum_{j=1}^m\int_{B_t(0)}|D^2u_j|^2\,\mr dx+ \mc B(\bm \kappa^*+\eta^2)r^{n-2\epsilon}.  
	\end{equation*}
Together with \eqref{eq:D2uj assumption at t} and $2\epsilon\leq\alpha$, we conclude that
\begin{equation}\label{eq:D2uj assumption}
	\sum_{j=1}^{m} \int_{B_r(0)}|D^2u_j|^2\,\mr dx\leq \mc B (\mc K+\bm\kappa^*+\eta^2)r^{n-\alpha}.  
	\end{equation}
Recall that by \eqref{eq:diff sum}, $\varphi$ satisfies the equation
\begin{equation}\label{eq:PDE of varphi second time}
      -\dv\Big(\bm A^*DD_k\varphi+{\bm b^*}D_k\varphi\Big)+\bm c^*DD_k\varphi +d^* D_k\varphi=-\dv\Big(\hat{\bm f}-m\partial_{x_k}\ms A(x,\varphi,D\varphi)\Big)+\hat g.
\end{equation} 
We now fix $1\leq k\leq n$. For any $R<t$, let $f$ be the weak solution of 
\[ \dv\Big(\bm A^*(0)Df\Big) =0, \quad f|_{\partial B_R(0)}=D_k\varphi-D_k\varphi(0). \]
Note that by \eqref{eq:convex function has minimum at average},
\begin{equation}\label{eq:concave}
	\int_{B_R(0)}|Df-(Df)_{B_R(0)}|^2\,\mr dx\leq \int_{B_R(0)}|Df-(DD_k\varphi)_{B_R(0)}|^2\,\mr dx.
	\end{equation}
By harmoniticity of $f$ \citelist{\cite{Han-Lin-PDE}*{Lemma 3.10}\cite{Jost13}*{Lemma 14.4.5}}, 
\begin{equation}\label{eq:2nd Campanato}
	\int_{B_r(0)}|Df-(Df)_{B_r(0)}|^2\,\mr dx\leq A(n)\Big(\frac{r}{R}\Big)^{n+2}\int_{B_R(0)}|Df-(Df)_{B_R(0)}|^2\,\mr dx.
	\end{equation}
Then we claim that 
\begin{equation}\label{eq:Df ave to Dvarphi ave}
	  \int_{B_R(0)}|Df-(DD_k\varphi)_{B_R(0)}|^2\,\mr dx\leq C(m,\Lambda,\lambda)\int_{B_R(0)}|DD_k\varphi-(DD_k\varphi)_{B_R(0)}|^2.   
	  \end{equation}
Indeed, since $\dv \bm A^*(0)(DD_k\varphi)_{B_R(0)}=0$, then by the definition of $f$,
\[
\dv\Big(\bm A^*(0)\big(Df-(DD_k\varphi)_{B_R(0)}\big)\Big)=0.
\]
By taking $\phi:=D_k\varphi-D_k\varphi(0)-f$ as a test function, we obtain
\begin{align*}
	&\ \ \ \ \int_{B_R(0)}|Df-(DD_k\varphi)_{B_R(0)}|^2\,\mr dx\\
	&\leq C(m,\lambda)\int_{B_R(0)}\bm A^*(0)\Big(Df-(DD_k\varphi)_{B_R(0)}\Big)\Big(Df- (DD_k\varphi)_{B_R(0)}\Big)\,\mr dx\\
	&= C(m,\lambda)\int_{B_R(0)}\bm A^*(0)\Big(Df-(DD_k\varphi)_{B_R(0)}\Big)\Big(DD_k\varphi- (DD_k\varphi)_{B_R(0)}\Big)\,\mr dx\\ 
	&\leq C(m,\Lambda,\lambda) \int_{B_R(0)}|DD_k\varphi- (DD_k\varphi)_{B_R(0)}|^2\,\mr dx+\frac{1}{2}\int_{B_R(0)}|Df-(DD_k\varphi)_{B_R(0)}|^2\,\mr dx,
\end{align*}
where we used \eqref{eq:bound A* b* c* d*} in the last inequality.
Then \eqref{eq:Df ave to Dvarphi ave} is proved. Combining \eqref{eq:concave} and \eqref{eq:Df ave to Dvarphi ave}, we obtain 
\begin{equation}\label{eq:Df and DDv}
	 \int_{B_R(0)}|Df-(Df)_{B_R(0)}|^2\,\mr dx\leq C(m,\Lambda,\lambda)\int_{B_R(0)}|DD_k\varphi-(DD_k\varphi)_{B_R(0)}|^2\,\mr dx. 
	 \end{equation}
Observe that 
\begin{equation}
   \int_{B_r(0)}|(Df)_{B_r(0)}-(DD_k\varphi)_{B_r(0)}|^2\,\mr dx \leq \int_{B_r(0)}|Df-DD_k\varphi|^2\,\mr dx.
\end{equation}
Then we have
\begin{align*}
	&\ \ \  \ \int_{B_r(0)}|DD_k\varphi-(DD_k\varphi)_{B_r(0)}|^2\,\mr dx\\
	&\leq 3\int_{B_r(0)}|DD_k\varphi -Df|^2+|Df-(Df)_{B_r(0)}|^2+|(Df)_{B_r(0)}-(DD_k\varphi)_{B_r(0)}|^2\,\mr dx\\
	&\leq 6\int_{B_r(0)}|DD_k\varphi -Df|^2\,\mr dx+3\int_{B_r(0)}|Df-(Df)_{B_r(0)}|^2\,\mr dx .
\end{align*}
Plugging \eqref{eq:Dvarphi-f} and \eqref{eq:2nd Campanato} in it, we obtain
 \begin{align*}
        \int_{B_r(0)}|DD_k\varphi-(DD_k\varphi)_{B_r(0)}|^2\,\mr dx
        \leq  C\int_{B_R(0)}\Big(R^{2\alpha}|D^2\varphi|^2 + (\bm\kappa^*+\eta^2)R^\alpha\Big)\,\mr dx+\\
        +3A(n)\Big(\frac{r}{R}\Big)^{n+2}\int_{B_R(0)}|Df-(Df)_{B_R(0)}|^2\,\mr dx,
 \end{align*}
where $C=C(n,m,\Lambda,\lambda)$.
This together with \eqref{eq:D2uj assumption} and \eqref{eq:Df and DDv} gives that 
\begin{align*}
      \int_{B_r(0)}|DD_k\varphi-(DD_k\varphi)_{B_r(0)}|^2\,\mr dx
          \leq 3A(n)\Big(\frac{r}{R}\Big)^{n+2} \int_{B_R(0)}|DD_k\varphi-(DD_k\varphi)_{B_R(0)}|^2\,\mr dx\\
        + C (\mc K+\bm\kappa^*+\eta^2)R^{n+\alpha},
	\end{align*}
where we used the fact that $\delta_0,\alpha\leq 1$.
Then by Lemma \ref{lem:order}, there exists $C=C(n,A,\alpha)$ so that for any $r<R\leq t$,
\begin{align*}    
     \int_{B_r(0)}|DD_k\varphi-(DD_k\varphi)_{B_r(0)}|^2\,\mr dx
         \leq C\Big(\frac{r}{R}\Big)^{n+1} \int_{B_R(0)}|DD_k\varphi-(DD_k\varphi)_{B_R(0)}|^2\,\mr dx+\\
         +C\big(\mc K+\bm\kappa^*+\eta^2\big)r^{n+\alpha}.  
  \end{align*}
This completes the proof.
	\end{proof}

So far, we have proved the improved order estimates over some tilt-plane for any given $r$ so that $\{u_j\}$ is connected. The following gives that the tilt-plane will be close to each other in the order of $\alpha$.

\begin{lemma}\label{lem:bdd average diff}
Let $\{u_j\}_{j=1}^m$, $(\mc Q,\{g_j\})$, $\mc K$, $t>0$ be the same as in Lemma \ref{lem:D2 integral beta}. In particular, we assume that $\{u_j\}$ is connected in $B_s(0)$. 
Let $C_1$ be the constant in Lemma \ref{lem:D2 integral beta}. 
Then for $s\leq r\leq R\leq t$,
\[   
    |(D^2\varphi)_{B_R(0)}-(D^2\varphi)_{B_r(0)}|^2
        \leq C_2  \Big(
                  \frac{R}{t^{n+1}}\int_{B_t(0)}
                  |D^2\varphi-(D^2\varphi)_{B_t(0)}|^2\,\mr dx
                  +(\mc K+\bm\kappa^*+ \eta^2)R^{\alpha}
                  \Big) ,
	\]
	where $C_2=C_2(n,\alpha,C_1)$.
\end{lemma}
\begin{proof}
For $r\in[s,t]$, let 
\[  
       \sigma(r):=\int_{B_r(0)} |D^2\varphi-(D^2\varphi)_{B_r(0)}|^2\,\mr dx.\]	
Then by Lemma \ref{lem:D2 integral beta},
\[ 
    \sigma(r)\leq C_1 \Big( \frac{r}{t} \Big)^{n+1}\sigma(t) + C_1( \mc K+\bm\kappa^*+\eta^2 )r^{n+\alpha}, \quad 
    C_1=C_1(n,m,\Lambda,\lambda,\alpha).  \]
For any $s\leq r<R\leq t$, we then have
\begin{align*}
    |(D^2\varphi)_{B_R(0)}-(D^2\varphi)_{B_r(0)}|^2
        &\leq \frac{2}{|B_r|} \int_{B_{r}(0)} 
            \Big( 
            |D^2\varphi-(D^2\varphi)_{B_R(0)}|^2
            +(|D^2\varphi-(D^2\varphi)_{B_r(0)}|^2
            \Big)\\
	&\leq \frac{2}{|B_r|}(  \sigma(R)+\sigma(r) )\\
	&\leq \frac{4C_1}{|B_r|}
            \Big[
                \Big(\frac{R}{t}\Big)^{n+1}\sigma(t) 
                +(\mc K+\bm\kappa^*+ \eta^2)R^{n+\alpha}
            \Big].
\end{align*}	
In particular, for $s\leq \hat R_i=2^{-i}R$,
\[ 
    \Big|(D^2\varphi)_{B_{\hat R_i}(0)}-(D^2\varphi)_{B_{\hat R_{i+1}}(0)}\Big|^2
       \leq \frac{2^{n+2}C_1}{|B_1|}\Big(
           \frac{\hat R_i}{t^{n+1}}\cdot\sigma(t)
           +(\mc K+\bm\kappa^*+\eta^2)\hat R_i^{\alpha}
           \Big).  \]
 It then implies that 
\begin{align*}   
    |(D^2\varphi)_{B_{R}(0)}-(D^2\varphi)_{B_{\hat R_{k}}(0)}|^2
        &\leq \Big(
            \sum_{i=0}^{k-1} \Big|(D^2\varphi)_{B_{\hat R_i}(0)}-(D^2\varphi)_{B_{\hat R_{i+1}}(0)}\Big|
            \Big)^2 \\
	&\leq \frac{2^{n+2}C_1}{|B_1|}\Big[
            \sum_{i=0}^{k-1}\Big(
                 \sqrt{\frac{\sigma(t)}{t^{n+1}}}\cdot\hat R_i^{\frac{1}{2}}
                 +\sqrt{\mc K+\bm\kappa^*+\eta^2}\hat R_i^{\frac{\alpha}{2}}
                 \Big)
            \Big]^2\\
	&\leq \frac{2^{n+2}C_1}{|B_1|}\Big[  
            \frac{1}{1-2^{-\frac{\alpha}{2}}} \Big(
                 \sqrt{\frac{\sigma(t)}{t^{n+1}}}\cdot R^{\frac{1}{2}}+
                 \sqrt{\mc K+\bm \kappa^*+\eta^2}R^{\frac{\alpha}{2}}
                 \Big)   
            \Big]^2\\
	&\leq \frac{2^{n+3}C_1}{|B_1|}\cdot 
            (1-2^{-\frac{\alpha}{2}})^{-2} \Big(
                \frac{\sigma(t)}{t^{n+1}}\cdot R
                +(\mc K+\bm\kappa^*+\eta^2)R^{\alpha}
                \Big) .
\end{align*}
Suppose that $\hat R_{k+1}\leq r<\hat R_{k}$. Then 
\begin{align*}
    |(D^2\varphi)_{B_{R}(0)}-(D^2\varphi)_{B_{r}(0)}|^2
    &\leq 2|(D^2\varphi)_{B_{R}(0)} - (D^2\varphi)_{B_{\hat R_{k}}(0)}|^2
        +2|(D^2\varphi)_{B_{\hat R_k}(0)}-(D^2\varphi)_{B_{r}(0)}|^2\\
    &\leq \frac{2^{n+4}C_1}{|B_1|} \cdot (1-2^{-\frac{\alpha}{2}})^{-2}\Big(
        \frac{\sigma(t)}{t^{n+1}}\cdot R
        +(\mc K+\bm\kappa^*+\eta^2)R^{\alpha}
        \Big)+\\
    &\hspace{8em}    +\frac{8C_1}{|B_r|}\Big[
        \Big(
        \frac{\hat R_k}{t}
        \Big)^{n+1}\sigma(t)
        +(\mc K+\bm\kappa^*+\eta^2)\hat R_k^{n+\alpha}
    \Big]\\
    &\leq \frac{2^{n+5}C_1}{|B_1|}\cdot(1-2^{-\frac{\alpha}{2}})^{-2}
        \Big(
        \frac{\sigma(t)}{t^{n+1}}\cdot R
        +(\mc K+\bm\kappa^*+\eta^2)R^{\alpha}
        \Big).
\end{align*}
Then the lemma follows by taking 
\[  
    C_2= \frac{2^{n+5}C_1}{|B_1|} \cdot (1-2^{-\frac{\alpha}{2}})^{-2}.\]	
\end{proof}

Recall that
\begin{gather*}
    \bm\sigma^*(r):= \sum_{j=1}^m\int_{B_r(0)}
        |D^2u_j- (D^2u_j)_{B_r(0)}|^2\,\mr dx.
       \end{gather*}
Lemma \ref{lem:D2 integral beta} combining with Lemma \ref{lem:diff 22} implies the order estimates of $\bm \sigma^*(r)$. 
\begin{lemma}\label{lem:improved order of D^2u}
Let $\{u_j\}_{j=1}^m$, $(\mc Q,\{g_j\})$, $\mc K$, $t>0$, $\bm\kappa^*
$ be the same as in Lemma  \ref{lem:D2 integral beta}.
Then there exists a constant $C$ depending only on
$n$, $m$, $\Lambda$, $\lambda$, $\alpha$, 
so that for all $r\leq t$,
	\[  \bm \sigma^*(r)\leq C \Big(\frac{r}{t}\Big)^{n+1}\bm\sigma^*(t)+C(\mc K+\bm\kappa^*+\eta^2)r^{n}.  \] 
	\end{lemma}
\begin{proof}
So long as $r\geq t/16$, then $t/r\leq 16$ and by \eqref{eq:convex function has minimum at average},
\begin{align*}
    \bm\sigma^*(r)
    \leq \sum_{j=1}^m\int_{B_r(0)}
        |D^2u_j-(D^2u_j)_{B_t(0)}|^2\,\mr dx
    \leq 16^{n+1}\Big(
        \frac{r}{t}
        \Big)^{n+1}\bm\sigma^*(t).
\end{align*}
This is the desired inequality. 
We now assume that $r\leq t/16$ and then Lemma \ref{lem:diff 22} can be applied.

We prove the lemma inductively. 
When $m=1$, then the desired inequality follows from Lemma \ref{lem:D2 integral beta} by taking $R=t$.

Now we assume that the lemma is true when the number of the sheets is less than or equal to $m-1$. 
We also assume that $\{u_j\}$ is connected in $B_t(0)$; otherwise, the desired inequality follows from induction. 
Let 
\[ s:=\inf\{ r>0;\{u_j \} \text{ is connected in $B_r(0)$}\}.  \]
By Lemma \ref{lem:D2 integral beta}, there exists $C=C(n,m,\Lambda,\lambda,\alpha)$
such that for $s< r\leq t$,
\begin{align*} 
    &\hspace{1.5em}\int_{B_r(0)} 
        |DD_k\varphi - (DD_k\varphi)_{B_r(0)}|^2
        \,\mr dx\\
    &\leq C \Big(
        \frac{r}{t}\Big)^{n+1} \int_{B_{t}(0)}
            |DD_k\varphi-(DD_k\varphi)_{B_{t}(0)}|^2
            \,\mr dx 
        +C(\mc K+\bm\kappa^*+\eta^2)r^{n+\alpha}\\
    &\leq C \Big(
        \frac{r}{t}
        \Big)^{n+1}\bm\sigma^*(t) 
        +C(\mc K+\bm\kappa^*+\eta^2)r^{n+\alpha}.
	 \end{align*}
Recall that by Lemma \ref{lem:diff 22}, for $s<r\leq t$,
\begin{equation}
\begin{split}
    &\hspace{1.5em}\int_{B_r(0)} 
        |D^2u_i-D^2u_j-(D^2u_i)_{B_r(0)}
        +(D^2u_j)_{B_r(0)}|^2
        \,\mr dx\\
    &\leq 2\int_{B_r(0)} 
        |D^2u_i-D^2u_j|^2
        \,\mr dx
        +2|B_r|\cdot |(D^2u_i-D^2u_j)_{B_r(0)}|^2
        \label{eq:D2ui-D2uj average}\\
    &\leq 4\int_{B_r(0)} 
        |D^2u_i-D^2u_j|^2
        \,\mr dx
        \leq C\cdot \Big((1+\eta^2)(\kappa+\kappa_0)^2+\kappa_1^2\Big)r^n.
	\end{split}
 \end{equation}
Thus we obtain that for $s< r\leq t$,
\begin{align*}
    &\hspace{1.5em} \int_{B_r(0)}
        |D^2u_i-(D^2u_i)_{B_r(0)}|^2
        \,\mr dx\\
    &\leq \int_{B_r(0)}
        \sum_{j=1}^{m}\Big(
            |D^2u_i-D^2u_j-(D^2u_i)_{B_r(0)}
            +(D^2u_j)_{B_r(0)}|^2
            \Big) 
        +m|D^2\varphi-(D^2\varphi)_{B_r(0)}|^2 
        \,\mr dx\\
    &\leq C\cdot \Big((1+\eta^2)(\kappa+\kappa_0)^2+\kappa_1^2\Big)r^n 
        +C \Big(
            \frac{r}{t}
            \Big)^{n+1} \bm\sigma^*(t)
        +C(\mc K+\bm\kappa^*+\eta^2)r^{n+\alpha}\\
    &\leq C \Big(
            \frac{r}{t}
            \Big)^{n+1} \bm\sigma^*(t)
        +C(\mc K+\bm\kappa^*+\eta^2)r^{n} ,
	  \end{align*}
where $C=C(n,m,\Lambda,\lambda,\alpha)$. This implies the desired inequality.   

\medskip
It remains to consider the case of $r<s$. Observe that we have proved that 
\begin{equation}\label{eq:order n of vanished average at s}
      \bm \sigma^*(s) \leq C\Big(\frac{s}{t}\Big)^{n+1}\bm\sigma^*(t)+C(\mc K+\bm\kappa^*+\eta^2)s^{n}.
\end{equation}
By Proposition \ref{lem:n-epsilon order for uj}, there exists a constant $\mc B=\mc B(n,m,\Lambda,\lambda,\alpha)$ so that
\begin{equation*}
      \sum_{j=1}^{m} \int_{B_s(0)}|D^2u_j|^2\,\mr dx\leq \mc B \Big(\frac{s}{t}\Big)^{n-\epsilon}\int_{B_{t}(0)}\sum_{j=1}^{m}|D^2u_j|^2\,\mr dx+ \mc B(\bm\kappa^*+\eta^2)s^{n-2\epsilon}.  
\end{equation*}
Together with \eqref{eq:D2uj assumption at t} and $2\epsilon\leq\alpha$, we obtain that
\begin{equation}\label{eq:D2uj assumption at s}
    \sum_{j=1}^{m} \int_{B_s(0)}
        |D^2u_j|^2
        \,\mr dx
    \leq \mc B (\mc K + \bm\kappa^*+\eta^2)s^{n-\alpha}.  
\end{equation}
Let $\mc K':=\mc B(\mc K+\bm\kappa^*+\eta^2)$. 
Note that in $B_s(0)$, $\{u_j\}_{j=1}^{m}$ can be divided into two disjoint ordered graphs 
$\{u_j\}_{j=1}^i$ and $\{u_j\}_{j=i+1}^{m}$. 
Thus by induction using \eqref{eq:D2uj assumption at s} in place of \eqref{eq:D2uj assumption at t}, we have 
for all $r<s$,
\begin{gather*}
\begin{split}
    \sum_{j=1}^i\int_{B_r(0)}
        |D^2u_j-(D^2u_j)_{B_r(0)}|^2
        \,\mr dx
    &\leq  C \Big( 
        \frac{r}{s} \Big)^{n+1} \sum_{j=1}^i \int_{B_s(0)}
            |D^2u_j-(D^2u_j)_{B_s(0)}|^2
            \,\mr dx+\\
        &\hspace{4em}+C(\mc K'+\bm\kappa^*+\eta^2)r^{n};
    \end{split}\\
\begin{split}
    \sum_{j=i+1}^{m}\int_{B_r(0)}
        |D^2u_j-(D^2u_j)_{B_r(0)}|^2
        \,\mr dx
    &\leq  C \Big(
        \frac{r}{s}
        \Big)^{n+1} \sum_{j=i+1}^m\int_{B_s(0)}
            |D^2u_j-(D^2u_j)_{B_s(0)}|^2
            \,\mr dx+\\
        &\hspace{4em}+C(\mc K'+\bm\kappa^*+\eta^2)r^{n}. 
      \end{split}
\end{gather*}
The taking the sum of the two inequalities above, we obtain 
\[
    \bm \sigma^*(r)\leq C\Big(
        \frac{r}{s}
        \Big)^{n+1} \bm \sigma^*(s)
        +C(\mc K+\bm\kappa^*+\eta^2)r^{n},
\]
where $C=C(n,m,\Lambda,\lambda,\alpha)$. 
Plugging \eqref{eq:order n of vanished average at s} into it, we obtain that 
\begin{align*}
    \bm \sigma^*(r)
    &\leq C\Big(
        \frac{r}{s}
        \Big)^{n+1} \Big[
            \Big(
            \frac{s}{t}
            \Big)^{n+1} \bm\sigma^*(t) 
        + C(\mc K+\bm\kappa^*+\eta^2)s^{n}
        \Big]
        +C(\mc K+\bm\kappa^*+\eta^2)r^{n}\\
    &\leq C\Big(
        \frac{r}{t}
        \Big)^{n+1}\bm\sigma^*(t) 
        + C(\mc K+\bm\kappa^*+\eta^2)r^{n}.
\end{align*}
This finishes the proof of Lemma \ref{lem:improved order of D^2u}.  
	\end{proof}

%%%%%%%%%%%%%%%%%%%%%%%%%%%%%%%%%%%%%%%%%%%%%%%%

\subsection{Monotonicity of the the second derivatives}
\label{subsec:monotonicity of second derivative}

Now we are ready to prove a monotonicity formula for the average of $L^2$-integral of $D^2u_j$.
When $m=1$, the desired result follows from Lemma \ref{lem:bdd average diff} and  Lemma \ref{lem:improved order of D^2u}. For $m\geq 2$, we will use Lemma \ref{lem:diff 22} and run an inductive method.

\begin{proposition}\label{lem:bounded hessian average}
Let $\{u_j\}_{j=1}^m$ be $C^{1,\alpha}$ ordered graphs on $B_t(0)$ ($t\leq 1$) which is a solution to $(\mc Q,\{g_j\})$.
Suppose that \eqref{eq:Lambda and lambda}, \eqref{eq:ms AB norm}, \eqref{eq:ms AB at x00}, \eqref{eq:small C1alpha} and \eqref{eq:eta kappa0 and kappa1} are satisfied. 
For simplicity, denote by
\[  
    \oli{\bm \sigma}(r):=\sum_{j=1}^m
        \frac{1}{|B_r|} \int_{B_r(0)}
            |D^2u_j|^2
            \,\mr dx , \quad 
    \forall\, r\leq t.  \]	
Then there exists a constant $C=C(n,m,\Lambda,\lambda,\alpha)$ so that 
\[
    \oli{\bm \sigma}(r)
    \leq C(\oli{\bm\sigma}(t)+\bm\kappa^* + \eta^2).  \]	
	\end{proposition}
\begin{proof}
We first note that 
\begin{equation}\label{eq:average less than bar sigma}
    \sum_{j=1}^m |(D^2u_j)_{B_t(0)}|^2
    \leq \sum_{j=1}^m
        \frac{1}{|B_t|}\int_{B_t(0)}
        |D^2u_j|^2
        \,\mr dx
    =\oli{\bm \sigma}(t).
    \end{equation}
By \eqref{eq:convex function has minimum at average}, we have that
\begin{equation}\label{eq:minus average is smaller}
    \frac{1}{|B_r|}\cdot \bm\sigma^*(r)=\sum_{j=1}^m\frac{1}{|B_r|}\int_{B_r(0)}
        |D^2u_j-(D^2u_j)_{B_r(0)}|^2
    \,\mr dx
    \leq\oli{\bm\sigma}(r).
\end{equation}  
We prove the proposition inductively. When $m=1$, we have by \eqref{eq:average less than bar sigma},
\begin{align*}
    \int_{B_r(0)}|D^2u|^2\,\mr dx&\leq 3\int_{B_r(0)}
        |D^2u-(D^2u)_{B_r(0)}|^2
        +|(D^2u)_{B_r(0)}-(D^2u)_{B_t(0)}|^2
        +|(D^2u)_{B_t(0)}|^2
        \,\mr dx\\
    &\leq 3\bm\sigma^*(r)+3|B_r|\cdot |(D^2u)_{B_r(0)}-(D^2u)_{B_t(0)}|^2+3|B_r|\oli{\bm \sigma}(t).
\end{align*}
Applying Lemma \ref{lem:bdd average diff} and Lemma \ref{lem:improved order of D^2u},
\begin{align*}
    \int_{B_r(0)}|D^2u|^2\,\mr dx
    &\leq C\Big(
        \frac{r}{t}
        \Big)^{n+1}\bm\sigma^*(t) 
        +C(\oli{\bm \sigma}(t)+\bm\kappa^*+\eta^2)r^{n}+\\
    &\hspace{6em}+C{|B_r|}\Big(
            \frac{t}{t^{n+1}}\bm \sigma^*(t)
            + (\oli{\bm \sigma}(t)
            +\bm\kappa^*+\eta^2)t^{\alpha}
            \Big)
            +C\oli{\bm \sigma}(t)\cdot r^n.
\end{align*}
This together with \eqref{eq:minus average is smaller} implies that 
\[
    \int_{B_r(0)}|D^2u|^2\,\mr dx
    \leq  C(\oli{\bm \sigma}(t)+\bm\kappa^*+\eta^2)r^{n}. 
\]
Then the desired inequality follows immediately.	
	
Now suppose that the lemma is true when the number of sheets is less than or equal to $m-1$.
We also assume that $\{u_j\}$ is connected in $B_t(0)$; otherwise, the desired inequality follows from induction. 
 Let 
\[ 
    s:=\inf\{ r>0;\{u_j \} \text{ is connected in $B_r(0)$} \}.  \]

\medskip
{\noindent\em Case 1: $s\leq r\leq t$.} Without loss of generality, we assume that $r<t/32$ so that Lemma \ref{lem:diff 22} applies.
Then by Lemma \ref{lem:bdd average diff}, for $s\leq r\leq t$,
\begin{equation}\label{eq:two average diff bounded by bar sigma}
\begin{split}   
    &\hspace{1.5em}|(D^2\varphi)_{B_t(0)}-(D^2\varphi)_{B_r(0)}|^2\\
    &\leq C\Big(
        \frac{t}{t^{n+1}}\int_{B_t(0)}
        |D^2\varphi-(D^2\varphi)_{B_t(0)}|^2
        \,\mr dx
        +(\oli{\bm \sigma}(t)
        + \bm\kappa^*+\eta^2)t^{\alpha}
        \Big)\\
    &\leq C\oli{\bm\sigma}(t)+C(\bm\kappa^*+\eta^2)t^{\alpha},
\end{split}
\end{equation}
where $C=C(n,m,\Lambda,\lambda,\alpha)$ and we used \eqref{eq:minus average is smaller} in the last inequality.
On the other hand, by Lemma \ref{lem:improved order of D^2u} and \eqref{eq:minus average is smaller}, there exists a constant $C$ depending only on 
$n$, $m$, $\Lambda$, $\lambda$, $\alpha$
so that for all $r\leq t$,
\begin{align*} 
    \int_{B_r(0)}
        |D^2\varphi-(D^2\varphi)_{B_r(0)}|^2
        \,\mr dx
    &\leq C\Big(
        \frac{r}{t}
        \Big)^{n+1}\bm\sigma^*(t) 
        +C(\oli{\bm \sigma}(t)
        + \bm\kappa^*+\eta^2)r^n\\
    &\leq C(\oli{\bm\sigma}(t)+\bm\kappa^*+\eta^2)r^n.
\end{align*}
This together with \eqref{eq:two average diff bounded by bar sigma} and \eqref{eq:average less than bar sigma} implies that 
\begin{align*}
    \int_{B_r(0)}|D^2\varphi|^2\,\mr dx
    &\leq 3\int_{B_r(0)} 
        |D^2\varphi- (D^2\varphi)_{B_r(0)}|^2 
        + |(D^2\varphi)_{B_r(0)}-(D^2\varphi)_{B_t(0)}|^2 
        +|(D^2\varphi)_{B_t(0)}|^2
        \,\mr dx\\
    &\leq C( \oli{\bm\sigma}(t) + \bm\kappa^* + \eta^2 )r^n.	
	\end{align*}	
Recall that by Lemma \ref{lem:diff 22}, there exists $C=C(n,m,\Lambda,\lambda)$ so that 
\[ 
    \int_{B_r(0)} |D^2u_i-D^2u_j|^2 \,\mr dx
    \leq  C\cdot \Big((1+\eta^2)(\kappa+\kappa_0)^2+\kappa_1^2\Big)r^n\leq C\cdot (\bm\kappa^*+\eta^2)r^n.  \]
We then conclude that 
\begin{align*}  
    \int_{B_r(0)} |D^2u_i|^2\,\mr dx
    &\leq 2\int_{B_r(0)}
        |D^2\varphi-D^2u_i|^2
        +|D^2\varphi|^2
        \,\mr dx\\
    &\leq C(m)\sum_{j=1}^{m}\int_{B_r(0)}
        |D^2u_j-D^2u_i|^2
        +|D^2\varphi|^2
        \,\mr dx\\
    &\leq C(\oli{\bm\sigma}(t)+\bm\kappa^* + \eta^2)r^n,
\end{align*}
where $C=C(n,m,\Lambda,\lambda,\alpha)$.
This is the desired inequality.

\medskip
{\noindent\em Case 2: $r<s$.} 
By the definition of $s$, $\{u_j\}_{j=1}^{m}$ can be divided into two ordered graphs 
$\{u_j\}_{j=1}^i$ and $\{u_j\}_{j={i+1}}^{m}$ in $B_s(0)$. 
Thus by induction, there exists $C=C(n,m,\Lambda,\lambda,\alpha)$ so that 
\begin{gather*}
    \sum_{j=1}^i \frac{1}{|B_r|}
        \int_{B_r(0)}|D^2u_j|^2
        \,\mr dx 
    \leq  C\Big(
        \sum_{j=1}^i \frac{1}{|B_s|}\int_{B_s(0)}
            |D^2u_j|^2
            \,\mr dx 
         + \bm\kappa^* + \eta^2
        \Big);\\
    \sum_{j=i+1}^{m} \frac{1}{|B_r|}\int_{B_r(0)}
        |D^2u_j|^2
        \,\mr dx
    \leq  C\Big(
        \sum_{j=i+1}^m \frac{1}{|B_s|} \int_{B_s(0)} 
            |D^2u_j|^2
            \,\mr dx 
        +\bm\kappa^* + \eta^2
        \Big).
\end{gather*}
Combining the two inequalities together, we conclude that
\begin{equation}  \label{eq:sigma r bounded by sigma s} 
    \oli{\bm \sigma}(r)
    \leq C(\oli{\bm \sigma}(s)+\bm \kappa^* + \eta^2).
    \end{equation}
Observe that by Case 1, 
\[   
    \oli{\bm \sigma}(s)
    \leq C(\oli{\bm\sigma}(t)+\bm \kappa^* + \eta^2).
\]
Plugging this into \eqref{eq:sigma r bounded by sigma s}, then the desired inequality follows immediately.
This finishes the proof of Proposition \ref{lem:bounded hessian average}.
	\end{proof}

%%%%%%%%%%%%%%%%%%%%%%%%%%%%%%%%%%%%%%%%%%%%%%%%%%%%
%Section Proof of the main theorem
%%%%%%%%%%%%%%%%%%%%%%%%%%%%%%%%%%%%%%%%%%%%%%%%%%%%
\section{Lipschitz upper bound for the first derivatives}\label{sec:lipschitz}

In this section, we will give the upper bound of the Lipschitz norm of the first derivatives of ordered graphs in Theorem \ref{thm:main thm}. 

Note that we have proved the uniform upper bound for the $\oli{\bm\sigma}(r)$ for all small $r$. This together with Poincar\'e inequality (Lemma \ref{lem:osc of u}) gives that the graphs $\{u_j\}$ will be close to a plane $P$ in the sense of $L^2$; see Lemma \ref{lem:int osc of Dvarphi} . We will then consider the differential equation of the difference to the plane $P$. Applying the $C^{1,\alpha}$ estimates to the divergence form, we will then obtain the upper bound of the derivative of such new functions by $\oli{\bm \sigma}(r)$ and $\bm \kappa^*$. We remark that Lemma \ref{lem:osc of diff} and Lemma \ref{lem:gradient diff} are used to bound the remainder terms in this process. To finish the proof, we will apply Proposition \ref{lem:bounded hessian average} and Proposition \ref{prop:ui is W22} to bound $\oli{\bm\sigma}(r)$ by $\bm\kappa^*$ and $\eta$; see Theorem \ref{thm:derivative is Lip} for more details.

Recall that 
\[    
    \bm\kappa^*:= \kappa^2+\kappa_0^2+\kappa_1^2+\delta_0^2,\]
and 
\[   \eta:=\Big(\|\partial_z\partial_x\ms A\|_{C^0(\bm U_1)}+\|\partial_x\ms B\|_{C^1(\bm U_1)}\Big)\cdot \sum_{j=1}^m\|u_j\|_{C^0(B_{1}(0))}.\]

%%%%%%%%%%%%%%%%%%%%%%%%%%%%%%%%%%%%%%%%%%%%%%%%%%
\subsection{The first derivative upper bound}\label{subsec:1st derivative estimates}
Let $\{u_j\}_{j=1}^m$ be $C^{1,\alpha}$ ordered graphs on $B_t(0)$ ($t\leq 1$) which is a solution to $(\mc Q,\{g_j\})$.
Recall that
\begin{equation*} 
    \oli{\bm \sigma}(r)
    :=\sum_{j=1}^m\frac{1}{|B_r|}\int_{B_r(0)}
        |D^2u_j|^2
        \,\mr dx.  
\end{equation*}
The following estimates are well-known and we provide it here for the readers' convenience.
\begin{lemma}[Poincar\'e inequalities \cite{GT}*{Estimate (7.45)}] \label{lem:osc of u}
For any convex domain $\Omega$ and 
$u\in W^{1,p}(\Omega)$, $1\leq p< +\infty$,
\[   
    \|u-u_{B_r}\|_{L^p(\Omega)}
    \leq r^{1-n}d^n\|Du\|_{L^p(\Omega)}, \quad d=\mr{diam}\,\Omega.
 \]
	\end{lemma}
This yields the following inequalities.
\begin{lemma}\label{lem:int osc of Dvarphi}
Let $\{u_j\}_{j=1}^m$ be $C^{1,\alpha}\cap W^{2,2}$ ordered graphs on $B_t(0)$ ($t\leq 1$). Then for $r\leq t$,
\begin{gather*}
    \sum_{j=1}^m\frac{1}{|B_r|}\int_{B_r(0)}
        |Du_j(x)-(Du_j)_{B_r(0)}|^2
        \,\mr dx 
    \leq {2^{2n}}\oli{\bm \sigma}(r) r^{2};\\
    \sum_{j=1}^m\frac{1}{|B_r|}\int_{B_r(0)}
        |u_j(x)-(u_j)_{B_r(0)}-(Du_j)_{B_r(0)}\cdot x|^2
        \,\mr dx
    \leq {2^{4n}} \oli{\bm \sigma}(r) r^{4}.  
	 \end{gather*}	
	\end{lemma}

Now we are going to use the equation to bound oscillations. 
We will use the following notions. 
Let $\{u_j\}$ be ordered graphs on $B_r(0)$ which is a solution to 
$(\mc Q,\{g_j\})$. 
Recall that by \eqref{eq:sum form}, 
\begin{equation}\label{eq:equation of varphi} 
    m\mc Q\varphi
    =\dv\bm f(\{u_j\})-g(\{u_j\},\{g_j\})
  \end{equation}
in $B_r(0)$ in the weak sense, where 
$\bm f(\{u_j\})$ and $g(\{u_j\},\{g_j\})$
are defined in \eqref{eq:def of bm f} and \eqref{eq:def of g}, respectively.
Let 
\[  
    \varphi_r=(\varphi)_{B_r(0)}+(D\varphi)_{B_r(0)}\cdot x.
\]
Note that 
\begin{align*}
    \mc Q\varphi_r:
    &=-\dv\,\ms A(x,\varphi_r,D\varphi_r)
        +\ms B(x,\varphi_r,D\varphi_r)\\
    &=-\mr{tr}\,(\partial_x\ms A(x,\varphi_r,D\varphi_r))
        -\partial_z\ms A(x,\varphi_r,D\varphi_r)\cdot (D\varphi)_{B_r(0)}
        +\ms B(x,\varphi_r,D\varphi_r).
\end{align*}

\begin{lemma}
There exists a constant $C=C(n,m,\Lambda,\lambda)$ such that for all $x\in B_r(0)$,
\begin{equation}\label{eq:bound mc Q varphir}
    -C(\delta_0+\eta)\leq \mc Q\varphi_r\leq C(\delta_0+\eta),
\end{equation}
\end{lemma}
\begin{proof}
Since $\ms A(x,0,0)=0$, then
\[
\begin{split}
&\ \ \ \ |\partial_x\ms A(x,\varphi_r,D\varphi_r)|\leq \|\partial_z\partial_x\ms A\|_{C^0(\bm U_r)}|\varphi_r|+\|\partial_{\bm p}\partial_x\ms A\|_{C^0(\bm U_r)}|D\varphi_r|\\
&\leq \|\partial_z\partial_x\ms A\|_{C^0(\bm U_r)}\Big( |\varphi|_{C^0(B_r(0))}+r|D\varphi|_{C^0(B_r(0))}\Big)+\|\partial_{\bm p}\partial_x\ms A\|_{C^0(\bm U_r)}|D\varphi|_{C^0(B_r(0))}\\
&\leq C(\eta+\delta_0).
\end{split}
\]
Similarly,
\[
|\ms B(x,\varphi_r,D\varphi_r)|\leq C(\eta+\delta_0).
\]
On the other hand,
\[  |\partial_z\ms A(x,\varphi_r,D\varphi_r)\cdot (D\varphi)_{B_r(0)}|\leq C\delta_0.\]
Hence the lemma is proved.
\end{proof}
\begin{comment}
Note that by \eqref{eq:small C1alpha}, 
$\|\varphi_r\|_{C^1(B_1(0))}\leq C(\delta_0+\delta^*)$,
where $\delta^*$ is a constant so that 
\begin{equation}\label{eq:def of delta*}
    \sup_{x\in B_r(0)}(|u_1|+|u_m|)\leq \delta^*.
\end{equation}  
Together with the fact that $\ms A(x,0,0)=0$, $\ms B(x,0,0)=0$ in \eqref{eq:ms AB at x00} and the norm of $\ms A$, $\ms B$ in \eqref{eq:ms AB norm}, we have that 
%where $C$ depends only on $m$ and $\Lambda$.
\end{comment}
For simplicity, we denote
\begin{gather*}
    \Psi(x):=\varphi(x)-(\varphi)_{B_r(0)} - (D\varphi)_{B_r(0)}\cdot x;\\
    \Psi_j(x):=u_j(x)-(u_j)_{B_r(0)} - (Du_j)_{B_r(0)}\cdot x.
	\end{gather*}
By \eqref{eq:equation of varphi} and Lemma \ref{lem:taking difference}, we obtain
\begin{equation}\label{eq:pde of Psi}
\begin{split}
    \dv\big(\bm A(\varphi,\varphi_r)D\Psi 
        +\bm b(\varphi,\varphi_r) \Psi\big) 
        +\bm c(\varphi,\varphi_r)D\Psi
        +d(\varphi,\varphi_r)\Psi\\
    =\dv\bm f(\{u_j\})
        -g(\{u_j\},\{g_j\})-m\mc Q\varphi_r,
\end{split}
\end{equation}
where $\bm f,g$ are the same in \eqref{eq:equation of varphi}.

The following result gives that the oscillation of $\Psi_j$ has the second order. Plugging this into the interior $C^{1,\alpha}$ estimates gives the oscillation of $Du_j$ in small balls. We remark that the interior estimates can not be applied to $Du_j$ because $u_j$ itself does not satisfy a weak partial differential equation.
\begin{lemma}\label{lem:osc of Phi}
Let $\{u_j\}_{j=1}^m$ be $C^{1,\alpha}$ ordered graphs on $B_1(0)$ which is a solution to $(\mc Q,\{g_j\})$. Suppose that \eqref{eq:Lambda and lambda}, \eqref{eq:ms AB norm}, \eqref{eq:ms AB at x00}, \eqref{eq:small C1alpha} and \eqref{eq:eta kappa0 and kappa1} are satisfied.
There exists a constant 
$C=C(n,m,\Lambda,\lambda,\alpha)$ 
so that for each $1\leq j\leq m$,
\begin{gather*}
    \sup_{x\in B_{r/2}(0)} 
        |u_j(x)-(u_j)_{B_r(0)}-(Du_j)_{B_{r}(0)}\cdot x|
    \leq C(
        \sqrt{\oli{\bm\sigma}(r)} + \kappa
        +\kappa_0 + \delta_0 + \eta
        )r^2;\\
    \sup_{x\in B_{r/4}(0)}
        |Du_j(x)-(Du_j)_{B_{r}(0)}|
    \leq C(
        \sqrt{\oli{\bm\sigma}(r)} + \kappa
        +\kappa_0 + \delta_0 + \eta
        )r.
\end{gather*}
	\end{lemma}
\begin{proof}
We divide the proof into two cases.

\medskip
{\noindent\em Case I: $r\geq 1/16^{m+1}$.} Then by \eqref{eq:small C1alpha},
\begin{gather*}
   \sup_{x\in B_{r/2}(0)} |u_j(x)-(u_j)_{B_r(0)}-(Du_j)_{B_{r}(0)}\cdot x|
    \leq C\delta_0\leq C\cdot 16^{2m+2} r^2;\\
    \sup_{x\in B_{r/4}(0)}
        |Du_j(x)-(Du_j)_{B_{r}(0)}|
    \leq C\delta_0\leq C\cdot 16^{2m+2} r^2.
\end{gather*}
These are the desired inequalities.

\medskip
{\noindent\em Case II: $r\leq 1/16^{m+1}$.}  Applying Lemma \ref{lem:large gap}, one can find $R\in [r, 16^mr]$ such that $\{u_j\}$ has the same connected components in $B_{16R}(0)$ and $B_R(0)$. Choose one connected component and still denote it by $\{u_j\}$. By Lemma \ref{lem:osc of diff} and Lemma \ref{lem:gradient diff}, there exists a constant $C=C(n,m,\Lambda,\lambda)$ so that for all $x\in B_R(0)\supset B_r(0)$,
\begin{equation}\label{eq:diff order}
    |u_i-u_j|\leq C\kappa r^2; \quad 
    |Du_i-Du_j|\leq C(\kappa+\kappa_0)r.
		\end{equation}
Let $\bm f$ and $g$ be the notions in \eqref{eq:pde of Psi}. 
Then by Lemma \ref{lem:bound bm f and g} and \eqref{eq:bound mc Q varphir},
\begin{gather}
    |\bm f(\{u_j\})|
    \leq C(m,\Lambda)(\kappa+\kappa_0)^2r^2 ;
        \label{eq:bound f by r2}\\
    \sup_{x,y\in B_r(0)}
        \frac{|\bm f(\{u_j\})(x)-\bm f(\{u_j\})(y)|}{|x-y|^\alpha}
    \leq C(m,\Lambda)(\kappa+\kappa_0)r;
    \label{eq:bound f alpha by r}\\
    \wti g:= |g(\{u_j\},\{g_j\})|+m|\mc Q\varphi_r|
    \leq C(m,\Lambda)\big(
        (\kappa+\kappa_0) r+\kappa_0+\delta_0+\eta
        \big).
    \label{eq:bound wti g by delta0}
 \end{gather}
Now we fix a constant $q>n$ (e.g. $q=n+1$).  
By a direct computation,
\[ 
    \int_{B_r(0)}\|\bm f\|^{q}\,\mr dx
    \leq C(\kappa+\kappa_0)^{2q}r^{n+2q}.
    \]
Thus by the local estimates for weak solutions \cite{GT}*{Theorem 8.17} with $u(x)=\Psi(x)$, $p=2$ and $2R=r$, then there exists $C=C(n,q)$ so that 
\begin{align*}
    \sup_{x\in B_{r/2}(0)} |\Psi(x)|
    \leq C(r^{-\frac{n}{2}} \|\Psi(x)\|_{L^2(B_r(0))}
        +r^{1-\frac{n}{q}}\|\bm f\|_{L^q(B_r(0))}
        +r^{2-\frac{2n}{q}}\|\wti g\|_{L^{q/2}(B_r(0))}) .
    \end{align*}
By Lemma \ref{lem:int osc of Dvarphi} and the estimates for $\bm f$ and $\wti g$ in \eqref{eq:bound f by r2} and \eqref{eq:bound wti g by delta0}, we hence conclude that 
\begin{equation}\label{eq:bound Psi by r2}
\begin{split}
    \sup_{x\in B_{r/2}(0)}|\Psi(x)|
    &\leq C\Big(
        r^{-\frac{n}{2}}
        \cdot \sqrt{ \oli{\bm\sigma}(r)} 
        \cdot r^{2+\frac{n}{2}} 
        +r^{1-\frac{n}{q}}(\kappa+\kappa_0)^2r^{\frac{n}{q}+2}+\\
        &\hspace{6em}+r^{2-\frac{2n}{q}}\big(
            (\kappa+\kappa_0) r +\kappa_0+\delta_0+\eta
            \big)\cdot r^{\frac{2n}{q}}
        \Big) \\
    &\leq C\Big(
        \sqrt{\oli{\bm\sigma}(r)}
        +\kappa+\kappa_0+\delta_0+\eta
        \Big)r^2. 
    \end{split}
        \end{equation}
On the other hand, by \eqref{eq:diff order},
\begin{align*}
    |\Psi_i-\Psi_j|
    &\leq |u_i-u_j|  
        - |(u_i)_{B_r(0)}-(u_j)_{B_r(0)}| 
        +|(Du_i)_{B_r(0)}-(Du_j)_{B_r(0)}|\cdot r\\
    &\leq C(\kappa+\kappa_0)r^2,
\end{align*} 
where $C=C(n,m,\Lambda,\lambda,\alpha)$. 
This together with \eqref{eq:bound Psi by r2} gives that for $x\in B_{r/2}(0)$,
\[      
    |\Psi_j|\leq |\Psi_j-\Psi|  + |\Psi|
    \leq \frac{1}{m}\sum_{i=1}^m|\Psi_j-\Psi_i| + |\Psi|
    \leq C\Big(
        \sqrt{\oli{\bm\sigma}(r)} 
        +\kappa+\kappa_0+\delta_0+\eta
        \Big)r^2. \]
This proves the first inequality in the lemma.

\medskip
It remains to estimate the derivative of $\Psi$. Indeed, by the $C^{1,\alpha}$ estimates Theorem \ref{thm:C1alpha},
\[
\begin{split}
    \sup_{x\in B_{r/4}(0)} |D\Psi| 
    \leq C(n,m,\Lambda) \Big(
        \sup_{x\in B_{r/2}(0)}r^{-1}|\Psi(x)| 
        + r^\alpha\sup_{x,x'\in B_{r/2}(0)}
            \frac{|\bm f(x)-\bm f(x')|}{|x-x'|^\alpha}+\\
        +\sup_{x\in B_{r/2}(0)} |\bm f(x)| + \sup_{x\in B_{r/2}(0)}r|\wti g|
        \Big).
\end{split}\] 
Plugging \eqref{eq:bound Psi by r2}, \eqref{eq:bound f by r2}, \eqref{eq:bound f alpha by r}, and \eqref{eq:bound wti g by delta0} into it, we conclude that 
\begin{equation}\label{eq:DPsi bound}
\begin{split}
    \sup_{x\in B_{r/4}(0)} |D\Psi| 
        &\leq C\Big(
            \sqrt{\oli{\bm\sigma}(r)}
            +\kappa+\kappa_0+\delta_0+\eta)r 
            +(\kappa+\kappa_0)r^{1+\alpha} 
            +(\kappa+\kappa_0)^2r^2
        \Big)\\
        &\leq C\Big(
            \sqrt{\oli{\bm\sigma}(r)}
            +\kappa+\kappa_0+\delta_0+\eta
        \Big)r.
\end{split}
\end{equation}
On the other hand, \eqref{eq:diff order} gives that for $x\in B_r(0)$,
\begin{equation*}
    |D\Psi_j-D\Psi|
    =|Du_j-(Du_j)_{B_r(0)}-D\varphi-(D\varphi)_{B_r(0)}|
    \leq C(\kappa+\kappa_0)r.
\end{equation*}
Combining with \eqref{eq:DPsi bound}, we conclude that for all $x\in B_{r/4}(0)$,
\[
    |D\Psi_j|
    \leq C\Big(
        \sqrt{\oli{\bm\sigma}(r)} 
        +\kappa +\kappa_0 +\delta_0+\eta
        \Big)r.
\] 
Hence the lemma is proved.
	\end{proof}

%%%%%%%%%%%%%%%%%%%%%%%%%%%%%%%%%%%%%%%%%%%%%%%%%%%%%
\subsection{Lipschitz estimates of the first derivatives}
\label{subsec:Lipschitz bound of the 1st derivative}
In this section, we will prove the Lipschitz upper bound for ordered graphs by using Lemma \ref{lem:osc of Phi}. 
The key step is to apply Proposition \ref{lem:bounded hessian average} and Proposition \ref{prop:ui is W22} to bound $\oli{\bm\sigma}(r)$ by $\bm\kappa^*$ and $\eta$.
\begin{comment}
For simplicity, denote by
\[  
    \oli\delta:=\sup_{x\in B_1(0)}(|u_1|+|u_m|).\]
If \eqref{eq:ms AB norm} is satisfied, then
\begin{equation} \label{eq:eta and oli delta}
\eta \leq C(m,\Lambda)\oli\delta.
\end{equation}
\end{comment}

\begin{theorem}\label{thm:derivative is Lip}
Let $\{u_j\}_{j=1}^m$ be $C^{1,\alpha}$ ordered graphs on $B_1(0)$ which is a solution to $(\mc Q,\{g_j\})$. 
Suppose that \eqref{eq:Lambda and lambda}, \eqref{eq:ms AB norm}, \eqref{eq:ms AB at x00}, \eqref{eq:small C1alpha} and \eqref{eq:eta kappa0 and kappa1} are satisfied.
Then there exists 
$C=C(n,m,\Lambda,\lambda,\alpha)$
so that for $j=1,\cdots, m$,
\[ 
    \sup_{x,y\in B_{1/2}(0)} \frac{|Du_j(x)-Du_j(y)|}{|x-y|} 
    \leq C(\kappa+\kappa_0+\kappa_1+\delta_0+\eta). \]
	\end{theorem}
\begin{proof}
Let 
$x_1,x_2\in B_{1/2}(0)$ and $r:=|x_1-x_2|$.
So long as $r\geq 1/32$, then by  \eqref{eq:small C1alpha},
\[    
    \frac{|Du_j(x_1)-Du_j(x_2)|}{|x_1-x_2|} 
    \leq \delta_0 \cdot |x_1-x_2|^{\alpha-1}\leq {32}\delta_0.
        \]
This is the desired inequality.
Now we assume that $r<1/32$.
By Proposition \ref{prop:ui is W22}, there exists $C=C(n,m,\Lambda,\lambda)$ so that 
\begin{equation}\label{eq:Hessian 2 bound by kappa*}
    \sum_{j=1}^m \int_{B_{3/4}(0)} |D^2 u_j|^2\,\mr dx 
    \leq C(\bm\kappa^*+\eta^2).  
    \end{equation}
For each $y\in B_{1/2}(0)$ and $t\in(0,1/2]$, we define 
\[  
    \oli{\bm\sigma}(t;y)
        :=\sum_{j=1}^m\frac{1}{|B_t|}\int_{B_t(y)}|D^2u_j|^2\,\mr dx.\]
Then \eqref{eq:Hessian 2 bound by kappa*} implies that for all $y\in B_{1/2}(0)$,
\[  
    \oli{\bm \sigma}(1/4;y)  \leq   C(\bm \kappa^*+\eta^2).\]
Applying Proposition \ref{lem:bounded hessian average} in $B_{1/4}(y)$, we obtain that for all $t\leq 1/4$, 
\begin{equation} \label{eq:sigma t less than sigma*}  
    \oli{\bm \sigma}(t;y)
         \leq C(\oli{\bm\sigma}(1/4;y)+\bm\kappa^*+\eta^2) 
         \leq C(\bm\kappa^*+\eta^2),
         \end{equation}
where $C=C(n,m,\Lambda,\lambda,\alpha)$.
Observe that $r\leq 1/32$ and $B_{4r}(x_1)\subset B_{1/4}(x_1)$.
Then by Lemma \ref{lem:osc of Phi} in $B_{1/4}(x_1)$, there exists a constant $C=C(n,m,\Lambda,\lambda,\alpha)$ so that
\begin{gather*}
    |Du_j(x_1)-(Du_j)_{B_{4r}(x_1)}|
         \leq C\big( \sqrt{\oli{\bm\sigma}(4r;x_1)}+\kappa+\kappa_0+\delta_0+\eta\big)r;\\
    |Du_j(x_2)-(Du_j)_{B_{4r}(x_1)}|
         \leq C\big( \sqrt{\oli{\bm\sigma}(4r;x_1)}+\kappa+\kappa_0+\delta_0+\eta\big)r. 
\end{gather*}
This together with \eqref{eq:sigma t less than sigma*} (by taking $y=x_1$ therein) implies 
\begin{align*}
    |Du_j(x_1)-Du_j(x_2)|
        &\leq |Du_j(x_1)-(Du_j)_{B_{4r}(x_1)}|+|Du_j(x_2)-(Du_j)_{B_{4r}(x_1)}|\\
        &\leq C\big( \sqrt{\oli{\bm\sigma}(4r;x_1)}+\kappa+\kappa_0+\delta_0+\eta\big)r\\
        &\leq C(\sqrt{\bm\kappa^*+\eta^2})\cdot |x_1-x_2|,
\end{align*}  
which yields
\[
|Du_j(x_1)-Du_j(x_2)|\leq C(\sqrt{\bm\kappa^*}+\eta)\cdot |x_1-x_2|.
\]
This completes the proof of Theorem \ref{thm:derivative is Lip}.
	\end{proof}

%%%%%%%%%%%%%%%%%%%%%%%%%%%%%%%%%%%%%%%%%%%%%
\appendix
\section{Estimates of second order partial differential equations}
We first consider the operators $L$ of the form
\[
    Lu=D_i(a^{ij}(x)D_ju+b^iu)+c^i(x)D_iu+d(x)u,
\]
where coefficients $a^{ij},b^i,c^i,d$ ($i,j=1,\cdots,n$) are assumed to be measurable functions on a domain $\Omega\subset \mb R^n$. Let $f^i,g, i=1,\cdots, n$ be locally integrable functions on $\Omega$. We shall assume that $L$ is strictly elliptic in $\Omega$; that is, there exists a
positive number $\lambda$ such that
\begin{equation}\label{eq:lambda}
    \sum_{i,j}a^{ij}\xi_i\xi_j\geq \lambda |\xi |^2,\quad \forall x\in\Omega, \xi\in\mb R^n.
\end{equation}
We also assume that the coefficients $a^{ij}, b^i\in C^{\alpha}(\oli\Omega)$, $c^i,d, g\in L^\infty(\Omega)$, and $f^i\in C^\alpha( \oli \Omega)$. Suppose
\begin{equation}\label{eq:K bound}
    \max_{1\leq i,j\leq n}\left\{  |a^{ij},b^i|_{C^\alpha(\oli\Omega)}, |c^i,d|_{L^\infty(\oli\Omega)} \right\}\leq K.
\end{equation}
We present some estimates of the solution to the equation
\begin{equation}\label{eq:Lu divergence form}
    Lu=D_i f^i+g, \quad \text{weakly in $B_r(0)$}.
\end{equation}
For the convenience, we denote $\mf f=(f^1,\cdots,f^n)$. 
\begin{theorem}[\cite{GT}*{Theorem 8.32}]\label{thm:C1alpha}
Let $u\in C^{1,\alpha}(\oli{B}_1(0))$ be a weak solution to \eqref{eq:Lu divergence form} in $B_r(0)$ for some $r\in(0,1]$. 
Suppose that $L$ satisfies the assumption \eqref{eq:lambda} and \eqref{eq:K bound} for $\Omega=B_1(0)$.
Then we have
\[
    |Du|_{C^0(B_{r/2}(0))}
    \leq C\Big(  r^{-1}|u|_{C^0(B_r(0))} + r|g|_{C^0(B_r(0))} + |f|_{C^0(B_r(0))} 
    + r^\alpha \sup_{x,y\in B_r(0)}\frac{|\mf f(x)-\mf f(y)|}{|x-y|^\alpha}  \Big)
\]
for $C = C(n, \lambda, K)$, where $\lambda$ is given by \eqref{eq:lambda}, and $K$ by \eqref{eq:K bound}.
\end{theorem}
\begin{proof}
Define 
\[
    \wti a^{ij}(x):=a^{ij}(rx), \quad 
    \wti b^i(x)=rb^i(rx);\quad 
    \wti c^i(x)=rc^i(rx), \quad 
    \wti d(x)=r^2d(rx).
\]
Then by \eqref{eq:lambda} and \eqref{eq:K bound},
\begin{gather*}
    \sum_{i,j}\wti a^{ij}(x)\xi_i\xi_j
        \geq \lambda |\xi |^2,\quad \forall x\in B_1(0), \xi\in\mb R^n;\\
    \max_{1\leq i,j\leq n}\left\{  |\wti a^{ij},\wti b^i|_{C^\alpha(\oli B_1(0))}, |\wti c^i,\wti d|_{L^\infty(\oli B_1(0))} \right\}
        \leq Kr^\alpha\leq K. 
\end{gather*}
Let 
\[
    \wti u(x):=r^{-1}\wti u(rx);\quad  
    \wti {\mf f}^i(x)=(\wti f^1,\cdots,\wti f^n):=\mf f^i(rx); \quad 
    \wti g(x):=rg(rx).
\]
Then $\wti u$ satisfies the equation
\[  
    D_i(\wti a^{ij}(x)D_j\wti u+\wti b^i\wti u) + \wti c^i(x)D_i\wti u+d(x)\wti u
        =D_i\wti f^i+\wti g.
\]
Applying \cite{GT}*{Theorem 8.32} to $\wti u$ in $B_1(0)$, we obtain that 
\begin{align*}
    |\wti u|_{C^{1,\alpha}(B_{1/2}(0))}
        &\leq C\Big( |\wti u|_{C^0(B_1(0))} + |\wti g|_{C^0(B_1(0))} 
        + |\wti {\mf f}|_{C^0(B_1(0))} 
        + \sup_{x,y\in B_1(0)} \frac{|\wti {\mf f}(x)-\wti {\mf f}(y)|}{|x-y|^\alpha}\Big)\\
        &=C\Big( r^{-1}|u|_{C^0(B_r(0))} + r|g|_{C^0(B_r(0))}
        +|\mf f|_{C^0(B_r(0))} + r^\alpha \sup_{x,y\in B_r(0)}\frac{|\mf f(x)-\mf f(y)|}{|x-y|^\alpha}\Big).
\end{align*}
for $C = C(n, \lambda, K)$. 
This finishes the proof of Theorem \ref{thm:C1alpha}.
\end{proof}

\bibliographystyle{amsalpha}
\bibliography{minmax}

\end{document}